\documentclass[a4paper,final]{article}

\usepackage[T1]{fontenc}
\usepackage[utf8]{inputenc}
\usepackage[francais]{layout}
\usepackage{csquotes}
\usepackage[french,english]{babel}

\usepackage{amsmath,amssymb,amsfonts,mathabx,stmaryrd}
\usepackage{amsthm}
\usepackage{xargs}
\usepackage{mdwlist}
\usepackage[poly, all]{xy}
\usepackage{ifthen}

\usepackage{soul}
\setuldepth{a}

\usepackage{graphicx}
\usepackage{enumitem}

\usepackage{ltugcomn}

\usepackage{ifdraft}
\usepackage{comment}

\usepackage[final]{hyperref}

\makeatletter
\hypersetup{
    unicode=true,
    pdftoolbar=true,
    pdfmenubar=true,
    pdffitwindow=false,
    pdfstartview={FitH},
    pdftitle={\@title},
    pdfauthor={\@author},
    pdfsubject={},
    pdfcreator={},
    pdfproducer={},
    pdfkeywords={},
    pdfnewwindow=true,
    colorlinks,
    linkcolor=black,
    citecolor=black,
    filecolor=black,
    urlcolor=black
}
\makeatother

\newcommand{\thmlevel}{subsection}
\usepackage{aliascnt}
\def\NewTheorem#1{%
  \newaliascnt{#1}{thm}
  \newtheorem{#1}[#1]{\csname #1Name\endcsname}
  \aliascntresetthe{#1}
  \expandafter\def\csname #1autorefname\endcsname{\csname #1Name\endcsname}
  \expandafter\def\csname #1Autorefname\endcsname{\csname #1Name\endcsname}
}

\newcommand{\theoremName}{\iflanguage{francais}{Th\'eor\`eme}{Theorem}}

\newcommand{\pbName}{\iflanguage{francais}{Probl\`eme}{Problem}}

\newcommand{\dfName}{\iflanguage{francais}{D\'efinition}{Definition}}

\ifthenelse{\equal{\thmlevel}{}}{\newtheorem{thm}{\theoremName}}{\newtheorem{thm}{\theoremName}[\thmlevel]}

\newtheorem{thmintro}{\theoremName}

\NewTheorem{lem}
\NewTheorem{sublem}
\NewTheorem{prop}
\NewTheorem{cor}
\NewTheorem{conj}
\NewTheorem{pb}

\theoremstyle{definition}
\NewTheorem{df}
\NewTheorem{exo}
\NewTheorem{ex}
\NewTheorem{rmq}

\newtheorem*{df*}{Definition}

\makeatletter
\renewenvironment{proof}[1][]{\par
  \pushQED{\qed}%
  \normalfont \topsep6\p@\@plus6\p@\relax
  \trivlist
  \item[\hskip\labelsep
        \bfseries
    \proofname\ifthenelse{\equal{#1}{}}{}{\textmd{ (#1)}}\@addpunct{.}]\ignorespaces
}{%
  \popQED\endtrivlist\@endpefalse
}
\makeatother

\makeatletter
\@addtoreset{thm}{section}
\makeatother
\let\originalleft\left
\let\originalright\right
\renewcommand{\left}{\mathopen{}\mathclose\bgroup\originalleft}
\renewcommand{\right}{\aftergroup\egroup\originalright}

\makeatletter
\def\DeclareMathBinOp{\@ifstar{\declaremathbinop@star}{\declaremathbinop@nostar}}
\def\declaremathbinop@star#1#2{\def#1{\test@subnexp@star#2}}
\def\test@subnexp@star#1{\@ifnextchar_{\isol@subnexp@star#1}{\test@exp@star#1}}
\def\test@exp@star#1{\@ifnextchar^{\isol@expnsub@star#1}{\mathbin{#1}}}
\def\isol@subnexp@star#1_#2{\@ifnextchar^{\eval@subnexp@star#1_#2}{\mathbin{\operatorname*{#1}_{#2}}}}
\def\eval@subnexp@star#1_#2^#3{\mathbin{\operatorname*{#1}_{#2}^{#3}}}
\def\isol@expnsub@star#1^#2{\@ifnextchar_{\eval@expnsub@star#1^#2}{\mathbin{\operatorname*{#1}^{#2}}}}
\def\eval@expnsub@star#1^#2_#3{\mathbin{\operatorname*{#1}_{#3}^{#2}}}
\def\declaremathbinop@nostar#1#2{\def#1{\test@subnexp@nostar#2}}
\def\test@subnexp@nostar#1{\@ifnextchar_{\isol@subnexp@nostar#1}{\test@exp@nostar#1}}
\def\test@exp@nostar#1{\@ifnextchar^{\isol@expnsub@nostar#1}{\mathbin{#1}}}
\def\isol@subnexp@nostar#1_#2{\@ifnextchar^{\eval@subnexp@nostar#1_#2}{\mathbin{\underset{#2}{#1}}}}
\def\eval@subnexp@nostar#1_#2^#3{\mathbin{\overset{#3}{\underset{#2}{#1}}}}
\def\isol@expnsub@nostar#1^#2{\@ifnextchar_{\eval@expnsub@nostar#1^#2}{\mathbin{\overset{#2}{#1}}}}
\def\eval@expnsub@nostar#1^#2_#3{\mathbin{\overset{#2}{\underset{#3}{#1}}}}
\makeatother

\let\OLDtimes\times
\DeclareMathBinOp*{\times}{\OLDtimes}
\let\OLDamalg\amalg
\DeclareMathBinOp*{\amalg}{\OLDamalg}
\let\OLDotimes\otimes
\DeclareMathBinOp*{\otimes}{\OLDotimes}
\let\OLDwedge\wedge
\DeclareMathBinOp*{\wedge}{\OLDwedge}

\makeatletter
\def\DeclareArrow#1#2{\def#1{\test@subnexp#2}}
\def\test@subnexp#1{\@ifnextchar_{\isol@subnexp#1}{\test@exp#1}}
\def\test@exp#1{\@ifnextchar^{\isol@expnsub#1}{#1}}
\def\isol@subnexp#1_#2{\@ifnextchar^{\eval@subnexp#1_#2}{\underset{#2}{#1}}}
\def\eval@subnexp#1_#2^#3{\underset{#2}{\overset{#3}{#1}}}
\def\isol@expnsub#1^#2{\@ifnextchar_{\eval@expnsub#1^#2}{\overset{#2}{#1}}}
\def\eval@expnsub#1^#2_#3{\overset{#2}{\underset{#3}{#1}}}
\makeatother

\let\OLDto\to
\DeclareArrow{\to}{\OLDto}
\DeclareArrow{\from}{\leftarrow}
\DeclareArrow{\lra}{\longrightarrow}
\DeclareArrow{\lla}{\longleftarrow}
\newcommand{\Indext}{\operatorname{\underline{\mathbf{Ind}}}^\mathbb U}
\newcommand{\bubblespace}{\operatorname{\mathfrak{B}}}
\newcommand{\kaploop}{\mathcal L}
\newcommand{\A}{\mathbb{A}}
\newcommand{\Lcot}{\mathbb{L}}
\newcommand{\T}{\mathbb{T}}
\newcommand{\ev}{\operatorname{ev}}
\newcommand{\pr}{\operatorname{pr}}
\newcommand{\op}{^{\mathrm{op}}}
\newcommand{\loccit}{\emph{loc. cit.} }
\newcommand{\Cc}{\mathcal{C}}
\newcommand{\Dd}{\mathcal{D}}
\newcommand{\Oo}{\mathcal{O}}
\newcommand{\sSets}{\mathbf{sSets}}
\newcommand{\presh}{\operatorname{\mathcal{P}}}
\newcommand{\dSt}{\mathbf{dSt}}
\newcommand{\Map}{\operatorname{Map}}
\newcommand{\Fct}{\operatorname{Fct}}
\newcommand{\Ind}{\operatorname{\mathbf{Ind}}}
\newcommand{\Pro}{\operatorname{\mathbf{Pro}}}
\newcommand{\Tate}{\operatorname{\mathbf{Tate}}}
\newcommand{\Indu}[1]{\Ind^{\mathbb{#1}}}
\newcommand{\Prou}[1]{\Pro^{\mathbb{#1}}}
\newcommand{\inftyCatu}[1]{\inftyCat^{\mathbb{#1}}}
\newcommand{\PresLeftu}[1]{\mathbf{Pr}^{\mathrm{L,}\mathbb{#1}}_\infty}
\newcommand{\homol}{\mathrm H}
\newcommand{\Spec}{\operatorname{Spec}}
\newcommand{\Gm}{\mathbb{G}_m}
\newcommandx*{\el}[4][2=\eldebutpardefaut,4={,}]{#1_{#2}#4\dots#4#1_{#3}}
\newcommand{\quot}[2]{\ensuremath \mathchoice {\displaystyle #1 \raisebox{-2pt}{$\displaystyle \hspace{-1pt}{/} $} \raisebox{-4pt}{$\displaystyle \hspace{-1pt}{#2}$}}{\textstyle #1 \raisebox{-1pt}{$\textstyle \hspace{-1pt}{/} $} \raisebox{-2pt}{$\textstyle \hspace{-1pt}{#2}$}}{\scriptstyle #1 \raisebox{-1pt}{$\scriptstyle \hspace{-1pt}{/} $} \raisebox{-2pt}{$\scriptstyle \hspace{-1pt}{#2}$}}{\scriptscriptstyle #1 \raisebox{-1pt}{$\scriptscriptstyle \hspace{-1pt}{/} $} \raisebox{-2pt}{$\scriptscriptstyle \hspace{-1pt}{#2}$}}}
\newcommand{\mymatrix}{\shorthandoff{;:!?} \xymatrix}
\newcommand{\N}{\mathbb{N}}
\newcommand{\R}{\mathbb{R}}
\newcommand{\pt}{{*}}
\newcommand{\id}{\operatorname{id}}
\newcommand{\dual}[1]{{#1}^{\vee}}
\newcommand{\inftyCat}{\mathbf{Cat}_\infty}
\newcommand{\PresLeft}{\mathbf{Pr}^{\mathrm{L}}_\infty}
\newcommand{\Qcoh}{\mathbf{Qcoh}}
\newcommand{\dAff}{\mathbf{dAff}}
\newcommand{\cdgaunbounded}{\mathbf{cdga}}
\newcommand{\cdga}{\cdgaunbounded^{\leq 0}}
\newcommand{\dgMod}{\mathbf{dgMod}}
\newcommand{\Perf}{\mathbf{Perf}}
\newcommand{\Mapstack}{\operatorname{\underline{Ma}p}}
\newcommand{\Homint}{\operatorname{\underline{Hom}}}
\newcommand{\RHomint}{\operatorname{\mathbb{R}\underline{Hom}}}
\DeclareMathOperator*{\colim}{colim}
\newcommand{\B}{\mathrm{B}}
\newcommand{\K}{\operatorname{K}}
\newcommand{\ptfin}{\mathrm{Fin}^\pt}
\newcommand{\Tateu}[1]{\Tate^{\mathbb{#1}}}
\newcommand{\PresRightu}[1]{\mathbf{Pr}^{\mathrm{R,}\mathbb{#1}}_\infty}
\newcommand{\monoidalinftyCatu}[1]{\inftyCat^{\otimes, \mathbb{#1}}}
\newcommand{\dStArt}{\dSt^{\mathrm{Art}}}
\newcommand{\dStArtlfp}{\dSt^{\mathrm{Art,lfp}}}
\newcommand{\Sym}{\operatorname{Sym}}
\newcommand{\Der}{\operatorname{Der}}
\newcommand{\eldebutpardefaut}{1}
\newcommandx*{\iel}[5][2=i,3=\eldebutpardefaut,5={,}]{#1_{#2_{#3}}#5\dots#5#1_{#2_{#4}}}
\newcommand{\app}[4]{\begin{array}{c@{\hskip 2pt}c@{\hskip 2pt}c} #1 & \to & #2 \\ #3 & \mapsto & #4 \end{array}}
\newcommand{\comma}[2]{\ensuremath \mathchoice {\raisebox{4pt}{$\displaystyle #1 $} \raisebox{2pt}{$\displaystyle / $} \displaystyle \hspace{-1pt}{#2}}{\raisebox{2pt}{$\textstyle #1 $} \raisebox{1pt}{$\textstyle / $} \textstyle \hspace{-1pt}{#2}}{\raisebox{2pt}{$\scriptstyle #1 $} \raisebox{1pt}{$\scriptstyle / $} \scriptstyle \hspace{-1pt}{#2}}{\raisebox{2pt}{$\scriptscriptstyle #1 $} \raisebox{1pt}{$\scriptscriptstyle / $} \scriptscriptstyle \hspace{-1pt}{#2}}}
\newcommandx*{\dcell}[6][1,2,3,4,5={=>},6={1pc},usedefault]{\ar@/^#6/[#1]^{#2}_{}="UP" \ar@/_#6/[#1]_{#3}^{}="DOWN" \ar @{#5} "UP";"DOWN" ^{#4} }
\newcommandx*{\cart}[3][1=1,2=5,3=10,usedefault]{\ar@{-}[]+D+<#3pt,#1pt>+<#2pt,0pt>;[]+D+<#3pt,-#3pt>+<#2pt,#1pt> \ar@{-}[]+D+<0pt,-#3pt>+<#2pt,#1pt>;[]+D+<#3pt,-#3pt>+<#2pt,#1pt>}
\newcommandx*{\cocart}[3][1=-1,2=8,3=10,usedefault]{\ar@{-}[]+U+<-#2pt,-#1pt>;[]+U+<-#2pt,-#1pt>+<0pt,#3pt> \ar@{-}[]+U+<-#2pt,-#1pt>;[]+U+<-#2pt,-#1pt>+<-#3pt,0pt>}
\newcommandx*{\timesunder}[5][1={},2={},3=-2pt,4=0pt,5=0mm,usedefault]{\times_{\makebox[#5]{\raisebox{#3}{\ensuremath{\scriptstyle #1}}}}^{\makebox[#5]{\raisebox{#4}{\ensuremath{\scriptstyle #2}}}}}
\newcommand{\IPP}{\operatorname{\mathbf{IPP}}}
\newcommand{\IP}{\mathbf{IP}}
\newcommand{\PI}{\mathbf{PI}}
\newcommand{\PIQ}{\mathbf{PIQ}}
\newcommand{\PIPerf}{\operatorname{\mathbf{PIPerf}}}
\newcommand{\IPPerf}{\operatorname{\mathbf{IPPerf}}}
\newcommand{\IPerf}{\operatorname{\mathbf{IPerf}}}
\newcommand{\PPerf}{\operatorname{\mathbf{PPerf}}}
\newcommand{\PIQcoh}{\operatorname{\mathbf{PIQcoh}}}
\newcommand{\cotangent}{\lambda}
\newcommand{\Proext}{\operatorname{\underline{\mathbf{Pro}}}^\mathbb U}
\newcommand{\Indextu}[1]{\operatorname{\underline{\mathbf{Ind}}}^\mathbb{#1}}
\newcommand{\Tateextu}[1]{\operatorname{\underline{\mathbf{Tate}}}^\mathbb{#1}}
\newcommand{\Proextu}[1]{\operatorname{\underline{\mathbf{Pro}}}^\mathbb{#1}}
\newcommand{\overcat}{\operatorname{O}}  
\newcommand{\btw}{\operatorname{\mathrm{B}}}
\newcommand{\bubblestack}{\operatorname{\underline{\mathfrak{B}}}}
\newcommand{\formalsphere}{{\hat{\mathrm{S}}}}
\newcommand{\kaplooppre}{\tilde{\mathcal L}}
\newcommand{\Tatestack}{\dSt^{\mathrm{Tate}}}
\newcommand{\IQcoh}{\operatorname{\mathbf{IQcoh}}}
\newcommand{\shybounded}{\mathbf{IP} \dSt^{\mathrm{shy,b}}}
\newcommand{\closedforms}[1]{\mathbf A^{#1\mathrm{,cl}}}
\newcommand{\forms}[1]{\mathbf A^{#1}}
\newcommand{\IPclosedforms}[1]{\closedforms{#1}_{\IP}}
\newcommand{\IPforms}[1]{\forms{#1}_{\IP}}

\usepackage{geometry}
\setlist[enumerate]{label=\emph{(\roman*)},ref=\emph{(\roman*)}}

\entrymodifiers={!!<0pt,0.7ex>+}

\usepackage[notref,notcite]{showkeys}

\title{Higher dimensional formal loop spaces}
\author{Benjamin Hennion}
\date{2015}

\newlist{assertions}{enumerate}{1}
\setlist[assertions]{label={(\alph*)}, ref={assertion (\alph*)}}
\newlist{disjunction}{enumerate}{1}
\setlist[disjunction]{label={(\arabic*)}, ref={case (\arabic*)}}

\begin{document}

\selectlanguage{english}
\maketitle

\begin{abstract}
If $M$ is a symplectic manifold then the space of smooth loops $\mathrm C^{\infty}(\mathrm S^1,M)$ inherits of a quasi-symplectic form. We will focus in this article on an algebraic analogue of that result.
In their article \cite{kapranovvasserot:loop1}, Kapranov and Vasserot introduced and studied the formal loop space of a scheme $X$.

We generalize their construction to higher dimensional loops. To any scheme $X$ -- not necessarily smooth -- we associate $\kaploop^d(X)$, the space of loops of dimension $d$. We prove it has a structure of (derived) Tate scheme -- ie its tangent is a Tate module: it is infinite dimensional but behaves nicely enough regarding duality.
We also define the bubble space $\bubblespace^d(X)$, a variation of the loop space.
We prove that $\bubblespace^d(X)$ is endowed with a natural symplectic form as soon as $X$ has one (in the sense of \cite{ptvv:dersymp}).

Throughout this paper, we will use the tools of $(\infty,1)$-categories and symplectic derived algebraic geometry.
\end{abstract}

\selectlanguage{french}
\begin{abstract}
{\bf Espaces des lacets formels de dimension supérieure : }L'espace des lacets $\mathrm C^{\infty}(\mathrm S^1,M)$ associé à une variété symplectique $M$ se voit doté d'une structure (quasi-)symplectique induite par celle de $M$.
Nous traiterons dans cet article d'un analogue algébrique de cet énoncé.
Dans leur article \cite{kapranovvasserot:loop1}, Kapranov et Vasserot ont introduit l'espace des lacets formels associé à un schéma.

Nous généralisons leur construction à des lacets de dimension supérieure. Nous associons à tout schéma $X$ -- pas forcément lisse -- l'espace $\kaploop^d(X)$ de ses lacets formels de dimension $d$.
Nous démontrerons que ce dernier admet une structure de schéma (dérivé) de Tate : son espace tangent est de Tate : de dimension infinie mais suffisamment structuré pour se soumettre à la dualité.
Nous définirons également l'espace $\bubblespace^d(X)$ des bulles de $X$, une variante de l'espace des lacets, et nous montrerons que le cas échéant, il hérite de la structure symplectique de $X$.
\end{abstract}
\selectlanguage{english}

\tableofcontents

\section*{Introduction}
\addcontentsline{toc}{section}{Introduction}%

Considering a differential manifold $M$, one can build the space of smooth loops $\operatorname{L}(M)$ in $M$. It is a central object of string theory. Moreover, if $M$ is symplectic then so is $\operatorname{L}(M)$ -- more precisely quasi-symplectic since it is not of finite dimension -- see for instance \cite{munozpresas:symp}.
We will be interested here in an algebraic analogue of that result.

The first question is then the following: what is an algebraic analogue of the space of smooth loops? 
An answer appeared in 1994 in Carlos Contou-Carrère's work (see \cite{contoucarrere:jacobienne}). He studies there $\Gm(\mathbb{C}(\!(t)\!))$, some sort of holomorphic functions in the multiplicative group scheme, and defines the famous Contou-Carrère symbol.
This is the first occurrence of a \emph{formal loop space} known by the author.
This idea was then generalised to algebraic groups as the affine Grassmannian $\mathfrak{Gr}_G = \quot{G(\mathbb C (\!(t)\!))}{G(\mathbb C[\![t]\!])}$ showed up and got involved in the geometric Langlands program.
In their paper \cite{kapranovvasserot:loop1}, Mikhail Kapranov and Éric Vasserot introduced and studied the formal loop space of a smooth scheme $X$. It is an ind-scheme $\kaploop(X)$ which we can think of as the space of maps $\Spec \mathbb C(\!(t)\!) \to X$. This construction strongly inspired the one presented in this article.

There are at least two ways to build higher dimensional formal loops. The most studied one consists in using higher dimensional local fields $k(\!( t_1 )\!) \dots (\!( t_d )\!)$ and is linked to Beilinson's adèles. There is also a generalisation of Contou-Carrère symbol in higher dimensions using those higher dimensional local fields -- see \cite{osipovzhu:contoucarrere} and \cite{bgw:contoucarrere}.
If we had adopted this angle, we would have considered maps from some torus\footnote{The variable $\el{t}{d}$ are actually ordered. The author likes to think of $\Spec(k(\!( t_1 )\!) \dots (\!( t_d )\!))$ as a formal torus equipped with a flag representing this order.} $\Spec(k(\!( t_1 )\!) \dots (\!( t_d )\!))$ to $X$.

The approach we will follow in this work is different.
We generalize here the definition of Kapranov and Vasserot to higher dimensional loops in the following way.
For $X$ a scheme of finite presentation, not necessarily smooth, we define $\kaploop^d(X)$, the space of formal loops of dimension $d$ in $X$.
We define $\kaploop^d_V(X)$ the space of maps from the formal neighbourhood of $0$ in $\A^d$ to $X$. This is a higher dimensional version of the space of germs of arcs as studied by Jan Denef and François Loeser in \cite{denefloeser:germs}.
Let also $\kaploop_U^d(X)$ denote the space of maps from a \emph{punctured} formal neighbourhood of $0$ in $\A^d$ to $X$. The formal loop space $\kaploop^d(X)$ is the formal completion of $\kaploop_V^d(X)$ in $\kaploop_U^d(X)$.
Understanding those three items is the main goal of this work.
The problem is mainly to give a meaningful definition of the punctured formal neighbourhood of dimension $d$. We can describe what its cohomology should be: 
\[
\homol^n(\hat \A^d\smallsetminus \{0\}) = \begin{cases}
k[\![\el{X}{d}]\!] & \text{ if } n = 0 \\ (\el{X}{d}[])^{-1} k[\el{X^{-1}}{d}] & \text{ if } n=d-1 \\ 0 & \text{ otherwise}
\end{cases}
\]
but defining this punctured formal neighbourhood with all its structure is actually not an easy task. Nevertheless, we can describe what maps out of it are, hence the definition of $\kaploop^d_U(X)$ and the formal loop space.
This geometric object is of infinite dimension, and part of this study is aimed at identifying some structure. Here comes the first result in that direction.
\begin{thmintro}[see \autoref{L-indpro}]\label{intro-kaploop}
The formal loop space of dimension $d$ in a scheme $X$ is represented by a derived ind-pro-scheme.
Moreover, the functor $X \mapsto \kaploop^d(X)$ satisfies the étale descent condition.
\end{thmintro}
We use here methods from derived algebraic geometry as developed by Bertrand Toën and Gabriele Vezzosi in \cite{toen:hagii}. The author would like to emphasize here that the derived structure is necessary since, when $X$ is a scheme, the underlying schemes of $\kaploop^d(X)$, $\kaploop^d_U(X)$ and $\kaploop^d_V(X)$ are isomorphic as soon as $d \geq 2$.
Let us also note that derived algebraic geometry allowed us to define $\kaploop^d(X)$ for more general $X$'s, namely any derived stack.
In this case, the formal loop space $\kaploop^d(X)$ is no longer a derived ind-pro-scheme but an ind-pro-stack.

The case $d = 1$ and $X$ is a smooth scheme gives a derived enhancement of Kapranov and Vasserot's definition. This derived enhancement is conjectured to be trivial when $X$ is a smooth affine scheme in  \cite[9.2.10]{gaitsgoryrozenblyum:dgindschemes}. Gaitsgory and Rozenblyum also prove in \loccit their conjecture holds when $X$ is an algebraic group.

The proof of \autoref{intro-kaploop} is based on an important lemma. We identify a full sub-category $\Cc$ of the category of ind-pro-stacks such that the realisation functor $\Cc \to \dSt_k$ is fully faithful. We then prove that whenever $X$ is a derived affine scheme, the stack $\kaploop^d(X)$ is in the essential image of $\Cc$ and is thus endowed with an \emph{essentially unique} ind-pro-structure satisfying some properties. 
The generalisation to any $X$ is made using a descent argument.
Note that for general $X$'s, the ind-pro-structure is not known to satisfy nice properties one could want to have, for instance on the transition maps of the diagrams.

We then focus on the following problem: can we build a symplectic form on $\kaploop^d(X)$ when $X$ is symplectic?
Again, this question requires the tools of derived algebraic geometry and \emph{shifted symplectic structures} as in \cite{ptvv:dersymp}.
A key feature of derived algebraic geometry is the cotangent complex $\Lcot_X$ of any geometric object $X$. A ($n$-shifted) symplectic structure on $X$ is a closed $2$-form $\Oo_X[-n] \to \Lcot_X \wedge \Lcot_X$ which is non degenerate -- ie induces an equivalence
\[
\T_X \to \Lcot_X[n]
\]
Because $\kaploop^d(X)$ is not finite, linking its cotangent to its dual -- through an alleged symplectic form -- requires to identify once more some structure. We already know that it is an ind-pro-scheme but the proper context seems to be what we call Tate stacks.

Before saying what a Tate stack is, let us talk about Tate modules. They define a convenient context for infinite dimensional vector spaces. They where studied by Lefschetz, Beilinson and Drinfeld, among others, and more recently by Bräunling, Gröchenig and Wolfson \cite{bgw:tate}.
We will use here the notion of Tate objects in the context of stable $(\infty,1)$-categories as developed in  \cite{hennion:tate}. If $\Cc$ is a stable $(\infty,1)$-category -- playing the role of the category of finite dimensional vector spaces, the category $\Tate(\Cc)$ is the full subcategory of the $(\infty,1)$-category of pro-ind-objects $\Pro \Ind(\Cc)$ in $\Cc$ containing both $\Ind(\Cc)$ and $\Pro(\Cc)$ and stable by extensions and retracts.

We will define the derived category of Tate modules on a scheme -- and more generally on a derived ind-pro-stack.
An Artin ind-pro-stack $X$ -- meaning an ind-pro-object in derived Artin stacks -- is then gifted with a cotangent complex $\Lcot_X$. This cotangent complex inherits a natural structure of pro-ind-module on $X$. This allows us to define a Tate stack as an Artin ind-pro-stack whose cotangent complex is a Tate module.
The formal loop space $\kaploop^d(X)$ is then a Tate stack as soon as $X$ is a finitely presented derived affine scheme. For a more general $X$, what precedes makes $\kaploop^d(X)$ some kind of \emph{locally} Tate stack. This structure suffices to define a determinantal anomaly
\[
\left[\mathrm{Det}_{\kaploop^d(X)}\right] \in \homol^2\left(\kaploop^d(X), \Oo_{\kaploop^d(X)}^{\times}\right)
\]
for any quasi-compact quasi-separated (derived) scheme $X$ -- this construction also works for slightly more general $X$'s, namely Deligne-Mumford stacks with algebraisable diagonal, see \autoref{determinantalanomaly}.
Kapranov and Vasserot proved in \cite{kapranovvasserot:loop4} that in dimension $1$, the determinantal anomaly governs the existence of sheaves of chiral differential operators on $X$.
One could expect to have a similar result in higher dimensions, with higher dimensional analogues of chiral operators and vertex algebras. The author plans on studying this in a future work.

Another feature of Tate modules is duality. It makes perfect sense and behaves properly. Using the theory of symplectic derived stacks developed by Pantev, Toën, Vaquié and Vezzosi in \cite{ptvv:dersymp}, we are then able to build a notion of symplectic Tate stack: a Tate stack $Z$ equipped with a ($n$-shifted) closed $2$-form which induces an equivalence
\[
\T_Z \to^\sim \Lcot_Z[n]
\]
of Tate modules over $Z$ between the tangent and (shifted) cotangent complexes of $Z$.

To make a step toward proving that $\kaploop^d(X)$ is a symplectic Tate stack, we actually study the bubble space $\bubblespace^d(X)$ -- see \autoref{dfbubble}. When $X$ is affine, we get an equivalence
\[
\bubblespace^d(X) \simeq \kaploop_V^d(X) \times_{\kaploop_U^d(X)} \kaploop^d_V(X)
\]
Note that the fibre product above is a \emph{derived} intersection. We then prove the following result
\begin{thmintro}[see \autoref{B-symplectic}]\label{intro-bubble}
If $X$ is an $n$-shifted symplectic stack then the bubble space $\bubblespace^d(X)$ is endowed with a structure of $(n-d)$-shifted symplectic Tate stack.
\end{thmintro}

The proof of this result is based on a classical method. The bubble space is in fact, as an ind-pro-stack, the mapping stack from what we call the formal sphere $\hat S^d$ of dimension $d$ to $X$.
There are therefore two maps
\[
\mymatrix{
\bubblespace^d(X) & \bubblespace^d(X) \times \hat S^d \ar[r]^-{\ev} \ar[l]_-{\pr} & X
}
\]
The symplectic form on $\bubblespace^d(X)$ is then $\int_{\hat S^d} \ev^* \omega_X$, where $\omega_X$ is the symplectic form on $X$.
The key argument is the construction of this integration on the formal sphere, ie on an oriented pro-ind-stack of dimension $d$.
The orientation is given by a residue map. On the level of cohomology, it is the morphism 
\[
\homol^d(\hat S^d) \simeq (\el{X}{d}[])^{-1} k[\el{X^{-1}}{d}] \to k
\]
mapping $(\el{X}{d}[])^{-1}$ to $1$.

This integration method would not work on $\kaploop^d(X)$, since the punctured formal neighbourhood does not have as much structure as the formal sphere: it is not known to be a pro-ind-scheme. Nevertheless, \autoref{intro-bubble} is a first step toward proving that $\kaploop^d(X)$ is symplectic. We can consider the nerve $Z_\bullet$ of the map $\kaploop^d_V(X) \to \kaploop^d_U(X)$. It is a groupoid object in ind-pro-stacks whose space of maps is $\bubblespace^d(X)$. The author expects that this groupoid is compatible in some sense with the symplectic structure so that $\kaploop^d_U(X)$ would inherit a symplectic form from realising this groupoid.
One the other hand, if $\kaploop^d_U(X)$ was proven to be symplectic, then the fibre product defining $\bubblespace^d(X)$ should be a Lagrangian intersection. The bubble space would then inherit a symplectic structure from that on $\kaploop^d(X)$.

\subsubsection*{Techniques and conventions}
Throughout this work, we will use the techniques of $(\infty,1)$-category theory. We will once in a while use explicitly the model of quasi-categories developed by Joyal and Lurie (see \cite{lurie:htt}). That being said, the results should be true with any equivalent model.
Let us fix now two universes $\mathbb U \in \mathbb V$ to deal with size issues. Every algebra, module or so will implicitly be $\mathbb U$-small.
The first part will consist of reminders about $(\infty,1)$-categories. We will fix there some notations.
Note that we will often refer to \cite{hennion:these} for some $(\infty,1)$-categorical results needed in this article.

We will also use derived algebraic geometry, as introduced in \cite{toen:hagii}. We refer to \cite{toen:dagems} for a recent survey of this theory.
We will denote by $k$ a base field and by $\dSt_k$ the $(\infty,1)$-category of ($\mathbb U$-small) derived stacks over $k$.
In the first section, we will dedicate a few page to introduce derived algebraic geometry.

\subsubsection*{Outline}
This article begins with a few paragraphs, recalling some notions we will use. Among them are $(\infty,1)$-categories and derived algebraic geometry.
In \autoref{chapterIP}, we set up a theory of geometric ind-pro-stacks.
We then define in \autoref{chapterSymptate} symplectic Tate stacks and give a few properties, including the construction of the determinantal anomaly (see \autoref{determinantalanomaly}).
Comes \autoref{chapterfloops} where we finally define higher dimensional loop spaces and prove \autoref{intro-kaploop} (see \autoref{L-indpro}). We finally introduce the bubble space and prove \autoref{intro-bubble} (see \autoref{B-symplectic}) in \autoref{chapterBubbles}.

\subsubsection*{Aknowledgements}
I would like to thank Bertrand Toën, Damien Calaque and Marco Robalo for the many discussions we had about the content of this work.
I am grateful to Mikhail Kapranov, James Wallbridge and Giovanni Faonte for inviting me at the IPMU. My stay there was very fruitful and the discussions we had were very interesting.
I learned after writing down this article that Kapranov had an unpublished document in which higher dimensional formal loops are studied. I am very grateful to Kapranov for letting me read those notes, both inspired and inspiring.

This work is extracted from my PhD thesis \cite{hennion:these} under the advisement of Bertrand Toën. I am very grateful to him for those amazing few years.
\section*{Preliminaries}%
\addcontentsline{toc}{section}{Preliminaries}%
In this part, we recall some results and definitions from $(\infty,1)$-category theory and derived algebraic geometry.

\subsection{A few tools from higher category theory}

In the last decades, theory of $(\infty,1)$-categories has tremendously grown.
The core idea is to consider categories enriched over spaces, so that every object or morphism is considered up to higher homotopy.
The typical example of such a category is the category of topological spaces itself: for any topological spaces $X$ and $Y$, the set of maps $X \to Y$ inherits a topology.
It is often useful to talk about topological spaces up to homotopy equivalences. Doing so, one must also consider maps up to homotopy.
To do so, one can of course formally invert every homotopy equivalence and get a set of morphisms $[X,Y]$. This process loses information and mathematicians tried to keep trace of the space of morphisms.

The first fully equipped theory handy enough to work with such examples, called model categories, was introduced by Quillen.
A model category is a category with three collections of maps -- weak equivalences (typically homotopy equivalences), fibrations and cofibrations -- satisfying a bunch of conditions.
The datum of such collections allows us to compute limits and colimits up to homotopy. We refer to \cite{hovey:modcats} for a comprehensive review of the subject.

Using model categories, several mathematicians developed theories of $(\infty,1)$-categories. Let us name here Joyal's quasi-categories, complete Segal spaces or simplicial categories.
Each one of those theories is actually a model category and they are all equivalent one to another -- see \cite{bergner:oneinfty} for a review.

In \cite{lurie:htt}, Lurie developed the theory of quasi-categories. In this book, he builds everything necessary so that we can think of $(\infty,1)$-categories as we do usual categories. To prove something in this context still requires extra care though. We will use throughout this work the language as developed by Lurie, but we will try to keep in mind the $1$-categorical intuition.

In this section, we will fix a few notations and recall some results to which we will often refer.

\paragraph*{Notations:}
Let us first fix a few notations, borrowed from \cite{lurie:htt}.
\begin{itemize}
\item We will denote by $\inftyCatu U$ the $(\infty,1)$-category of $\mathbb U$-small $(\infty,1)$-categories -- see \cite[3.0.0.1]{lurie:htt};
\item Let $\PresLeftu U$ denote the $(\infty,1)$-category of $\mathbb U$-presentable (and thus $\mathbb V$-small) $(\infty,1)$-categories with left adjoint functors -- see \cite[5.5.3.1]{lurie:htt};
\item The symbol $\sSets$ will denote the $(\infty,1)$-category of $\mathbb U$-small simplicial sets up to homotopy equivalences (this is equivalent to the category of (nice) topological spaces up to homotopy);
\item For any $(\infty,1)$-categories $\Cc$ and $\Dd$ we will write $\Fct(\Cc,\Dd)$ for the $(\infty,1)$-category of functors from $\Cc$ to $\Dd$ (see \cite[1.2.7.3]{lurie:htt}). The category of presheaves will be denoted $\presh(\Cc) = \Fct(\Cc\op, \sSets)$.
\item For any $(\infty,1)$-category $\Cc$ and any objects $c$ and $d$ in $\Cc$, we will denote by $\Map_{\Cc}(c,d)$ the space of maps from $c$ to $d$.
\item For any simplicial set $K$, we will denote by $K^{\triangleright}$ the simplicial set obtained from $K$ by formally adding a final object. This final object will be called the cone point of $K^\triangleright$.
\end{itemize}
The following theorem is a concatenation of results from Lurie.

\begin{thm}[Lurie]\label{indu-thm}
Let $\Cc$ be a $\mathbb V$-small $(\infty,1)$-category.
There is an $(\infty,1)$-category $\Indu U(\Cc)$ and a functor $j \colon \Cc \to \Indu U(\Cc)$ such that
\begin{enumerate}
\item The $(\infty,1)$-category $\Indu U(\Cc)$ is $\mathbb V$-small;
\item The $(\infty,1)$-category $\Indu U(\Cc)$ admits $\mathbb U$-small filtered colimits and is generated by $\mathbb U$-small filtered colimits of objects in $j(\Cc)$;
\item The functor $j$ is fully faithful and preserves finite limits and finite colimits which exist in $\Cc$;
\item For any $c \in \Cc$, its image $j(c)$ is $\mathbb U$-small compact in $\Indu U (\Cc)$;
\item For every $(\infty,1)$-category $\Dd$ with every $\mathbb U$-small filtered colimits, the functor $j$ induces an equivalence
\[
 \Fct^{\mathbb U\mathrm{-c}}(\Indu U(\Cc), \Dd) \to^\sim \Fct(\Cc,\Dd)
\]
where $\Fct^{\mathbb U \mathrm{-c}}(\Indu U(\Cc), \Dd)$ denote the full subcategory of $\Fct(\Indu U(\Cc),\Dd)$ spanned by functors preserving $\mathbb U$-small filtered colimits.
\item If $\Cc$ is $\mathbb U$-small and admits all finite colimits then $\Indu U(\Cc)$ is $\mathbb U$-presentable;
\item If $\Cc$ is endowed with a symmetric monoidal structure then there exists such a structure on $\Indu U(\Cc)$ such that the monoidal product preserves $\mathbb U$-small filtered colimits in each variable.\label{indmonoidal}
\end{enumerate}
\end{thm}

\begin{proof}
Let us use the notations of \cite[5.3.6.2]{lurie:htt}. Let $\mathcal K$ denote the collection of $\mathbb U$-small filtered simplicial sets. We then set $\Indu U (\Cc) = \presh^\mathcal K _\emptyset(\Cc)$.
It satisfies the required properties because of \loccit 5.3.6.2 and 5.5.1.1. We also need tiny modifications of the proofs of \loccit 5.3.5.14 and 5.3.5.5. The last item is proved in \cite[6.3.1.10]{lurie:halg}.
\end{proof}

\begin{rmq}
Note that when $\Cc$ admits finite colimits then the category $\Indu U(\Cc)$ embeds in the $\mathbb V$-presentable category $\Indu V(\Cc)$.
\end{rmq}

\begin{df}
Let $\Cc$ be a $\mathbb V$-small $\infty$-category. We define $\Prou U(\Cc)$ as the $(\infty,1)$-category
\[
\Prou U(\Cc) = \left( \Indu U(\Cc\op) \right) \op
\]
It satisfies properties dual to those of $\Indu U(\Cc)$.
\end{df}

\begin{df} \label{indext}
Let $\Cc$ be a $\mathbb V$-small $(\infty,1)$-category.
Let 
\[
i \colon \Fct(\Cc, \inftyCatu V) \to \Fct(\Indu U(\Cc), \inftyCatu V)
\]
denote the left Kan extension functor.
We will denote by $\Indext_\Cc$ the composite functor
\[
\mymatrix{
\Fct(\Cc, \inftyCatu V) \ar[r]^-i & \Fct(\Indu U(\Cc), \inftyCatu V) \ar[rr]^-{\Indu U \circ -} && \Fct(\Indu U(\Cc), \inftyCatu V)
}
\]
We will denote by $\Proext_\Cc$ the composite functor
\[
\mymatrix{
\Fct(\Cc, \inftyCatu V) \ar[rr]^{\Prou U \circ -} && \Fct(\Cc, \inftyCatu V) \ar[r] & \Fct(\Prou U(\Cc), \inftyCatu V)
}
\]
We define the same way
\begin{align*}
&\Indextu V_C \colon \Fct(\Cc, \inftyCatu V) \to \Fct(\Indu V(\Cc), \inftyCat) \\
&\Proextu V_C \colon \Fct(\Cc, \inftyCatu V) \to \Fct(\Prou V(\Cc), \inftyCat) \\
\end{align*}
\end{df}

\begin{rmq}
The \autoref{indext} can be expanded as follows. To any functor $f \colon \Cc \to \inftyCatu V$ and any ind-object $c$ colimit of a diagram
\[
\mymatrix{
K \ar[r]^-{\bar c} & \Cc \ar[r] & \Indu U(\Cc)
}
\]
we construct an $(\infty,1)$-category
\[
\Indext_\Cc(f)(c) \simeq \Indu U\left(\colim f(\bar c)\right)
\]
To any pro-object $d$ limit of a diagram
\[
\mymatrix{
K\op \ar[r]^-{\bar d} & \Cc \ar[r] & \Prou U(\Cc)
}
\]
we associate an $(\infty,1)$-category
\[
\Proext_\Cc(f)(d) \simeq \lim \Prou U(f(\bar d))
\]
\end{rmq}

\begin{df}
Let $\inftyCat^{\mathbb V\mathrm{,st}}$ denote the subcategory of $\inftyCatu V$ spanned by stable categories with exact functors between them -- see \cite[1.1.4]{lurie:halg}.
Let $\inftyCat^{\mathbb V\mathrm{,st,id}}$ denote the full subcategory of $\inftyCat^{\mathbb V\mathrm{,st}}$ spanned by idempotent complete stable categories.
\end{df}

\begin{rmq}
It follows from \cite[1.1.4.6, 1.1.3.6, 1.1.1.13 and 1.1.4.4]{lurie:halg} that the functors $\Indext_\Cc$ and $\Proext_\Cc$ restricts to functors
\begin{align*}
& \Indext_\Cc \colon \Fct(\Cc, \inftyCat^{\mathbb V\mathrm{,st}}) \to \Fct(\Indu U (\Cc), \inftyCat^{\mathbb V\mathrm{,st}}) \\
& \Proext_\Cc \colon \Fct(\Cc, \inftyCat^{\mathbb V\mathrm{,st}}) \to \Fct(\Prou U (\Cc), \inftyCat^{\mathbb V\mathrm{,st}})
\end{align*}
\end{rmq}

\paragraph{Symmetric monoidal $(\infty,1)$-categories:}
We will make use in the last chapter of the theory of symmetric monoidal $(\infty,1)$-categories as developed in \cite{lurie:halg}. Let us give a (very) quick review of those objects.

\begin{df}\label{ptfin}
Let $\ptfin$ denote the category of pointed finite sets. For any $n \in \N$, we will denote by $\langle n \rangle$ the set $\{\pt, 1, \dots ,n\}$ pointed at $\pt$.
For any $n$ and $i \leq n$, the Segal map $\delta^i \colon \langle n \rangle \to \langle 1 \rangle$ is defined by $\delta^i(j) = 1$ if $j = i$ and $\delta^i(j) = \pt$ otherwise.
\end{df}
\begin{df}(see \cite[2.0.0.7]{lurie:halg})\label{monoidalcats}
Let $\Cc$ be an $(\infty,1)$-category. A symmetric monoidal structure on $\Cc$ is the datum of a coCartesian fibration $p \colon \Cc^{\otimes} \to \ptfin$ such that
\begin{itemize}
\item The fibre category $\Cc^{\otimes}_{\langle 1 \rangle}$ is equivalent to $\Cc$ and
\item For any $n$, the Segal maps induce an equivalence $\Cc^{\otimes}_{\langle n \rangle} \to (\Cc^{\otimes}_{\langle 1 \rangle})^n \simeq \Cc^n$.
\end{itemize}
where $\Cc^{\otimes}_{\langle n \rangle}$ denote the fibre of $p$ at $\langle n \rangle$.
We will denote by $\monoidalinftyCatu V$ the $(\infty,1)$-category of $\mathbb V$-small symmetric monoidal $(\infty,1)$-categories -- see \cite[2.1.4.13]{lurie:halg}.
\end{df}
Such a coCartesian fibration is classified by a functor $\phi \colon \ptfin \to \inftyCatu V$ -- see \cite[3.3.2.2]{lurie:htt} -- such that $\phi(\langle n \rangle) \simeq \Cc^n$.
The tensor product on $\Cc$ is induced by the map of pointed finite sets $\mu \colon \langle 2 \rangle \to \langle 1 \rangle$ mapping both $1$ and $2$ to $1$
\[
\otimes = \phi(\mu) \colon \Cc^2 \to \Cc
\]

\begin{rmq}\label{indext-monoidal}
The forgetful functor $\monoidalinftyCatu V \to \inftyCatu V$ preserves all limits as well as filtered colimits -- see \cite[3.2.2.4 and 3.2.3.2]{lurie:halg}.
Moreover, it follows from \autoref{indu-thm} - \ref{indmonoidal} that the functor $\Indu U$ induces a functor
\[
\Indu U \colon \monoidalinftyCatu V \to \monoidalinftyCatu V
\]
The same holds for $\Prou U$.
The constructions $\Indext$ and $\Proext$ therefore restrict to
\begin{align*}
&\Indext_\Cc \colon \Fct(\Cc,\monoidalinftyCatu V) \to \Fct(\Indu U(\Cc), \monoidalinftyCatu V) \\
&\Proext_\Cc \colon \Fct(\Cc,\monoidalinftyCatu V) \to \Fct(\Prou U(\Cc), \monoidalinftyCatu V)
\end{align*}
\end{rmq}

\paragraph{Tate objects:}
We now recall the definition and a few properties of Tate objects in a stable and idempotent complete $(\infty,1)$-category. The content of this paragraph comes from \cite{hennion:tate}. See also \cite{hennion:these}.
\begin{df}
Let $\Cc$ be a stable and idempotent complete $(\infty,1)$-category.
Let $\Tateu U(\Cc)$ denote the smallest full subcategory of $\Prou U \Indu U(\Cc)$ containing $\Indu U(\Cc)$ and $\Prou U(\Cc)$, and both stable and idempotent complete.
\end{df}

The category $\Tateu U(\Cc)$ naturally embeds into $\Indu U\Prou U(\Cc)$ as well.

\begin{prop}\label{tate-dual}
If moreover $\Cc$ is endowed with a duality equivalence $\Cc\op \to^\sim \Cc$ then the induced functor
\[
\Prou U \Indu U(\Cc) \to \left(\Prou U \Indu U(\Cc)\right)\op \simeq \Indu U\Prou U(\Cc)
\]
preserves Tate objects and induces an equivalence $\Tateu U(\Cc) \simeq \Tateu U(\Cc)\op$.
\end{prop}

\begin{df}
Let $\Cc$ be a $\mathbb V$-small $(\infty,1)$-category.
We define the functor
\[
\Tateextu U \colon \mymatrix@1{\Fct(\Cc,\inftyCat^{\mathbb V\mathrm{,st}}) \ar[r]^-i & \Fct(\Indu U(\Cc),\inftyCat^{\mathbb V\mathrm{,st}}) \ar[rr]^-{\Tateu U \circ -} && \Fct(\Indu U(\Cc),\inftyCat^{\mathbb V\mathrm{,st,id}})}
\]
\end{df}

\subsection{Derived algebraic geometry}

We present here some background results about derived algebraic geometry.
Let us assume $k$ is a field of characteristic $0$.
First introduced by Toën and Vezzosi in \cite{toen:hagii}, derived algebraic geometry is a generalisation of algebraic geometry in which we replace commutative algebras over $k$ by commutative differential graded algebras (or cdga's).
We refer to \cite{toen:dagems} for a recent survey of this theory.

\paragraph{Generalities on derived stacks:}We will denote by $\cdga_k$ the $(\infty,1)$-category of cdga's over $k$ concentrated in non-positive cohomological degree. It is the $(\infty,1)$-localisation of a model category along weak equivalences. 
Let us denote $\dAff_k$ the opposite $(\infty,1)$-category of $\cdga_k$. It is the category of derived affine schemes over $k$.
In this work, we will adopt a cohomological convention for cdga's.

A derived prestack is a presheaf $\dAff_k\op \simeq \cdga_k \to \sSets$. We will thus write $\presh(\dAff_k)$ for the $(\infty,1)$-category of derived prestacks.
A derived stack is a prestack satisfying the étale descent condition. We will denote by $\dSt_k$ the $(\infty,1)$-category of derived stacks.
It comes with an adjunction
\[
(-)^+ \colon \presh(\dAff_k) \rightleftarrows \dSt_k
\]
where the left adjoint $(-)^+$ is called the stackification functor.
\begin{rmq}
The categories of varieties, schemes or (non derived) stacks embed into $\dSt_k$.
\end{rmq}

\begin{df}
The $(\infty,1)$-category of derived stacks admits an internal hom $\Mapstack(X,Y)$ between two stacks $X$ and $Y$. It is the functor $\cdga_k \to \sSets$ defined by
\[
A \mapsto \Map_{\dSt_k}(X \times \Spec A, Y)
\]
We will call it the mapping stack from $X$ to $Y$.
\end{df}

There is a derived version of Artin stacks of which we first give a recursive definition.
\begin{df}(see for instance \cite[5.2.2]{toenmoerdijk:crm})\label{artinstacks}
Let $X$ be a derived stack.
\begin{itemize}
\item We say that $X$ is a derived $0$-Artin stack if it is a derived affine scheme ;
\item We say that $X$ is a derived $n$-Artin stack if there is a family $(T_\alpha)$ of derived affine schemes and a smooth atlas
\[
u \colon \coprod T_\alpha \to X
\]
such that the nerve of $u$ has values in derived $(n-1)$-Artin stacks ;
\item We say that $X$ is a derived Artin (or algebraic) stack if it is an $n$-Artin stack for some $n$.
\item We say that $X$ is locally of finite presentation if there exists a smooth atlas $\bigsqcup T_\alpha \to X$ as above, such that the derived affine schemes $T_\alpha$ are all finitely presented (ie their cdga of functions is finitely presented, or equivalently compact is the category of cdga's).
We also say that $X$ is finitely presented if there is such an atlas with a finite number of $T_\alpha$'s.
\end{itemize}
We will denote by $\dStArt_k$ the full subcategory of $\dSt_k$ spanned by derived Artin stacks.
\end{df}

\begin{df}
A morphism $X \to Y$ of derived stacks is called algebraic if for any cdga $A$ and any map $\Spec A \to Y$, the derived intersection $X \times_Y \Spec A$ is an algebraic stack.
\end{df}

To any cdga $A$ we associate the category $\dgMod_A$ of dg-modules over $A$.
Similarly, to any derived stack $X$ we can associate a derived category $\Qcoh(X)$ of quasicoherent sheaves. It is a $\mathbb U$-presentable $(\infty,1)$-category given by the formula
\[
\Qcoh(X) \simeq \lim_{\Spec A \to X} \dgMod_A
\]
Moreover, for any map $f \colon X \to Y$, there is a natural pull back functor $f^* \colon \Qcoh(Y) \to \Qcoh(X)$. This functor admits a right adjoint, which we will denote by $f_*$.
This construction is actually a functor of $(\infty,1)$-categories.
\begin{df}
Let us denote by $\Qcoh$ the functor
\[
\Qcoh \colon \dSt_k\op \to \PresLeftu U
\]
For any $X$ we can identify a full subcategory $\Perf(X) \subset \Qcoh(X)$ of perfect complexes. This defines a functor
\[
\Perf \colon \dSt_k\op \to \inftyCatu U
\]
\end{df}

\begin{rmq}
For any derived stack $X$ the categories $\Qcoh(X)$ and $\Perf(X)$ are actually stable and idempotent complete $(\infty,1)$-categories. The inclusion $\Perf(X) \to \Qcoh(X)$ is exact. Moreover, for any map $f \colon X \to Y$ the pull back functor $f^*$ preserves perfect modules and is also exact. 
\end{rmq}

\begin{df}[see for instance the appendix of \cite{halpernleistern-preygel:properness}]
Let $X$ be a derived stack and let $\pi \colon X \to \pt$ denote the projection. We say that $X$ is of finite cohomological dimension if there is a non-negative integer $d$ such that the complex $\pi_* \Oo_X = \R \Gamma(X,\Oo_X) \in \dgMod_k$ is concentrated in degree lower or equal to $d$.
\end{df}

\begin{ex}
Any derived affine scheme is of finite cohomological dimension (take $d = 0$). Any quasi-compact quasi-separated derived stack (ie a finite colimit of derived affine schemes) is of finite cohomological dimension.
\end{ex}

Any derived Artin stack $X$ over a basis $S$ admits a cotangent complex $\Lcot_{X/S} \in \Qcoh(X)$. If $X$ is locally of finite presentation, then the its cotangent complex is perfect
\[
\Lcot_{X/S} \in \Perf(X)
\]

\paragraph{Symplectic structures:}
Following \cite{ptvv:dersymp}, to any derived stack $X$ we associate two complexes $\forms p(X)$ and $\closedforms p(X)$ in $\dgMod_k$, respectively of $p$-forms and closed $p$-forms on $X$.
They come with a natural morphism $\closedforms p(X) \to \forms p(X)$ forgetting the lock closing the forms\footnote{This lock is a structure on the form: being closed in not a property in this context.}.
This actually glues into a natural transformation
\[
\mymatrix{
\dSt_k \dcell[r][\closedforms p][\forms p] & \dgMod_k
}
\]
Let us emphasize that, as soon as $X$ is Artin, the complex $\forms 2(X)$ is canonically equivalent to the global sections complex of $\Lcot_X \wedge \Lcot_X$. In particular, any $n$-shifted $2$-forms $k[-n] \to \forms p(X)$ induces a morphism $\Oo_X[-n] \to \Lcot_X \wedge \Lcot_X$ in $\Qcoh(X)$.
If $X$ is locally of finite presentation, the cotangent $\Lcot_X$ is perfect and we then get a map
\[
\T_X[-n] \to \Lcot_X
\]
\begin{df}
Let $X$ be a derived Artin stack locally of finite presentation.
\begin{itemize}
\item An $n$-shifted $2$-form $\omega_X \colon k[-n] \to \forms 2(X)$ is called non-degenerated if the induced morphism $\T_X[-n] \to \Lcot_X$ is an equivalence;
\item An $n$-shifted symplectic form on $X$ is a $n$-shifted closed $2$-form $\omega_X \colon k[-n] \to \closedforms 2(X)$ such that underlying $2$-form $k[-n] \to \closedforms 2(X) \to \forms 2(X)$ is non degenerate.
\end{itemize}
\end{df}

\paragraph{Obstruction theory:}
Let $A \in \cdga_k$ and let $M \in \dgMod_A^{\leq -1}$ be an $A$-module concentrated in negative cohomological degrees. Let $d$ be a derivation $A \to A \oplus M$ and $s \colon A \to A \oplus M$ be the trivial derivation.
The square zero extension of $A$ by $M[-1]$ twisted by $d$ is the fibre product
\[
\mymatrix{
A \oplus_d M[-1] \cart \ar[r]^-p \ar[d] & A \ar[d]^d \\ A \ar[r]^-s & A \oplus M
}
\]
Let now $X$ be a derived stack and $M \in \Qcoh(X)^{\leq -1}$. We will denote by $X[M]$ the trivial square zero extension of $X$ by $M$.
Let also $d \colon X[M] \to X$ be a derivation -- ie a retract of the natural map $X \to X[M]$. We define the square zero extension of $X$ by $M[-1]$ twisted by $d$ as the colimit
\[
X_d[M[-1]] = \colim_{f \colon \Spec A \to X} \Spec(A \oplus_{f^*d} f^* M[-1])
\]
It is endowed with a natural morphism $X \to X_d[M[-1]]$ induced by the projections $p$ as above.
\begin{prop}[Obstruction theory on stacks]\label{obstruction}
Let $F \to G$ be an algebraic morphism of derived stacks. Let $X$ be a derived stack and let $M \in \Qcoh(X)^{\leq -1}$. Let $d$ be a derivation
\[
d \in \Map_{X/-}(X[M], X)
\]
We consider the map of simplicial sets
\[
\psi \colon \Map(X_d[M[-1]],F) \to \Map(X,F) \times_{\Map(X,G)} \Map(X_d[M[-1]],G)
\]
Let $y \in \Map(X,F) \times_{\Map(X,G)} \Map(X_d[M[-1]],G)$ and let $x \in \Map(X,F)$ be the induced map.
There exists a point $\alpha(y) \in \Map(x^* \Lcot_{F/G}, M)$ such that the fibre $\psi_y$ of $\psi$ at $y$ is equivalent to the space of paths from $0$ to $\alpha(y)$ in $\Map(x^* \Lcot_{F/G}, M)$
\[
\psi_y \simeq \Omega_{0,\alpha(y)}\Map(x^* \Lcot_{F/G}, M)
\]
\end{prop}
\begin{proof}
This is a simple generalisation of \cite[1.4.2.6]{toen:hagii}. The proof is very similar.
We have a natural commutative square
\[
\mymatrix{
X[M] \ar[r]^-d \ar[d] & X \ar[d] \\ X \ar[r] & X_d[M[-1]]
}
\]
It induces a map
\[
\alpha \colon \Map(X,F) \times_{\Map(X,G)} \Map(X_d[M[-1]],G) \to \Map_{X/-/G}(X[M],F) \simeq \Map(x^* \Lcot_{F/G},M)
\]
Let $\Omega_{0,\alpha(y)}\Map_{X/-/G}(X[M],F)$ denote the space of paths from $0$ to $\alpha(y)$. It is the fibre product
\[
\mymatrix{
\Omega_{0,\alpha(y)}\Map_{X/-/G}(X[M],F) \cart[][20] \ar[r] \ar[d] & \pt \ar[d]^{\alpha(y)} \\
\pt \ar[r]^-0 & \Map_{X/-/G}(X[M],F)
}
\]
The composite map $\alpha \psi$ is by definition homotopic to the $0$ map. This defines a morphism
\[
f \colon \Omega_{0,\alpha(y)}\Map_{X/-/G}(X[M],F) \to \psi_y
\]
It now suffices to see that the category of $X$'s for which $f$ is an equivalence contains every derived affine scheme and is stable by colimits.
The first assertion is exactly \cite[1.4.2.6]{toen:hagii} and the second one is trivial.
\end{proof}

\paragraph{Postnikov towers:}
To any cdga $A$, one can associate its $n$-truncation $A_{\leq n}$ for some $n$. It is, by definition, the universal cdga with vanishing cohomology $\homol^p(A_{\leq n})$ for $p < -n$ associated to $A$.
The truncation comes with a canonical map $A \to A_{\leq n}$ so that one can form the diagram
\[
A_{\leq 0} \from A_{\leq 1} \from \dots
\]
This induces a canonical morphism $A \to \lim_n A_{\leq n}$ which is an equivalence.

This phenomenon has a counterpart when dealing with derived stacks. We denote by $\cdgaunbounded_k^{[-n,0]}$ the category of cdga's with cohomology concentrated in degrees $-n$ to $0$. It comes with the fully faithful embedding $i_n \colon \cdgaunbounded_k^{[-n,0]} \to \cdga_k$.

For any prestack $X \colon \cdga_k \to \sSets$, we define its truncation $\tau_{\leq n} X$ as the restriction of $X$ to the category $\cdgaunbounded_k^{[-n,0]}$. We will abuse notations and also denote by $\tau_{\leq n} X$ the functor obtained after left Kan extending along $i_n$.
The prestack $\tau_{\leq n} X$ comes with a canonical morphism $\tau_{\leq n} X \to X$ which, as $n$ varies, assembles to define a canonical map
\[
\colim_n \tau_{\leq n} X \to X
\]
Remark that this morphism is not necessarily an equivalence. We will study it in \autoref{coconnective}.

\paragraph{Algebraisable stacks:}
Let $X$ be a derived stack and $A$ be a cdga. Let $a = (\el{a}{p})$ be a sequence of elements of $A^0$ forming a regular sequence in $\homol^0(A)$.
Let $\quot{A}{\el{a^n}{p}}$ denote the Kozsul complex associated with the regular sequence $(\el{a^n}{p})$. It is endowed with a cdga structure.
There is a canonical map 
\[
\psi(A)_a \colon \colim_n X\left( \textstyle \quot{A}{\el{a^n}{p}} \right) \to X\left(\lim_n \textstyle \quot{A}{\el{a^n}{p}} \right)
\]
This map is usually not an equivalence.
\begin{df}\label{alg-diag}
A derived stack $X$ is called algebraisable if for any $A$ and any regular sequence $a$ the map $\psi(A)_a$ is an equivalence.

A map $f \colon X \to Y$ is called algebraisable if for any derived affine scheme $T$ and any map $T \to Y$, the fibre product $X \times_Y T$ is algebraisable.

We will say that a derived stack $X$ has algebraisable diagonal if the diagonal morphism $X \to X \times X$ is algebraisable.
\end{df}

\begin{rmq}
A derived stack $X$ has algebraisable diagonal if for any $A$ and $a$ the map $\psi(A)_a$ is fully faithful.
One could also rephrase the definition of being algebraisable as follows. A stack is algebraisable if it does not detect the difference between
\[
\colim_n \Spec\left( \textstyle \quot{A}{\el{a^n}{p}} \right) \text{~~~and~~~} \Spec\left( \lim_n \textstyle \quot{A}{\el{a^n}{p}} \right)
\]
\end{rmq}

\begin{ex}
Any derived affine scheme is algebraisable. Another important example of algebraisable stack is the stack of perfect complexes. 
In \cite{bhatt:algebraisable}, Bhargav Bhatt gives some more examples of algebraisable (non-derived) stacks -- although our definition slightly differs from his.
He proves that any quasi-compact quasi-separated algebraic space is algebraisable and also provides with examples of non-algebraisable stacks. Let us name $\mathrm K(\Gm,2)$ -- the Eilenberg-Maclane classifying stack of $\Gm$ -- as an example of non-algebraisable stack.
Algebraisability of Deligne-Mumford stacks is also look at in \cite{lurie:dagxii}.
\end{ex}

\section{Ind-pro-stacks}\label{chapterIP}%

Throughout this section, we will denote by $S$ a derived stack over some base field $k$ and by $\dSt_S$ the category of derived stack over the base $S$.

\subsection{Cotangent complex of a pro-stack}
\begin{df}
A pro-stack over $S$ an object of $\Prou U \dSt_S$.
\end{df}

\begin{rmq}
Note that the category $\Prou U \dSt_S$ is equivalent to the category of pro-stacks over $k$ with a morphism to $S$.
\end{rmq}

\begin{df}\label{iperf-df}
Let $\Perf \colon \dSt_S\op \to \inftyCatu U$ denote the functor mapping a stack to its category of perfect complexes.
We will denote by $\IPerf$ the functor
\[
\IPerf = \Indext_{\dSt_S\op}(\Perf) \colon (\Prou U\dSt_S)\op \to \PresLeft
\]
where $\Indext$ was defined in \autoref{indext}.
Whenever $X$ is a pro-stack, we will call $\IPerf(X)$ the derived category of ind-complexes on $X$. It is $\mathbb U$-presentable.
If $f \colon X \to Y$ is a map of pro-stacks, then the functor
\[
\IPerf(f) \colon \IPerf(Y) \to \IPerf(X)
\]
admits a right adjoint (as both the involved categories are presentable and the functor preserves all colimits). We will denote $f_\mathbf{I}^* = \IPerf(f)$ and $f^\mathbf{I}_*$ its right adjoint.
\end{df}

\begin{rmq}
Let $X$ be a pro-stack and let $\bar X \colon K\op \to \dSt_S$ denote a $\mathbb U$-small cofiltered diagram of whom $X$ is a limit in $\Prou U\dSt_S$.
The derived category of ind-perfect complexes on $X$ is by definition the category
\[
\IPerf(X) = \Indu U(\colim \Perf(\bar X))
\]
It thus follows from \cite[1.1.4.6 and 1.1.3.6]{lurie:halg} that $\IPerf(X)$ is stable.
Note that it is also equivalent to the colimit
\[
\IPerf(X) = \colim \IPerf(\bar X) \in \PresLeftu V
\]
It is therefore equivalent to the limit of the diagram
\[
\IPerf_*(\bar X) \colon K \to \dSt_S\op \to \PresLeftu V \simeq (\PresRightu V)\op
\]
An object $E$ in $\IPerf(X)$ is therefore the datum of an object $E_k$ of $\IPerf(\bar X(k))$ for each $k \in K$ and of some compatibilities between them. We will then have $E_k \simeq {p_k}_* E$ where $p_k \colon X \to \bar X(k)$ is the natural projection. 
\end{rmq}

\begin{df}
Let $X$ be a pro-stack. We define its derived category of pro-perfect complexes
\[
\PPerf(X) = \left( \IPerf(X) \right) \op
\]
Recall that perfect complexes are precisely the dualizable objects in the category of quasi-coherent complexes. They therefore come with a duality equivalence $\Perf(-) \to^\sim (\Perf(-))\op$. This gives rise to the equivalence
\[
\PPerf(X) \simeq \Prou U(\colim \Perf(\bar X))
\]
whenever $\bar X \colon \K\op \to \dSt_S$ is a cofiltered diagram of whom $X$ is a limit in $\Prou U\dSt_S$.
\end{df}

\begin{df}
Let us define the functor $\Tateu U_\mathbf P \colon (\Prou U \dSt_S)\op \to \inftyCat^{\mathbb V\mathrm{,st,id}}$
\[
\Tateu U_\mathbf P = \Tateextu U_{\dSt_S\op}(\Perf)
\]
\end{df}

\begin{rmq}
The functor $\Tateu U_\mathbf P$ maps a pro-stack $X$ given by a diagram $\bar X \colon K\op \to \dSt_S$ to the stable $(\infty,1)$-category
\[
\Tateu U_\mathbf P(X) = \Tateu U(\colim \Perf(\bar X))
\]
There is a canonical fully faithful natural transformation
\[
\Tateu U_\mathbf P \to \Prou U \circ \IPerf
\]
We also get a fully faithful
\[
\Tateu U_\mathbf P \to \Indu U \circ \PPerf
\]
\end{rmq}

\begin{df}
Let $\Qcoh \colon \dSt_S\op \to \inftyCatu V$ denote the functor mapping a derived stack to its derived category of quasi-coherent sheaves. It maps a morphism between stacks to the appropriate pullback functor.
We will denote by $\IQcoh$ the functor
\[
\IQcoh = \Indext_{\dSt_S\op}(\Qcoh) \colon (\Prou U\dSt_S)\op \to \inftyCatu V
\]
If $f \colon X \to Y$ is a map of pro-stacks, we will denote by $f^*_\mathbf{I}$ the functor $\IQcoh(f)$.
We also define
\[
\IQcoh^{\leq 0} = \Indext_{\dSt_S\op}(\Qcoh^{\leq 0})
\]
the functor of connective modules.
\end{df}

\begin{rmq}
There is a fully faithful natural transformation $\IPerf \to \IQcoh$ ; for any map $f \colon X \to Y$ of pro-stacks, there is therefore a commutative diagram
\[
\mymatrix{
\IPerf(Y) \ar[r] \ar[d]_{f_\mathbf I^*} & \IQcoh(Y) \ar[d]^{f_\mathbf I^*}\\
\IPerf(X) \ar[r] & \IQcoh(X)
}
\]
The two functors denoted by $f_\mathbf I^*$ are thus compatible.
Let us also say that the functor 
\[
f_\mathbf I^* \colon \IQcoh(Y) \to \IQcoh(X)
\]
does \emph{not} need to have a right adjoint. We next show that it sometimes has one.
\end{rmq}
\begin{prop}\label{prop-iqcoh}
Let $f \colon X \to Y$ be a map of pro-stacks. If $Y$ is actually a stack then the functor $f_\mathbf I^* \colon \IQcoh(Y) \to \IQcoh(X)$ admits a right adjoint.
\end{prop}
\begin{proof}[sketch of]
For a complete proof, we refer to \cite[1.2.0.8]{hennion:these}.
Let us denote by $\bar X \colon K\op \to \dSt_S$ a cofiltered diagram of whom $X$ is a limit in $\Prou U \dSt_S$. The map $X \to Y$ factors through the projection $X \to \bar X(k)$ for some $k \in K$. The right adjoint of $f_\mathbf I^*$ is then (informally) given by the limit
\[
\lim_{k \to l} \bar f(l)_*
\]
where $\bar f(l)_*$ is the right adjoint to the induced functor $\bar f(l)^* \colon \IQcoh(Y) \to \IQcoh(\bar X(l))$.
\end{proof}

\newcommand{\IQ}{\mathbf{IQ}}
\begin{df}
Let $f \colon X \to Y$ be a map of pro-stacks. We will denote by $f_*^\IQ$ the right adjoint to $f_\mathbf I^* \colon \IQcoh(Y) \to \IQcoh(X)$ \emph{if it exists}.
\end{df}

\begin{rmq}
In the situation of \autoref{prop-iqcoh}, there is a natural transformation
\[
\mymatrix{
\IPerf(X) \ar[r] \ar[d]_{f_*^\mathbf I} & \IQcoh(X) \ar[d]^{f_*^\IQ} \\ \IPerf(Y) \ar@{=>}[ur] \ar[r] & \IQcoh(Y)
}
\]
It does not need to be an equivalence.
\end{rmq}

\begin{df}
Let $X$ be a pro-stack over $S$. The structural sheaf $\Oo_X$ of $X$ is the pull-back of $\Oo_S$ along the structural map $X \to S$.
\end{df}

\begin{ex}
Let $X$ be a pro-stack over $S$ and $\bar X \colon K\op \to \dSt_S$ be a $\mathbb U$-small cofiltered diagram of whom $X$ is a limit in $\Prou U\dSt_S$.
Let $k$ be a vertex of $K$, let $X_k$ denote $\bar X(k)$ and let $p_k$ denote the induced map of pro-stacks $X \to X_k$. If $f \colon k \to l$ is an arrow in $K$, we will also denote by $f$ the map of stacks $\bar X(f)$.
We have
\[
(p_k)_*^\IQ (\Oo_X) \simeq \colim_{f \colon k \to l} f_* \Oo_{X_l}
\]
One can see this using \cite[1.2.0.7]{hennion:these}.
\end{ex}

\begin{df}
For any category $\Cc$ with finite colimits, we will denote by $\btw^{\amalg}_\Cc$ the functor $\Cc^{\Delta^1} \to \inftyCatu V$ mapping a morphism $\phi \colon c \to d$ to the category of factorizations $c \to e \to d$ of $\phi$.
For a formal definition in the context of $(\infty,1)$-categories, we refer to \cite[1.3.0.14]{hennion:these}.
\end{df}

\newcommand{\sqzext}{\operatorname{Ex}}
\begin{df}\label{derivation-pdst}
Let $T$ be a stack over $S$. Let us consider the functor 
\[
\Qcoh(T)^{\leq 0} \to \btw_{\dSt_S\op}^{\amalg}(\id_T) \simeq \left(\comma{T}{\dSt_T}\right)\op
\]
mapping a quasi-coherent sheaf $E$ to the square zero extension $T \to T[E] \to T$. This construction is functorial in $T$ and actually comes from a natural transformation
\[
\sqzext \colon \Qcoh^{\leq 0} \to \btw_{\dSt_S\op}^{\amalg}(\id_-)
\]
of functors $\dSt_S\op \to \inftyCatu V$.
We will denote by $\sqzext^{\Pro}$ the natural transformation
\[
\sqzext^{\Pro} = \Indext_{\dSt_S\op}(\sqzext) \colon \IQcoh^{\leq 0} \to \Indextu U_{\dSt_S\op}(\btw_{\dSt_S\op}^{\amalg}(\id_-)) \simeq \btw_{(\Prou 
U\dSt_S)\op}^{\amalg}(\id_-)
\]
between functors $(\Prou U\dSt_S)\op \to \inftyCat$. The equivalence on the right is the one from \cite[1.3.0.18]{hennion:these}.
If $X$ is a pro-stack and $E \in \IQcoh(X)^{\leq 0}$ then we will denote by $X \to X[E] \to X$ the image of $E$ by the functor $\sqzext^{\Pro}(X)$.
\end{df}
\begin{rmq}
Let us give a description of this functor.
Let $X$ be a pro-stack and let $\bar X \colon K\op \to \dSt_S$ denote a $\mathbb U$-small cofiltered diagram of whom $X$ is a limit in $\Prou U\dSt_S$. For every $k \in K$ we can compose the functor mentioned above with the base change functor
\[
\mymatrix{
(\Qcoh(X_k))\op \ar[r]^-{X_k[-]} & \comma{X_k}{\dSt_{X_k}} \ar[r]^-{- \times_{X_k} X} & \comma{X}{\Prou U\dSt_X}
}
\]
This is functorial in $k$ and we get a functor $\left(\colim \Qcoh(\bar X) \right)\op \to \comma{X}{\Prou U\dSt_X}$ which we extend and obtain a more explicit description of the square zero extension functor
\[
X[-] \colon (\IQcoh(X))\op \to \comma{X}{\Prou U\dSt_X}
\]
\end{rmq}

\begin{df}
Let $X$ be a pro-stack. 
\begin{itemize}
\item We finally define the functor of derivations over $X$ :
\[
\Der(X,-) = \Map_{X/-/S}(X[-],X) \colon \IQcoh(X)^{\leq 0} \to \sSets
\]
\item We say that $X$ admits a cotangent complex if the functor $\Der(X,-)$ is corepresentable -- ie there exists a $\Lcot_{X/S} \in \IQcoh(X)$ such that
for any $E \in \IQcoh(X)^{\leq 0}$
\[
\Der(X,E) \simeq \Map(\Lcot_{X/S},E)
\]
\end{itemize}
\end{df}

\begin{df}
Let $\dStArt_S$ be the full sub-category of $\dSt_S$ spanned by derived Artin stacks over $S$. 
An Artin pro-stack is an object of $\Prou U \dStArt_S$.
Let $\dStArtlfp_S$ be the full sub-category of $\dStArt_S$ spanned by derived Artin stacks locally of finite presentation over $S$.
An Artin pro-stack locally of finite presentation is an object of $\Prou U\dStArtlfp_S$
\end{df}

\begin{prop}\label{cotangent-pdst}
Any Artin pro-stack $X$ over $S$ admits a cotangent complex $\Lcot_{X/S}$.
Let us assume that $\bar X \colon K\op \to \dStArt_S$ is a $\mathbb U$-small cofiltered diagram of whom $X$ is a limit in $\Prou U\dStArt_S$. When $k$ is a vertex of $K$, let us denote by $X_k$ the derived Artin stack $\bar X(k)$. If $f \colon k \to l$ is an arrow in $K$, we will also denote by $f \colon X_l \to X_k$ the map of stacks $\bar X(f)$.
The cotangent complex is given by the formula
\[
\Lcot_{X/S} = \colim_k p_k^* \Lcot_{X_k/S} \in \Indu U\left(\colim \Qcoh(\bar X) \right) \simeq \IQcoh(X)
\]
where $p_k$ is the canonical map $X \to X_k$.
The following formula stands
\[
{p_k}_*^\IQ \Lcot_{X/S} \simeq \colim_{f \colon k \to l} f_* \Lcot_{X_l/S}
\]
If $X$ is moreover locally of finite presentation over $S$, then its cotangent complex belongs to $\IPerf(X)$.
\end{prop}
Before proving this proposition, let us fix the following notation
\begin{df}
Let $\Cc$ be a full sub-category of an $\infty$-category $\Dd$. There is a natural transformation from $\overcat_\Dd \colon d \mapsto \quot{\Dd}{d}$ to the constant functor $\Dd \colon \Dd \to \inftyCat$. We denote by $\overcat_\Dd^\Cc$ the fiber product
\[
\overcat_\Dd^\Cc = \overcat_\Dd \times_\Dd \Cc \colon \Dd \to \inftyCat
\]
\end{df}
\begin{rmq}
The functor $\overcat_\Dd^\Cc \colon \Dd \to \inftyCat$ maps an object $d \in \Dd$ to the comma category of objects in $\Cc$ over $d$
\[
\quot{\Cc}{d} = (\Cc \times \{d\} ) \timesunder[\Dd \times \Dd] \Dd^{\Delta^1}
\]
\end{rmq}
\begin{proof}[of the proposition]
The cotangent complex defines a natural transformation
\[
\cotangent \colon \overcat_{\dSt_S\op}^{(\dStArt_S)\op}  \to \Qcoh(-)
\]
To any stack $T$ and any Artin stack $U$ over $S$ with a map $f \colon T \to U$, it associates the quasi-coherent complex $f^* \Lcot_{U/S}$ on $T$.
Applying the functor $\Indext_{\dSt_S\op}$ we get a natural transformation $\cotangent^{\Pro}$
\[
\cotangent^{\Pro} = \Indext_{\dSt_S\op}(\cotangent) \colon \overcat^{(\Prou U\dStArt_S)\op}_{(\Prou U\dSt_S)\op} \to \IQcoh(-)
\]
Specifying it to $X$ we get a functor
\[
\cotangent^{\Pro}_X \colon \left( \comma{X}{\Prou U\dStArt_S} \right)\op \to \IQcoh(X)
\]
Let us set $\Lcot_{X/S} =\lambda^{\Pro}_X(X) \in \IQcoh(X)$. We have by definition the equivalence
\[
\Lcot_{X/S} \simeq \colim_k p_k^* \Lcot_{X_k/S}
\]
Let us now check that it satisfies the required universal property.
The functor $\Der(X,-)$ is the limit of the diagram $K\op \to \Fct(\IQcoh(X)^{\leq 0}, \sSets)$
\[
\Map_{X/-/S}(X[-], \bar X)
\]
Fixing $k \in K$, the functor $\Map_{X/-/S}(X[-], X_k) \colon \IQcoh(X)^{\leq 0} \simeq \Indu U(\colim \Qcoh(\bar X)^{\leq 0}) \to \sSets$ preserves filtered colimits. It is hence induced by its restriction to $\colim \Qcoh(\bar X)^{\leq 0}$. It follows that the diagram $\Map_{X/-/S}(X[-], \bar X)$ factors through a diagram
\[
\delta \colon K\op \to \Fct\left(\colim \Qcoh(\bar X)^{\leq 0}, \sSets\right) \simeq \lim \Fct(\Qcoh(\bar X)^{\leq 0}, \sSets)
\]
Similarly, the functor $\Map(\Lcot_{X/S},-)$ is the limit of a diagram
\[
\mymatrix{K\op \ar[r]^-\mu & \lim \Fct(\Qcoh(\bar X)^{\leq 0}, \sSets) \ar[r] & \Fct(\IQcoh(X)^{\leq 0}, \sSets)}
\]
The universal property of the usual cotangent complex defines an equivalence between $\delta$ and $\mu$.

To get the formula for ${p_k}_*^\IQ \Lcot_{X/S}$, one uses \cite[1.2.0.7]{hennion:these} and the last statement is obvious.
\end{proof}

\begin{rmq}
The definition of the derived category of ind-quasi-coherent modules on a pro-stack is build for the above proposition and remark to hold.
\end{rmq}

\begin{rmq}\label{cotangent-prostacks}
We have actually proven that for any pro-stack $X$, the two functors 
\[
\IQcoh(X)^{\leq 0} \times \comma{X}{\dStArt_S} \to \sSets
\]
defined by
\begin{align*}
(E,Y) &\mapsto \Map_{X/-/S}(X[E], Y) \\
(E,Y) &\mapsto \Map_{\IQcoh(X)}(\lambda_X^{\Pro}(Y), E)
\end{align*}
are equivalent.
\end{rmq}

\subsection{Cotangent complex of an ind-pro-stack}

\begin{df}
An ind-pro-stack is an object of the category
\[
\IP\dSt_S = \Indu U \Prou U\dSt_S
\]
\end{df}

\begin{df}
Let us define the functor $\PIPerf \colon (\IP\dSt_S)\op \to \inftyCatu V$ as
\[
\PIPerf = \Proext_{(\Prou U\dSt_S)\op}(\IPerf)
\]
where $\Proext$ was defined in \autoref{indext}.
Whenever we have a morphism $f \colon X \to Y$ of ind-pro-stacks, we will denote by $f^*_\PI$ the functor
\[f^*_\PI = \PIPerf(f) \colon \PIPerf(Y) \to \PIPerf(X)
\]
\end{df}

\begin{rmq}
Let $X$ be an ind-pro-stack. Let $\bar X \colon K \to \Prou U\dSt_S$ denote a $\mathbb U$-small filtered diagram of whom $X$ is a colimit in $\IP\dSt_S$.
We have by definition
\[
\PIPerf(X) \simeq \lim \Prou U(\IPerf(\bar X))
\]
\end{rmq}

\begin{prop}\label{prop-piperf-right-adjoint}
Let $f \colon X \to Y$ be a map of ind-pro-stacks. If $Y$ is a pro-stack then the functor $f^*_\PI \colon \PIPerf(Y) \to \PIPerf(X)$ admits a right adjoint.
\end{prop}

\begin{df}
Let $f \colon X \to Y$ be a map of ind-pro-stacks. If the functor
\[
f^*_\PI \colon \PIPerf(Y) \to \PIPerf(X)
\]
admits a right adjoint, we will denote it by $f^\PI_*$.
\end{df}

\begin{proof}[of the proposition]
If both $X$ and $Y$ are pro-stacks, then $f^\PI_* = \Prou U(f^\mathbf I_*)$ is right adjoint to $f_\PI^* = \Prou U(f_\mathbf I^*)$.
Let now $X$ be an ind-pro-stack and let $\bar X \colon K \to \Prou U\dSt_S$ denote a $\mathbb U$-small filtered diagram of whom $X$ is a colimit in $\IP\dSt_S$. We then have
\[
f_\PI^* \colon \PIPerf(Y) \to \PIPerf(X) \simeq \lim \PIPerf(\bar X) 
\]
The right adjoint is the informally given by the formula
\[
f^\PI_* = \lim_k \bar f(k)^\PI_*
\]
where $\bar f(k)$ is the induced map $\bar X(k) \to Y$. For a formal proof, we refer to \cite[1.2.0.5]{hennion:these}.
\end{proof}

\begin{df}
Let $X \in \IP\dSt_S$. We define $\IPPerf(X) = (\PIPerf(X))\op$.
If $X$ is the colimit in $\IP\dSt_S$ of a filtered diagram $K \to \Prou U\dSt_S$ then we have
\[
\IPPerf(X) \simeq \lim (\Indu U \circ \PPerf \circ \bar X)
\]
We will denote by $\dual{(-)} \colon \IPPerf(X) \to (\PIPerf(X))\op$ the duality functor.
\end{df}

\begin{df}
Let us define the functor $\Tateu U_\IP \colon (\IP\dSt_S)\op \to \inftyCat^{\mathbb V\mathrm{,st,id}}$ as the right Kan extension of $\Tateu U_\mathbf P$ along the inclusion $(\Prou U \dSt_S)\op \to (\IP\dSt_S)\op$.
It is by definition endowed with a canonical fully faithful natural transformation
\[
\Tateu U_\IP \to \PIPerf
\]
For any $X \in \IP\dSt_S$, an object of $\Tateu U_\IP(X)$ will be called a Tate module on $X$.
\end{df}

\begin{rmq}
We can characterise Tate objects: 
a module $E \in \PIPerf(X)$ is a Tate module if and only if for any pro-stack $U$ and any morphism $f \colon U \to X \in \IP\dSt_S$, the pullback $f^*_\IP(E)$ is in $\Tateu U_\mathbf P(U)$.

Let us also remark here that 
\end{rmq}

\begin{lem}\label{ipdst-tate-in-ipp}
Let $X$ be an ind-pro-stack over $S$. The fully faithful functors
\[
\mymatrix{
\Tateu U_\IP(X) \ar[r] & \PIPerf(X) \ar@{=}[r]^-{\dual{(-)}}& (\IPPerf(X))\op & \left(\Tateu U_\IP(X)\right)\op \ar[l]
}
\]
have the same essential image. We thus have an equivalence
\[
\dual{(-)} \colon \Tateu U_\IP(X) \simeq \left(\Tateu U_\IP(X)\right)\op
\]
\end{lem}
\begin{proof}
This is a corollary of \autoref{tate-dual}.
\end{proof}

\begin{df}
Let us define $\PIQcoh \colon (\IP\dSt_S)\op \to \inftyCatu V$ to be the functor
\[
\PIQcoh = \Proext_{(\Prou U\dSt_S)\op}(\IQcoh)
\]
From \autoref{indext-monoidal}, for any ind-pro-stack $X$, the category $\PIQcoh(X)$ admits a natural monoidal structure.
We also define the subfunctor
\[
\PIQcoh^{\leq 0} = \Proext_{(\Prou U\dSt_S)\op}(\IQcoh^{\leq 0})
\]
\end{df}

\begin{rmq}
Let us give an informal description of the above definition. To an ind-pro-stack $X = \colim_\alpha \lim_\beta X_{\alpha\beta}$ we associate the category
\[
\PIQcoh(X) = \lim_\alpha \Prou U \Indu U\left(\colim_\beta \Perf(X_{\alpha\beta})\right)
\]
\end{rmq}

\begin{df}
Let $f \colon X \to Y$ be a map of ind-pro-stacks. We will denote by $f_\PI^*$ the functor $\PIQcoh(f)$. Whenever it exists, we will denote by $f^\PIQ_*$ the right adjoint to $f_\PI^*$.
\end{df}

\begin{prop}\label{prop-piqcoh-right-adjoint}
Let $f \colon X \to Y$ be a map of ind-pro-stacks. If $Y$ is actually a stack, then the induced functor $f^*_\PI$ admits a right adjoint.
\end{prop}
\begin{proof}
This is very similar to the proof of \autoref{prop-piperf-right-adjoint} but using \autoref{prop-iqcoh}.
\end{proof}

\begin{rmq}
There is a fully faithful natural transformation $\PIPerf \to \PIQcoh$. Using the same notation $f_\PI^*$ for the images of a map $f \colon X \to Y$ is therefore only a small abuse.
Moreover, for any such map $f \colon X \to Y$, for which the right adjoints drawn below exist, there is a natural tranformation
\[
\mymatrix{
\PIPerf(Y) \ar[r] \ar[d]_{f^\PI_*} & \PIQcoh(Y) \ar[d]^{f^\PIQ_*} \\ \PIPerf(X) \ar[r] \ar@{=>}[ur] & \PIQcoh(X)
}
\]
It is generally not an equivalence.
\end{rmq}

\begin{df}
Let $\sqzext^{\IP}$ denote the natural transformation $\Proext_{(\Prou U\dSt_S)\op}(\sqzext^{\Pro})$
\[
\sqzext^{\IP} \colon \PIQcoh^{\leq 0} \to \Proext_{(\Prou U\dSt_S)\op} \left( \btw_{(\Prou U\dSt_S)\op}^{\amalg}(\id_-) \right) \simeq \btw_{(\IP\dSt_S)\op}^{\amalg}(\id_-)
\]
of functors $(\IP\dSt_S)\op \to \inftyCat$.
If $X$ is an ind-pro-stack and $E \in \PIQcoh(X)^{\leq 0}$ then we will denote by $X \to X[E] \to X$ the image of $E$ by the functor 
\[
\sqzext^{\IP}(X) \colon \PIQcoh(X)^{\leq 0} \to \left(\comma{X}{\IP\dSt_X}\right)\op
\]
\end{df}

\begin{rmq}
Let us decipher the above definition. Let $X = \colim_\alpha \lim_\beta X_{\alpha\beta}$ be an ind-pro-stack and let $E$ be a pro-ind-module over it. By definition $E$ is the datum, for every $\alpha$, of a pro-ind-object $E^\alpha$ in the category $\Prou U \Indu U (\colim_\beta \Qcoh^{\leq 0}(X_{\alpha\beta}))$.
Let us denote $E^\alpha = \lim_\gamma \colim_\delta E^\alpha_{\gamma\delta}$.
For any $\gamma$ and $\delta$, there is a $\beta_0(\gamma,\delta)$ such that $E^\alpha_{\gamma\delta}$ is in the essential image of $\Qcoh^{\leq 0} (X_{\alpha\beta_0(\gamma,\delta)})$.
We then have
\[
X[E] = \colim_{\alpha,\gamma} \lim_{\delta} \lim_{\beta \geq \beta_0(\gamma,\delta)} X_{\alpha\beta}[E_{\gamma\delta}] \in \IP\dSt_S
\]
\end{rmq}

\begin{df}\label{derivation-ipdst}
Let $X$ be an ind-pro-stack.
\begin{itemize}
\item We define the functor of derivations on $X$ 
\[
\Der(X,-) = \Map_{X/-/S}(X[-],X)
\]
\item We say that $X$ admits a cotangent complex if there exists $\Lcot_{X/S} \in \PIQcoh(X)$ such that for any $E \in \PIQcoh(X)^{\leq 0}$
\[
\Der(X,E) \simeq \Map(\Lcot_{X/S},E)
\]
\item Let us assume that $f \colon X \to Y$ is a map of ind-pro-stacks and that $Y$ admits a cotangent complex. We say that $f$ is formally étale if $X$ admits a cotangent complex and the natural map $f^* \Lcot_{Y/S} \to \Lcot_{X/S}$ is an equivalence.
\end{itemize}
\end{df}

\begin{df}
An Artin ind-pro-stack over $S$ is an object in the category
\[
\IP\dStArt_S = \Indu U \Prou U\dStArt_S
\]
An Artin ind-pro-stack locally of finite presentation is an object of
\[
\IP\dStArtlfp_S = \Indu U \Prou U\dStArtlfp_S
\]
\end{df}

\begin{prop}\label{ipcotangent}
Any Artin ind-pro-stack $X$ admits a cotangent complex
\[
\Lcot_{X/S} \in \PIQcoh(X)
\]
Let us assume that $\bar X \colon K \to \Pro\dStArt_S$ is a $\mathbb U$-small filtered diagram of whom $X$ is a colimit in $\IP\dStArt_S$.
For any vertex $k \in K$ we will denote by $X_k$ the pro-stack $\bar X(k)$ and by $i_k$ the structural map $X_k \to X$.
For any $f \colon k \to l$ in $K$, let us also denote by $f$ the induced map $X_k \to X_l$.
We have for all $k \in K$
\[
i_{k,\PI}^* \Lcot_{X/S} \simeq \lim_{f \colon k \to l} f^*_\mathbf I \Lcot_{X_l/S} \in \PIQcoh(X_k)
\]
If moreover $X$ is locally of finite presentation then $\Lcot_{X/S}$ belongs to $\PIPerf(X)$.
\end{prop}

\begin{proof}
Let us recall the natural transformation $\cotangent^{\Pro}$ from the proof of \autoref{cotangent-pdst}
\[
\cotangent^{\Pro} = \Indext_{\dSt_S\op}(\cotangent) \colon \overcat^{(\Prou U\dStArt_S)\op}_{(\Prou U\dSt_S)\op} \to \IQcoh(-)
\]
of functors $(\Prou U\dSt_S)\op \to \inftyCat$. Applying the functor $\Proext_{(\Prou U\dSt_S)\op}$ we define the natural transformation $\cotangent^{\IP}$
\[
\cotangent^{\IP} = \Proext_{(\Prou U\dSt_S)\op} \left( \cotangent^{\Pro} \right) \colon \overcat_{(\IP\dSt_S)\op}^{(\IP\dStArt_S)\op} \to \PIQcoh(-)
\]
between functors $(\IP\dSt_S)\op \to \inftyCat$.
Specifying to $X$ we get a functor
\[
\cotangent^{\IP}_X \colon \left(\comma{X}{\IP\dStArt_S}\right) \op \to \PIQcoh(X)
\]
We now define $\Lcot_{X/S} = \cotangent^{\IP}_X(X)$.
By definition we have
\[
i_{k,\PI}^* \Lcot_{X/S} \simeq \lim \cotangent^{\Pro}_{X_k}(\bar X) \simeq \lim_{f \colon k \to l} f^*_\mathbf I \Lcot_{X_l/S}
\]
for every $k \in K$.
Let us now prove that it satisfies the expected universal property. It suffices to compare for every $k \in K$ the functors
\[
\Map_{X_k/-/S}(X_k[-], X) \hspace{1cm} \text{and} \hspace{1cm} \Map_{\PIQcoh(X_k)}(i_{k,\PI}^*\Lcot_{X/S}, -)
\]
defined on $\PIQcoh(X_k)^{\leq 0}$. They are both pro-extensions to $\PIQcoh(X_k)^{\leq 0}$ of their restrictions $\IQcoh(X_k)^{\leq 0} \to \sSets$.
The restricted functor $\Map_{X_k/-/S}(X_k[-], X)$ is a colimit of the diagram
\[
\Map_{X_k/-/S}(X_k[-], \bar X) \colon \left(\comma{k}{K}\right)\op \to \Fct(\IQcoh(X_k)^{\leq 0}, \sSets)
\]
while $\Map_{\PIQcoh(X_k)}(i_{k,\PI}^*\Lcot_{X/S}, -)$ is a colimit to the diagram
\[
\Map_{\IQcoh(X_k)}(\cotangent^{\Pro}_{X_k}(\bar X), -) \colon \left(\comma{k}{K}\right)\op \to \Fct(\IQcoh(X_k)^{\leq 0},\sSets)
\]
We finish the proof with \autoref{cotangent-prostacks}.
\end{proof}

Let us record here a technical result we will use later on. For a proof, we refer to \cite[2.1.2.20]{hennion:these}.
\begin{prop}\label{IPcotangent-underlying}
Let $X \in \IP\dStArt_S$.
Let us denote by $\pi \colon X \to S$ the structural map.
Let also $\tilde \Lcot^\IP$ denote the functor
\[
\left(\IP\dStArt_S\right)\op \to \Prou U \Indu U \Qcoh(S)
\]
obtained by extending the functor $(\dStArt_S)\op \to \Qcoh(S)$ mapping $f \colon T \to S$ to $f_* \Lcot_{T/S}$.
Then we have $\pi_*^\PIQ \Lcot_{X/S} \simeq \tilde \Lcot^{\IP} (X)$
\end{prop}

\begin{df}
Let $X$ by an Artin ind-pro-stack locally of finite presentation over $S$. We will call the tangent complex of $X$ the ind-pro-perfect complex on $X$
\[
\T_{X/S} = \dual \Lcot_{X/S} \in \IPPerf(X)
\]
\end{df}

\subsection{Uniqueness of pro-structure}

\begin{lem}\label{coconnective}
Let $Y$ and $Z$ be derived Artin stacks. The following is true
\begin{enumerate}
\item The canonical map 
\[
\Map(Z,Y) \to \lim_n \Map(\tau_{\leq n} Z,Y)
\]
is an equivalence;
\item If $Y$ is $q$-Artin and $Z$ is $m$-truncated then the mapping space $\Map(Z,Y)$ is $(m + q)$-truncated.
\end{enumerate}
\end{lem}

\begin{proof}
We prove both items recursively on the Artin degree of $Z$. 
The case of $Z$ affine is proved in \cite[C.0.10 and 2.2.4.6]{toen:hagii}. We assume that the result is true for $n$-Artin stacks. Let $Z$ be $(n+1)$-Artin. There is an atlas $u \colon U \to Z$.
Let us remark that for $k \in\N$ the truncation $\tau_{\leq k} u \colon \tau_{\leq k} U \to \tau_{\leq k} Z$ is also a smooth atlas --- indeed we have $\tau_{\leq k} U \simeq U \times_Z \tau_{\leq k} Z$. Let us denote by $U_\bullet$ the nerve of $u$ and by $\tau_{\leq k}U_\bullet$ the nerve of $\tau_{\leq k} u$. Because $k$-truncated stacks are stable by flat pullbacks, the groupoid $\tau_{\leq k}U_\bullet$ is equivalent to $\tau_{\leq k}(U_\bullet)$. We have
\[
\Map(Z,Y) \simeq \lim_{[p] \in \Delta} \Map(U_p,Y) \simeq \lim_{[p] \in \Delta} \lim_k \Map(\tau_{\leq k}U_p,Y) \simeq \lim_k \Map(\tau_{\leq k} Z,Y)
\]
That proves item \emph{(i)}.
If moreover $Z$ is $m$-truncated, then we can replace $U$ by $\tau_{\leq m} U$. If follows that $\Map(Z,Y)$ is a limit of $(m + q)$-truncated spaces. This finishes the proof of \emph{(ii)}.
\end{proof}

We will use this well known lemma:
\begin{lem}[see {\cite[Chapter XI]{bousfieldkan:holim}}]\label{finite-groupoids}
Let $S \colon \Delta \to \sSets$ be a cosimplicial object in simplicial sets. Let us assume that for any $[p] \in \Delta$ the simplicial set $S_p$ is $n$-coconnective. Then the natural morphism
\[
\lim_{[p] \in \Delta} S_p \to \underset{p \leq n+1}{\lim_{[p] \in \Delta}} S_p
\]
is an equivalence.
\end{lem}

\begin{lem}\label{truncatedcompact}
Let $\bar X \colon \N\op \to \dSt_S$ be a diagram such that
\begin{enumerate}
\item There are two integers $m$ and $n$ such that for any $k \in \N$ the stack $\bar X(k)$ is $n$-Artin, $m$-truncated and of finite presentation;
\item There exists a diagram $\bar u \colon \N\op \times \Delta^1 \to \dSt_S$ such that the restriction of $\bar u$ to $\N\op \times \{1\}$ is equivalent to $\bar X$, every map $\bar u(k) \colon \bar u(k)(0) \to \bar u(k)(1) \simeq \bar X(k)$ is a smooth atlas and the limit $\lim_k \bar u(k)$ is an epimorphism.
\label{truncatedcompact:atlas}
\end{enumerate}
If $Y$ is an algebraic derived stack of finite presentation then the canonical morphism
\[
\colim \Map\left( \bar X,Y\right) \to \Map\left(\lim \bar X, Y \right)
\]
is an equivalence.
\end{lem}

\begin{proof}
Let us prove the statement recursively on the Artin degree $n$. If $n$ equals $0$, this is a simple reformulation of the finite presentation of $Y$.
Let us assume that the statement at hand is true for some $n$ and let $\bar X(0)$ be $(n+1)$-Artin.
Considering the nerves of the epimorphisms $\bar u(k)$, we get a diagram
\[
\bar Z \colon \N\op \times \Delta\op \to \dSt_S
\]
Note that $\bar Z$ has values in $n$-Artin stacks.
We observe that the diagram $\lim_k \bar Z(k)_\bullet \colon \Delta\op \to \dSt_S$ is the nerve of the map $\lim_k \bar u(k)$. Since $\lim_k \bar u(k)$ is by assumption an epimorphism (whose target is $\lim \bar X$), the natural map
\[
\colim_{[p] \in \Delta} \lim_{k \in \N} \bar Z(k)_p \to \lim \bar X  \simeq \lim_{k \in \N} \colim_{[p] \in \Delta} \bar Z(k)_p
\]
is an equivalence.
We now write
\begin{align*}
\Map\left(\lim \bar X,Y \right)
&\simeq \Map\left(\colim_{[p] \in \Delta} \lim_{k \in \N} \bar Z(k)_p,Y \right) \\
&\simeq \lim_{[p] \in \Delta} \Map\left(\lim_{k \in \N} \bar Z(k)_p,Y \right) \\
&\simeq \lim_{[p] \in \Delta} \colim_{k \in \N} \Map\left(\bar Z(k)_p,Y \right)
\end{align*}
We also have
\[
\colim \Map\left( \bar X,Y \right) \simeq \colim_{k \in \N} \lim_{[p] \in \Delta} \Map\left( \bar Z(k)_p, Y \right)
\]
It thus suffices to prove that the canonical morphism of simplicial sets
\[
\colim_{k \in \N} \lim_{[p] \in \Delta} \Map\left( \bar Z(k)_p, Y \right) \to \lim_{[p] \in \Delta} \colim_{k \in \N} \Map\left(\bar Z(k)_p,Y \right)
\]
is an equivalence. Let us notice that each $\bar Z(k)_p$ is $m$-truncated. It is indeed a fibre product of $m$-truncated derived stacks along flat maps.
Let $q$ be an integer such that $Y$ is $q$-Artin. The simplicial set $\Map(\bar Z(k)_p,Y)$ is then $(m + q)$-coconnective (\autoref{coconnective}).
It follows from \autoref{finite-groupoids} that the limit at hand is in fact finite and we have the required equivalence.
\end{proof}

\begin{lem}\label{exact-seq}
Let $\bar M \colon \N\op \to \sSets$ be a diagram. For any $i \in \N$ and any point $x = (x_n) \in \lim \bar M$, we have the following exact sequence
\[
\mymatrix{
0 \ar[r] & \displaystyle {\lim_n}^1 \pi_{i+1}(\bar M(n),x_n) \ar[r] & \displaystyle \pi_i\left(\lim_n \bar M(n), x\right) \ar[r] & \displaystyle \lim_n \pi_i(\bar M(n), x_n) \ar[r] & 0
}
\]
\end{lem}
A proof of that lemma can be found for instance in \cite{hirschhorn:limfibrations}.

\begin{lem}\label{colim-lim-swap}
Let $M \colon \N\op \times K \to \sSets$ denote a diagram, where $K$ is a filtered simplicial set. If for any $i \in \N$ there exists $N_i$ such that for any $n \geq N_i$ and any $k \in K$ the induced morphism $M(n,k) \to M(n-1,k)$ is an $i$-equivalence then the canonical map
\[
\phi \colon \colim_{k \in K} \lim_{n \in \N} M(n,k) \to \lim_{n \in \N} \colim_{k \in K} M(n,k)
\]
is an equivalence. We recall that an $i$-equivalence of simplicial sets is a morphism which induces isomorphisms on the homotopy groups of dimension lower or equal to $i$.
\end{lem}

\begin{proof}
We can assume that $K$ admits an initial object $k_0$.
Let us write $M_{nk}$ instead of $M(n,k)$. Let us fix $i \in \N$. If $i \geq 1$, we also fix a base point $x \in \lim_n M_{nk_0}$. Every homotopy group below is computed at $x$ or at the natural point induced by $x$. We will omit the reference to the base point. We have a morphism of short exact sequences
\[
\mymatrix{
0 \ar[r] &
\displaystyle \colim_k {\lim_{n}}^1 \pi_{i+1}(M_{nk}) \ar[r] \ar[d] &
\displaystyle \colim_k \pi_{i}\left(\lim_{n} M_{nk}\right) \ar[r] \ar[d] &
\displaystyle \colim_k \lim_{n} \pi_i(M_{nk}) \ar[d] \ar[r] & 0 \\
0 \ar[r] &
\displaystyle {\lim_{n}}^1 \colim_k \pi_{i+1}(M_{nk}) \ar[r] &
\displaystyle \pi_{i}\left(\lim_{n} \colim_k M_{nk}\right) \ar[r] &
\displaystyle \lim_{n} \colim_k \pi_i(M_{nk}) \ar[r] & 0 \\
}
\]
We can restrict every limit to $n \geq N_{i+1}$. Using the assumption we see that the limits on the right hand side are then constant and so are the $1$-limits on the left.  If follows that the vertical maps on the sides are isomorphisms, and so is the middle map.
This begin true for any $i$, we conclude that $\phi$ is an equivalence.
\end{proof}

\begin{df}\label{shy}
Let $\bar X \colon \N\op \to \dSt_S$ be a diagram. We say that $\bar X$ is a shy diagram if
\begin{enumerate}
\item For any $k \in \N$ the stack $\bar X(k)$ is algebraic and of finite presentation;
\item For any $k \in \N$ the map $\bar X(k \to k+1) \colon \bar X(k+1) \to \bar X(k)$ is affine;
\item The stack $\bar X(0)$ is of finite cohomological dimension.
\end{enumerate}
If $X$ is the limit of $\bar X$ in the category of prostacks, we will also say that $\bar X$ is a shy diagram for $X$.
\end{df}

\begin{prop}\label{prop-cocompact}
Let $\bar X \colon \N\op \to \dSt_S$ be a shy diagram.
If $Y$ is an algebraic derived stack of finite presentation then the canonical morphism
\[
\colim \Map\left( \bar X,Y\right) \to \Map\left(\lim \bar X, Y \right)
\]
is an equivalence.
\end{prop}

\begin{proof}
Since for any $n$, the truncation functor $\tau_{\leq n}$ preserves shy diagrams, let us use \autoref{coconnective} and \autoref{truncatedcompact}
\begin{align*}
\Map(\lim \bar X,Y) \simeq & \lim_n \Map(\tau_{\leq n}(\lim \bar X),Y)\\ & \simeq \lim_n \Map(\lim \tau_{\leq n}\bar X,Y) \simeq \lim_n \colim \Map(\tau_{\leq n} \bar X,Y)
\end{align*}
On the other hand we have 
\[
\colim \Map(\bar X,Y) \simeq \colim \lim_n \Map(\tau_{\leq n} \bar X,Y)
\]
and we are to study the canonical map 
\[
\phi \colon \colim \lim_n \Map(\tau_{\leq n} \bar X,Y) \to \lim_n \colim \Map(\tau_{\leq n} \bar X,Y)
\]
Because of \autoref{colim-lim-swap}, it suffices to prove the assertion
\begin{enumerate}[label=(\arabic*)]
\item For any $i \in \N$ there exists $N_i \in \N$ such that for any $n \geq N_i$ and any $k \in \N$ the map
\[
p_{n,k} \colon \Map\left( \tau_{\leq n} \bar X(k), Y \right) \to \Map\left( \tau_{\leq n -1} \bar X(k), Y \right)
\]
induces an equivalence on the $\pi_j$'s for any $j \leq i$.
\end{enumerate}
For any map $f \colon \tau_{\leq n-1} \bar X(k) \to Y$ we will denote by $F_{n,k}(f)$ the fibre of $p_{n,k}$ at $f$. We have to prove that for any such $f$ the simplicial set $F_{n,k}(f)$ is $i$-connective.
Let thus $f$ be one of those maps. The derived stack $\tau_{\leq n} \bar X(k)$ is a square zero extension of $\tau_{\leq n-1} \bar X(k)$ by a module $M[n]$, where
\[
M = \ker \left(\Oo_{\tau_{\leq n} \bar X(k)} \to \Oo_{\tau_{\leq n-1} \bar X(k)} \right) [-n]
\]
Note that $M$ is concentrated in degree $0$. It follows from the obstruction theory of $Y$ --see \autoref{obstruction} -- that $F_{n,k}(f)$ is not empty if and only if the obstruction class 

\[
\alpha(f) \in G_{n,k}(f) = \Map_{\Oo_{\tau_{\leq n-1} \bar X(k)}}(f^* \Lcot_Y, M[n+1])
\]
of $f$ vanishes.
Moreover, if $\alpha(f)$ vanishes, then we have an equivalence
\[
F_{n,k}(f) \simeq \Map_{\Oo_{\tau_{\leq n-1} \bar X(k)}}(f^* \Lcot_Y, M[n])
\]
Using assumptions \emph{(iii)} and \emph{(ii)} we have that $\bar X(k)$ --- and therefore its truncation too --- is of finite cohomological dimension $d$.
Let us denote by $[a,b]$ the Tor-amplitude of $\Lcot_Y$.
We get that $G_{n,k}(f)$ is $(s+1)$-connective for $s = a + n - d$ and that $F_{n,k}(f)$ is $s$-connective if $\alpha(f)$ vanishes.

Let us remark here that $d$ and $a$ do not depend on either $k$ or $f$ and thus neither does $N_i = i + d - a$ (we set $N_i = 0$ if this quantity is negative). For any $n \geq N_i$ and any $f$ as above, the simplicial set $G_{n,k}(f)$ is at least $1$-connective. The obstruction class $\alpha(f)$ therefore vanishes and $F_{n,k}(f)$ is indeed $i$-connective. This proves (1) and concludes this proof.
\end{proof}

\newcommand{\shy}{\mathbf P\dSt^{\mathrm{shy}}}
\begin{df}\label{prochamps}
Let $\shy_S$ denote the full subcategory of $\Prou U\dSt_S$ spanned by the prostacks which admit shy diagrams.
Every object $X$ in $\shy_S$ is thus the limit of a shy diagram $\bar X \colon \N\op \to \dSt_S$.

We will say that $X$ is of cotangent tor-amplitude in $[a,b]$ if there exists a shy diagram $\bar X \colon \N\op \to \dSt_S$ for $X$ such that every cotangent $\Lcot_{\bar X(n)}$ is of tor-amplitude in $[a,b]$. We will also say that $X$ is of cohomological dimension at most $d$ if there is a shy diagram $\bar X$ with values in derived stacks of cohomological dimension at most $d$.
The pro-stack $X$ will be called $q$-Artin if there is a shy diagram for it, with values in $q$-Artin derived stacks.
Let us denote by $\Cc^{[a,b]}_{d,q}$ the full subcategory of $\shy_S$ spanned by objects of cotangent tor-amplitude in $[a,b]$, of cohomological dimension at most $d$ and $q$-Artin.
\end{df}

\begin{thm}
The limit functor $i_\mathrm{shy} \colon \shy_S \to \dSt_S$ is fully faithful and has values in Artin stacks.
\end{thm}

\begin{proof}
This follows directly from \autoref{prop-cocompact}.
\end{proof}

\begin{df}
A map of pro-stacks $f \colon X \to Y$ if an open immersion if there exists a diagram
\[
\bar f \colon \N\op \times \Delta^1 \to \dSt_k
\]
such that 
\begin{itemize}
\item The limit of $\bar f$ in maps of pro-stacks is $f$;
\item The restriction $\N\op \times \{0\} \to \dSt_k$ of $\bar f$ is a shy diagram for $X$ and the restriction $\N\op \times \{1\} \to \dSt_k$ is a shy diagram for $Y$;
\item For any $n$, the induced map of stacks $\{n\} \times \Delta^1 \to \dSt_k$ is an open immersion.
\end{itemize}
\end{df}

\subsection{Uniqueness of ind-pro-structures}
\label{unique-ipstructure}

\begin{df}\label{indprochamps}
Let $\shybounded_S$ denote the full subcategory of $\Indu U(\shy_S)$ spanned by colimits of $\mathbb U$-small filtered diagrams $K \to \shy_S$  which factors through $\Cc^{[a,b]}_{d,q}$ for some 4-uplet $a,b,d,q$.
For any $X \in \shybounded_S$ we will say that $X$ is of cotangent tor-amplitude in $[a,b]$ and of cohomological dimension at most $d$ if it is the colimit (in $\Indu U(\shy_S)$) of a diagram $K \to \Cc^{[a,b]}_{d,q}$.
\end{df}

\begin{thm}\label{ff-realisation}
The colimit functor $\Indu U(\shy_S) \to \dSt_S$ restricts to a full faithful embedding $\shybounded_S \to \dSt_S$.
\end{thm}

\begin{lem}\label{iequi}
Let $a,b,d,q$ be integers with $a \leq b$. Let $T \in \shy_S$ and $\bar X \colon K \to \Cc^{[a,b]}_{d,q}$ be a $\mathbb U$-small filtered diagram. For any $i \in \N$ there exists $N_i$ such that for any $n \geq N_i$ and any $k \in K$, the induced map
\[
\Map(\tau_{\leq n}T, \bar X(k)) \to \Map(\tau_{\leq n-1} T, \bar X(k))
\]
is an $i$-equivalence. We recall that an $i$-equivalence of simplicial sets is a morphism which induces isomorphisms on the homotopy groups of dimension lower or equal to $i$.
\end{lem}

\begin{rmq}
For the proof of this lemma, we actually do not need the integer $q$.
\end{rmq}

\begin{proof}
Let us fix $i \in \N$. Let $k \in K$ and $\bar T \colon \N \to \dSt_S$ be a shy diagram for $T$.
We observe here that $\tau_{\leq n} \bar T$ is a shy diagram whose limit is $\tau_{\leq n} T$. Let also $\bar Y_k \colon \N \to \dSt_S$ be a shy diagram for $\bar X(k)$.
The map at hand 
\[
\psi_{nk} \colon \Map(\tau_{\leq n}T, \bar X(k)) \to \Map(\tau_{\leq n-1} T, \bar X(k))
\]
is then the limit of the colimits 
\[
\lim_{p \in \N} \colim_{q \in \N} \Map(\tau_{\leq n} \bar T(q), \bar Y_k(p)) \to 
\lim_{p \in \N} \colim_{q \in \N} \Map(\tau_{\leq n-1} \bar T(q), \bar Y_k(p))
\]
Let now $f$ be a map $\tau_{\leq n-1} T \to \bar X(k)$. It corresponds to a family of morphisms
\[
f_p \colon \pt \to \colim_{q \in \N} \Map(\tau_{\leq n-1} \bar T(q), \bar Y_k(p))
\]
For any $p$, let $F_{nk}^p(f)$ denote the fibre of the map
\[
\psi_{nk}^p \colon \colim_{q \in \N} \Map(\tau_{\leq n} \bar T(q), \bar Y_k(p)) \to 
\colim_{q \in \N} \Map(\tau_{\leq n-1} \bar T(q), \bar Y_k(p))
\]
over the point $f_p$.
We also set $F_{nk}(f) = \lim_p F_{nk}^p(f)$ and observe that $F_{nk}(f)$ is nothing but the fibre of $\psi_{nk}$ over $f$.

To prove the result, it suffices to show that for any such $f$, the fibre $F_{nk}(f)$ is $i$-connective. 
Using the exact sequence of \autoref{exact-seq}, it suffices to prove that $F_{nk}^p(f)$ is $(i+1)$-connective for any $f$ and any $p$.

Fixing such an $f$ and such a $p$, there exists $q_0 \in \N$ such that the map $f_p$ factors through the canonical map
\[
\Map(\tau_{\leq n-1} \bar T(q_0), \bar Y_k(p)) \to \colim_{q \in \N} \Map(\tau_{\leq n-1} \bar T(q), \bar Y_k(p))
\]
We deduce that $F_{nk}^p(f)$ is equivalent to the colimit
\[
F_{nk}^p(f) \simeq \colim_{q \geq q_0} G_{nk}^{pq}(f)
\]
where $G_{nk}^{pq}(f)$ is the fibre at the point induced by $f_p$ of the map
\[
\Map(\tau_{\leq n} \bar T(q), \bar Y_k(p)) \to \Map(\tau_{\leq n-1} \bar T(q), \bar Y_k(p))
\]
The interval $[a,b]$ contains the tor-amplitude of $\Lcot_{\bar Y_k(p)}$ and $d$ is an integer greater than the cohomological dimension of $\bar T(q)$. We saw in the proof of \autoref{prop-cocompact} that $G_{nk}^{pq}(f)$ is then $(a + n - d)$-connective.
We set $N_i = i + d - a +1$.
\end{proof}

\begin{proof}[of \autoref{ff-realisation}]
We will prove the sufficient following assertions
\begin{enumerate}[label=(\arabic*)]
\item The colimit functor $\Indu U(\shy_S) \to \presh(\dAff_S)$ restricts to a fully faithful functor
\[
\eta \colon \shybounded_S \to \presh(\dAff_S)
\]
\item The functor $\eta$ has values in the full subcategory of stacks.
\end{enumerate}
Let us focus on assertion (1) first. We consider two $\mathbb U$-small filtered diagrams $\bar X \colon K \to \shy_S$ and $\bar Y \colon L \to \shy_S$.
We have
\[
\Map_{\Indu U(\shy_S)}\left(\colim \bar X, \colim \bar Y\right) \simeq \lim_k \Map_{\Indu U(\shy_S)}(\bar X(k), \colim \bar Y)
\]
and
\[
\Map_{\presh(\dAff)}\left( \colim i_\mathrm{shy} \bar X, \colim i_\mathrm{shy} \bar Y \right) \simeq \lim_k \Map_{\presh(\dAff)}\left( i_\mathrm{shy} \bar X(k), \colim i_\mathrm{shy} \bar Y \right)
\]
We can thus replace the diagram $\bar X$ in $\shy_S$ by a simple object $X \in \shy_S$. We now assume that $\bar Y$ factors through $\Cc^{[a,b]}_{d,q}$ for some $a,b,d,q$. We have to prove that the following canonical morphism is an equivalence
\[
\phi \colon \colim_{l \in L} \Map(i_\mathrm{shy} X,i_\mathrm{shy} \bar Y(l)) \to \Map\left(i_\mathrm{shy} X, \colim i_\mathrm{shy} \bar Y\right)
\]
where the mapping spaces are computed in prestacks.
If $i_\mathrm{shy} X$ is affine then $\phi$ is an equivalence because colimits in $\presh(\dAff_S)$ are computed pointwise. Let us assume that $\phi$ is an equivalence whenever $i_\mathrm{shy} X$ is $(q-1)$-Artin and let us assume that $i_\mathrm{shy} X$ is $q$-Artin.
Let $u \colon U \to i_\mathrm{shy} X$ be an atlas of $i_\mathrm{shy} X$ and let $Z_\bullet$ be the nerve of $u$ in $\dSt_S$. We saw in the proof of \autoref{truncatedcompact} that $Z_\bullet$ factors through $\shy_S$.
The map $\phi$ is now equivalent to the natural map
\begin{align*}
\colim_{l \in L} \Map(i_\mathrm{shy} X,i_\mathrm{shy} \bar Y(l)) \to \lim_{[p]\in \Delta} & \colim_{l \in L} \Map(Z_p,i_\mathrm{shy} \bar Y(l)) \\ & \simeq \lim_{[p]\in \Delta} \Map\left(Z_p,\colim i_\mathrm{shy} \bar Y\right) \simeq \Map(i_\mathrm{shy} X,\colim i_\mathrm{shy} \bar Y)
\end{align*}
Remembering \autoref{coconnective}, it suffices to study the map
\[
\colim_{l \in L} \lim_n \Map(\tau_{\leq n} i_\mathrm{shy} X,i_\mathrm{shy} \bar Y(l)) \to \lim_{[p]\in \Delta} \colim_{l \in L} \lim_n \Map(\tau_{\leq n}Z_p,i_\mathrm{shy} \bar Y(l))
\]
Applying \autoref{iequi} and then \autoref{colim-lim-swap}, we see that $\phi$ is an equivalence if the natural morphism
\[
\lim_n \colim_{l \in L} \lim_{[p] \in \Delta} \Map(\tau_{\leq n}Z_p,i_\mathrm{shy} \bar Y(l)) \to \lim_n \lim_{[p] \in \Delta} \colim_{l \in L} \Map(\tau_{\leq n}Z_p, i_\mathrm{shy} \bar Y(l))
\]
is an equivalence. The stack $i_\mathrm{shy} \bar Y(l)$ is by assumption $q$-Artin, where $q$ does not depend on $l$. Now using \autoref{coconnective} and \autoref{finite-groupoids}, we conclude that $\phi$ is an equivalence. This proves (1). We now focus on assertion (2).
If suffices to see that the colimit in $\presh(\dAff_S)$ of the diagram $i_\mathrm{shy} \bar Y$ as above is actually a stack. Let $H_\bullet \colon \Delta\op \cup \{-1\} \to \dAff_S$ be an hypercovering of an affine $\Spec(A) = H_{-1}$. We have to prove the following equivalence
\[
\colim_l \lim_{[p] \in \Delta} \Map(H_p,i_\mathrm{shy} \bar Y(l)) \to \lim_{[p] \in \Delta} \colim_l \Map(H_p,i_\mathrm{shy} \bar Y(l))
\]
Using the same arguments as for the proof of (1), we have
\begin{align*}
\colim_l \lim_{[p] \in \Delta} \Map(H_p,i_\mathrm{shy} \bar Y(l))
&\simeq \colim_l \lim_{[p] \in \Delta} \lim_n \Map(\tau_{\leq n}H_p,i_\mathrm{shy} \bar Y(l)) \\
&\simeq \lim_n \colim_l \lim_{[p] \in \Delta} \Map(\tau_{\leq n}H_p,i_\mathrm{shy} \bar Y(l)) \\
&\simeq \lim_n \lim_{[p] \in \Delta} \colim_l \Map(\tau_{\leq n}H_p,i_\mathrm{shy} \bar Y(l)) \\
&\simeq \lim_{[p] \in \Delta} \colim_l \lim_n \Map(\tau_{\leq n}H_p,i_\mathrm{shy} \bar Y(l)) \\
&\simeq \lim_{[p] \in \Delta} \colim_l \Map(H_p,i_\mathrm{shy} \bar Y(l))
\end{align*}
\end{proof}

We will need one last lemma about that category $\shybounded_S$.
\begin{lem}\label{shydaff-limits}
The fully faithful functor $\shybounded_S \cap \IP\dAff_S \to \IP\dSt_S \to \dSt_S$ preserves finite limits.
\end{lem}

\begin{proof}
The case of an empty limit is obvious. Let then $X \to Y \from Z$ be a diagram in $\shybounded_S \cap \IP\dAff_S$. There exist $a$ and $b$ and a diagram 
\[
\sigma \colon K \to \Fct\left( \Lambda^2_1, \Cc^{[a,b]}_{0,0} \right)
\]
such that $K$ is a $\mathbb U$-small filtered simplicial set and the colimit in $\IP\dSt_S$ is $X \to Y \from Z$. We can moreover assume that $\sigma$ has values in $\Fct(\Lambda^2_1, \Prou U(\dAff_S)) \simeq \Prou U(\Fct(\Lambda^2_1, \dAff_S))$.
We deduce that the fibre product $X \times_Y Z$ is the realisation of the ind-pro-diagram in derived affine stacks with cotangent complex of tor amplitude in $[a-1,b+1]$. It follows that $X \times_Y Z$ is again in $\shybounded_S \cap \IP\dAff_S$. 
\end{proof}

\section{Symplectic Tate stacks}\label{chapterSymptate}%
\subsection{Tate stacks: definition and first properties}
\label{tatestacks}We can now define what a Tate stack is.
\begin{df}
A Tate stack is a derived Artin ind-pro-stack locally of finite presentation whose cotangent complex -- see \autoref{ipcotangent} -- is a Tate module.
Equivalently, an Artin ind-pro-stack locally of finite presentation is Tate if its tangent complex is a Tate module.
 We will denote by $\Tatestack_k$ the full subcategory of $\IP\dSt_k$ spanned by Tate stacks.
\end{df}

This notion has several good properties. For instance, using \autoref{ipdst-tate-in-ipp}, if a $X$ is a Tate stack then comparing its tangent $\T_X$ and its cotangent $\Lcot_X$ makes sense, in the category of Tate modules over $X$.
We will explore that path below, defining symplectic Tate stacks.

Another consequence of Tatity\footnote{or Tateness or Tatitude} is the existence of a determinantal anomaly as defined in \cite{kapranovvasserot:loop2}.
Let us consider the natural morphism of prestacks
\[
\theta \colon \Tateu U \to \mathrm K^{\Tate}
\]
where $\Tateu U$ denote the prestack $A \mapsto \Tateu U(\Perf(A))$ and $\mathrm K^{\Tate} \colon A \mapsto \K(\Tateu U(\Perf(A)))$ -- $\mathrm K$ denoting the connective $K$-theory functor.
From \cite[Section 5]{hennion:tate} we have a determinant
\[
\mathrm K^{\Tate} \to \mathrm K(\Gm,2)
\]
where $\mathrm K(\Gm,2)$ is the Eilenberg-Maclane classifying stack.
\begin{df}
We define the Tate determinantal map as the composite map
\[
\Tateu U \to \mathrm K(\Gm,2)
\]
To any derived stack $X$ with a Tate module $E$, we associate the determinantal anomaly $[\det_E] \in \homol^2(X,\Oo_X^{\times})$, image of $E$ by the morphism
\[
\Map(X,\Tateu U) \to \Map(X,\mathrm K(\Gm,2))
\]
\end{df}

Let now $X$ be an ind-pro-stack. Let also $R$ denote the realisation functor $\Prou U\dSt_k \to \dSt_k$ mapping a pro-stack to its limit in $\dSt_k$. Let finally $\bar X \colon K \to \Prou U\dSt_k$ denote a $\mathbb U$-small filtered diagram whose colimit in $\IP\dSt_k$ is $X$. We have a canonical functor
\[
F_X \colon \lim \Tateu U_\mathbf P(\bar X) \simeq \Tateu U_\IP(X) \to \lim \Tateu U(R \bar X)
\]
\begin{df}\label{determinantalanomaly}
Let $X$ be an ind-pro-stack and $E$ be a Tate module on $X$. Let $X'$ be the realisation of $X$ in $\Indu U\dSt_k$ and $X''$ be its image in $\dSt_k$.
We define the determinantal anomaly of $E$ the image of $F_X(E)$ by the map
\[
\mymatrix{
\Map_{\Indu U\dSt_k}(X',\Tateu U) \to \Map_{\Indu U\dSt_k}(X',\mathrm K(\Gm,2)) \simeq \Map_{\dSt_k}(X'',\mathrm K(\Gm,2))
}
\]
In particular if $X$ is a Tate stack, we will denote by $[\det_X] \in \homol^2(X'',\Oo_{X''}^{\times})$ the determinantal anomaly associated to its tangent $\T_X \in \Tateu U_\IP(X)$.
\end{df}
The author plans on studying more deeply this determinantal class in future work.
For now, let us conclude this section with following
\begin{lem}\label{tate-limits}
The inclusion $\Tatestack_k \to \IP\dSt_k$ preserves finite limits.
\end{lem}

\begin{proof}
Let us first notice that a finite limit of Artin ind-pro-stacks is again an Artin ind-pro-stack. Let now $X \to Y \from Z$ be a diagram of Tate stacks.
The fibre product
\[
\mymatrix{
X \times_Y Z \cart \ar[d]_{p_Z} \ar[r]^-{p_X} & X \ar[d]^g \\ Z \ar[r] & Y
}
\]
is an Artin ind-pro-stack. It thus suffices to test if its tangent $\T_{X \times_Y Z}$ is a Tate module. The following cartesian square concludes
\[
\mymatrix{
\T_{X \times_Y Z} \cart \ar[r] \ar[d] & p_X^* \T_X \ar[d] \\ p_Z^* \T_Z \ar[r] & p_X^* g^* \T_Y
}
\]
\end{proof}

\subsection{Shifted symplectic Tate stacks}

\newcommand{\deRham}{\operatorname{\mathbf{DR}}}
\newcommand{\NCw}{\operatorname{\mathrm{NC^w}}}
\newcommand{\globalcomplex}{\operatorname{C}}

We assume now that the basis $S$ is the spectrum of a ring $k$ of characteristic zero.
Recall from \cite{ptvv:dersymp} the stack in graded complexes $\deRham$ mapping a cdga over $k$ to its graded complex of forms.
It actually comes with a mixed structure induced by the de Rham differential.
The authors also defined there the stack in graded complexes $\NCw$ mapping a cdga to its graded complex of closed forms.
Those two stacks are linked by a morphism $\NCw \to \deRham$ forgetting the closure.

We will denote by $\forms p, \closedforms p \colon \cdga_k \to \dgMod_k$ the complexes of weight $p$ in $\deRham[-p]$ and $\NCw[-p]$ respectively.
The stack $\forms p$ will therefore map a cdga to its complexes of $p$-forms while $\closedforms p$ will map it to its closed $p$-forms.
For any cdga $A$, a cocycle of degree $n$ of $\forms p (A)$ is an $n$-shifted $p$-forms on $\Spec A$. 
The functors $\closedforms p$ and $\forms p$ extend to functors
\[
\closedforms p,~ \forms p \colon \dSt_k\op \to \dgMod_k
\]
\begin{df}
Let us denote by $\IPforms p$ and $\IPclosedforms p$ the extensions
\[
(\IP\dSt_k)\op \to \Prou U\Indu U\dgMod_k
\]
of $\forms p$ and $\closedforms p$, respectively.
They come with a natural projection $\IPclosedforms p \to \IPforms p$.

Let $X \in \IP\dSt_k$. An $n$-shifted (closed) $p$-form on $X$ is a morphism $k[-n] \to \IPforms p(X)$ (resp. $\IPclosedforms p(X)$).
For any closed form $\omega \colon k[-n] \to \IPclosedforms p(X)$, the induced map $k[-n] \to \IPclosedforms p(X) \to \IPforms p(X)$ is called the underlying form of $\omega$.
\end{df}

\begin{rmq}
In the above definition, we associate to any ind-pro-stack $X = \colim_\alpha \lim_\beta X_{\alpha\beta}$ its complex of forms
\[
\IPforms p(X) = \lim_\alpha \colim_\beta \forms p(X_{\alpha\beta}) \in \Prou U \Indu U \dgMod_k
\]
\end{rmq}

For any ind-pro-stack $X$, the derived category $\PIQcoh(X)$ is endowed with a canonical monoidal structure. In particular, one defines a symmetric product $E \mapsto \Sym^2_\PI(E)$ as well as an antisymmetric product
\[
E \wedge_\PI E = \Sym_\PI^2(E[-1])[2]
\]

\begin{thm}\label{prop-formsareforms}
Let $X$ be an Artin ind-pro-stack over $k$ and $\pi \colon X \to \pt$ the projection.
The push-forward functor
\[
\pi^\PIQ_* \colon \PIQcoh(X) \to \Prou V \Indu V(\dgMod_k)
\]
exists (see \autoref{prop-piqcoh-right-adjoint}) and maps $\Lcot_X \wedge_\PI \Lcot_X$ to $\IPforms 2(X)$.
In particular, any $2$-form $k[-n] \to \IPforms 2(X)$ corresponds to a morphism $\Oo_X[-n] \to \Lcot_X \wedge_\PI \Lcot_X$ in $\PIQcoh(X)$.
\end{thm}
\begin{proof}
This follows from \cite[1.14]{ptvv:dersymp}, from \autoref{IPcotangent-underlying} and from the equivalence
\[
\cotangent^\IP \wedge_\PI \cotangent^\IP = \Proextu U \Indextu U (\lambda) \wedge_\PI \Proextu U \Indextu U(\lambda) \simeq \Proextu U \Indextu U (\cotangent \wedge \cotangent)
\]
where $\cotangent^\IP$ is defined in the proof of \autoref{ipcotangent}.
\end{proof}

\begin{df}
Let $X$ be a Tate stack.
Let $\omega \colon k[-n] \to \IPforms 2(X)$ be an $n$-shifted $2$-form on $X$.
It induces a map in the category of Tate modules on $X$
\[
\underline \omega \colon \T_X \to \Lcot_X[n]
\]
We say that $\omega$ is non-degenerate if the map $\underline \omega$ is an equivalence.
A closed $2$-form is non-degenerate if the underlying form is.
\end{df}

\begin{df}\index{Symplectic Tate stack}
A symplectic form on a Tate stack is a non-degenerate closed $2$-form. A symplectic Tate stack is a Tate stack equipped with a symplectic form.
\end{df}

\subsection{Mapping stacks admit closed forms}

\newcommand{\EV}{\operatorname{EV}}

In this section, we will extend the proof from \cite{ptvv:dersymp} to ind-pro-stacks. Note that if $X$ is a pro-ind-stack and $Y$ is a stack, then $\Mapstack(X,Y)$ is an ind-pro-stack.
We will then need an evaluation functor $\Mapstack(X,Y) \times X \to Y$.
It appears that this evaluation map only lives in the category of ind-pro-ind-pro-stacks
\[
\colim_\alpha \lim_{\beta} \colim_{\xi} \lim_{\zeta} \Mapstack(X_{\alpha\zeta},Y) \times X_{\beta\xi} \to Y
\]

To use this map properly, we will need the following remark.
\begin{df}
Let $\Cc$ be a category.
There is one natural fully faithful functor
\[
\phi \colon \PI(\Cc) \to (\IP)^2(\Cc)
\]
but three $\IP(\Cc) \to (\IP)^2(\Cc)$. The first one is given by applying $\IP$ to the canonical embedding functor $\Cc \to \IP(\Cc)$. The second one by considering the canonical embedding functor $\Dd \to \IP(\Dd)$ for $\Dd = \IP(\Cc)$. In this work, we will only consider the third functor
\[
\psi \colon \IP(\Cc) \to (\IP)^2(\Cc)
\]
given by applying $\Indu U$ to the canonical embedding $\Dd \to \PI(\Dd)$ for $\Dd = \Prou U(\Cc)$.
Let us also denote by $\xi$ the natural fully faithful functor $\Cc \to (\IP)^2 (\Cc)$.
\end{df}

\begin{df}
Let $Y$ be a stack and $X$ be a pro-ind-stack.
Let us denote the evaluation map in $\IP^2\dSt_S$
\[
\ev^{X,Y} \colon \mymatrix@1{\displaystyle \psi \Mapstack_S(X,Y) \times_S \phi X  \ar[r] & \xi Y}
\]
For a formal definition of this map, we refer to \cite[2.2.3.2]{hennion:these}.
\end{df}

We assume now that $S = \Spec k$.
Let us recall the following definition from \cite[2.1]{ptvv:dersymp}
\begin{df}
A derived stack $X$ is $\Oo$-compact if for any derived affine scheme $T$ the following conditions hold
\begin{itemize}
\item The quasi-coherent sheaf $\Oo_{X \times T}$ is compact in $\Qcoh(X \times T)$ ;
\item Pushing forward along the projection $X \times T \to T$ preserves perfect complexes.
\end{itemize}
Let us denote by $\dSt_k^{\Oo}$ the full subcategory of $\dSt_k$ spanned by $\Oo$-compact derived stacks.
\end{df}

\begin{df}
An $\Oo$-compact pro-ind-stack is a pro-ind-object in the category of $\Oo$-compact derived stacks.
We will denote by $\PI\dSt^{\Oo}_k$ their category.
\end{df}

\begin{lem}
There is a functor
\[
\PI \dSt_k^\Oo \to \Fct\left( \IP\dSt_k \times \Delta^1 \times \Delta^1, (\IP)^2 (\dgMod_k)\op \right)
\]
defining for any $\Oo$-compact pro-ind-stack $X$ and any ind-pro-stack $F$ a commutative square
\[
\mymatrix{
\closedforms p_{\IP^2} (\psi F \times \phi X) \ar[r] \ar[d] & \IPclosedforms p(\psi F) \otimes_k \phi \Oo_X \ar[d] \\
\forms p_{\IP^2} (\psi F \times \phi X) \ar[r] & \IPforms p(\psi F) \otimes_k \phi \Oo_X
}
\]
where $\closedforms p_{\IP^2}$ and $\forms p_{\IP^2}$ are the extensions of $\IPclosedforms p$ and $\IPforms p$ to
\[
(\IP)^2\dSt_k \to (\IP)^2 (\dgMod_k\op)
\]
\end{lem}

\begin{proof}
Recall in \cite[part 2.1]{ptvv:dersymp} the construction for any $\Oo$-compact stack $X$ and any stack $F$ of a commutative diagram (of graded complexes):
\[
\mymatrix{
\NCw(F \times X) \ar[r] \ar[d] & \NCw(F) \otimes_k \pi_* \Oo_X \ar[d] \\
\deRham(F \times X) \ar[r] & \deRham(F) \otimes_k \pi_* \Oo_X
}
\]
where $\pi \colon X \to \pt$.
Taking the part of weight $p$ and shifting, we get
\[
\mymatrix{
\closedforms p(F \times X) \ar[r] \ar[d] & \closedforms p(F) \otimes_k \pi_* \Oo_X \ar[d] \\
\forms p(F \times X) \ar[r] & \forms p(F) \otimes_k \pi_* \Oo_X
}
\]
This construction is functorial in both $F$ and $X$ so it corresponds to a functor
\[
\dSt^\Oo_k \to \Fct(\dSt_k \times \Delta^1 \times \Delta^1, \dgMod_k\op)
\]
We can now form the functor
\begin{align*}
\PI\dSt^\Oo_k
&\to \PI \Fct\left(\Pro \dSt_k \times \Delta^1 \times \Delta^1, \Pro (\dgMod_k\op) \right) \\
&\to \Fct\left(\Pro \dSt_k \times \Delta^1 \times \Delta^1, \PI \Pro (\dgMod_k\op) \right) \\
&\to \Fct\left(\IP \dSt_k \times \Delta^1 \times \Delta^1, (\IP)^2 (\dgMod_k\op) \right) \\
\end{align*}
By construction, for any ind-pro-stack $F$ and any $\Oo$-compact pro-ind-stack, it induces the commutative diagram
\[
\mymatrix{
\closedforms p_{\IP^2}(\psi F \times \phi X) \ar[r] \ar[d] & \psi\IPclosedforms p(F) \otimes_k \phi\Oo_X \ar[d] \\
\forms p_{\IP^2}(\psi F \times \phi X) \ar[r] & \psi\IPforms p(F) \otimes_k \phi\Oo_X
}
\]
\end{proof}
\begin{rmq}
Let us remark that we can informally describe the horizontal maps using the maps from \cite{ptvv:dersymp}:
\begin{align*}
\Theta_{\IP^2}(\psi F \times \phi X) = \lim_\alpha & \colim_\beta \lim_\gamma \colim_\delta \Theta(F_{\alpha\delta} \times X_{\beta\gamma})\\
& \to \lim_\alpha \colim_\beta \lim_\gamma \colim_\delta \Theta(F_{\alpha\delta}) \otimes (\Oo_{X_{\beta\gamma}}) = \psi \Theta_{\IP}(F) \otimes \phi \Oo_X \\
\end{align*}
where $\Theta$ is either $\closedforms p$ or $\forms p$.
\end{rmq}

\begin{df}
Let $F$ be an ind-pro-stack and let $X$ be an $\Oo$-compact pro-ind-stack. Let $\eta \colon \Oo_X \to k[-d]$ be a map of ind-pro-$k$-modules. Let finally $\Theta$ be either $\closedforms p$ or $\forms p$. We define the integration map
\[
\int_\eta \colon \mymatrix@1{\Theta_{\IP^2}(\psi F \times \phi X) \ar[r] & \psi \Theta_{\IP}(F) \otimes \phi \Oo_X \ar[r]^-{\id \otimes \phi \eta} & \psi\Theta_{\IP}(F)[-d]}
\]
\end{df}

\begin{thm}\label{ipdst-form}
Let $Y$ be a derived stack and $\omega_Y$ be an $n$-shifted closed $2$-form on $Y$. Let $X$ be an $\Oo$-compact pro-ind-stack, let $\pi \colon X \to \pt$ be the projecton, and let $\eta \colon \pi_* \Oo_X \to k[-d]$ be a map. The mapping ind-pro-stack $\Mapstack(X,Y)$ admits an $(n-d)$-shifted closed $2$-form.
\end{thm}
\begin{proof}
Let us denote by $Z$ the mapping ind-pro-stack $\Map(X,Y)$.
We consider the diagram
\[
\mymatrix@1{
\chi k[-n] \ar[r]^-{\omega_Y} & \chi \closedforms 2 (Y) \ar[r]^-{\ev^*} & \closedforms 2_{\IP^2} (X \times Z) \ar[r]^-{\int_\eta} & \psi\closedforms 2_{\IP} (Z)[-d]
}
\]
where $\chi \colon \dgMod_k\op \to^\xi \IP(\dgMod_k\op) \to^\psi (\IP)^2(\dgMod_k\op)$ is the canonical inclusion.
Note that since the functor $\psi$ is fully faithful, this induces a map in $\IP(\dgMod_k\op)$
\[
\mymatrix@1{
\xi k \ar[r] & \IPclosedforms 2 (Z)[n-d]}
\]
and therefore a $(n-d)$-shifted closed $2$-form on $Z = \Map(X,Y)$. The underlying form is given by the composition
\[
\mymatrix@1{
\chi k[-n] \ar[r]^-{\omega_Y} & \chi \closedforms 2 (Y) \ar[r] & \chi \forms 2 (Y) \ar[r]^-{\ev^*} & \forms 2_{\IP^2} (X \times Z) \ar[r]^-{\int_\eta} & \psi\forms 2_{\IP} (Z)[-d]
}
\]
\end{proof}
\begin{rmq} \label{describe-form}
Let us describe the form issued by \autoref{ipdst-form}.
We set the notations $X = \lim_\alpha \colim_\beta X_{\alpha\beta}$ and $Z_{\alpha\beta} = \Map(X_{\alpha\beta},Y)$. By assumption, we have a map
\[
\eta \colon \colim_\alpha \lim_\beta \Oo_{X_{\alpha\beta}} \to k[-d]
\]
For any $\alpha$, there exists therefore $\beta(\alpha)$ and a map $\eta_{\alpha\beta(\alpha)} \colon \Oo_{X_{\alpha\beta(\alpha)}} \to k[-d]$ in $\dgMod(k)$.
Unwinding the definitions, we see that the induced form $\int_\eta \omega_Y$
\[
\mymatrix@1{
\xi k \ar[r] & \IPforms 2 (\Mapstack(X,Y))[n-d]  \simeq \lim_\alpha \colim_\beta \forms 2 (Z_{\alpha\beta})[n-d]
}
\]
is the universal map obtained from the maps
\[
\mymatrix@1{
k \ar[r]^-{\omega_{\alpha\beta(\alpha)}} & \forms 2 (Z_{\alpha\beta(\alpha)})[n-d] \ar[r] & \colim_\beta \forms 2 (Z_{\alpha\beta})[n-d]
}
\]
where $\omega_{\alpha\beta(\alpha)}$ is built using $\eta_{\alpha\beta(\alpha)}$ and the procedure of \cite{ptvv:dersymp}. Note that $\omega_{\alpha\beta(\alpha)}$ can be seen as a map $\T_{X_{\alpha\beta(\alpha)}} \otimes \T_{X_{\alpha\beta(\alpha)}} \to \Oo_{X_{\alpha\beta(\alpha)}}$.
We also know from \autoref{prop-formsareforms} that the form $\int_\eta \omega_Y$ induces a map
\[
\T_Z \otimes \T_Z \to \Oo_Z[n-d]
\]
in $\IPP(Z)$. Let us fix $\alpha_0$ and pull back the map above to $Z_{\alpha_0}$. We get
\[
\colim_{\alpha \geq \alpha_0} \lim_\beta g_{\alpha_0\alpha}^* p_{\alpha\beta}^* ( \T_{Z_{\alpha\beta}} \otimes \T_{Z_{\alpha\beta}}) \simeq i_{\alpha_0}^* (\T_Z \otimes \T_Z) \to \Oo_{Z_{\alpha_0}}[n-d]
\]
This map is the universal map obtained from the maps
\begin{align*}
\lim_\beta g_{\alpha_0\alpha}^* p_{\alpha\beta}^* ( \T_{Z_{\alpha\beta}} \otimes \T_{Z_{\alpha\beta}}) \to{} & g_{\alpha_0\alpha}^* p_{\alpha\beta(\alpha)}^* ( \T_{Z_{\alpha\beta(\alpha)}} \otimes \T_{Z_{\alpha\beta(\alpha)}})
\\ &\to g_{\alpha_0\alpha}^* p_{\alpha\beta(\alpha)}^* (\Oo_{X_{\alpha\beta(\alpha)}})[n-d] \simeq \Oo_{X_{\alpha_0}}[n-d]
\end{align*}
where $g_{\alpha_0\alpha}$ is the structural map $Z_{\alpha_0} \to Z_\alpha$ and $p_{\alpha\beta}$ is the projection $Z_\alpha = \lim_\beta Z_{\alpha\beta} \to Z_{\alpha\beta}$.
\end{rmq}

\subsection{Mapping stacks have a Tate structure}
\newcommand{\coker}{\operatorname{coker}}
\begin{df}\label{map-cotate}
Let $S$ be an $\Oo$-compact pro-ind-stack. We say that $S$ is an $\Oo$-Tate stack if
there exist a poset $K$ and a diagram $\bar S \colon K\op \to \Indu U \dSt_k$ such that
\begin{enumerate}
\item The limit of $\bar S$ in $\PI\dSt_k$ is equivalent to $S$ ;\label{map-diagproj}
\item For any $i \leq j \in K$ the pro-module over $\bar S(i)$ \label{map-structuralring}
\[
\coker\left(\Oo_{\bar S(i)} \to \bar S(i \leq j)_* \Oo_{\bar S(j)} \right)
\]
is trivial in the pro-direction -- ie belong to $\Qcoh(\bar S(i))$.
\item For any $i \leq j \in K$ the induced map $\bar S(i \leq j)$ is represented by a diagram
\[
\bar f \colon L \times \Delta^1 \to \dSt_k
\]
such that
\begin{itemize}
\item For any $l \in L$ the projections $\bar f(l,0) \to \pt$ and $\bar f(l,1) \to \pt$ satisfy the base change formula ;
\item For any $l \in L$ the map $\bar f(l)$ satisfies the base change and projection formulae ;
\item For any $m \leq l \in L$ the induced map $\bar f(m \leq l, 0)$ satisfies the base change and projection formulae.
\end{itemize}
\end{enumerate}
\end{df}

\begin{rmq}
We will usually work with pro-ind-stacks $S$ given by an explicit diagram already satisfying those assumptions.
\end{rmq}

\begin{prop}\label{map-tate}
Let us assume that $Y$ is a derived Artin stack locally of finite presentation.
Let $S$ be an $\mathcal O$-compact pro-ind-stack.
If $S$ is an $\Oo$-Tate stack then the ind-pro-stack $\Mapstack(S,Y)$ is a Tate stack.
\end{prop}
\begin{proof}
Let $Z = \Map(S,Y)$ as an ind-pro-stack.
Let $\bar S \colon K\op \to \Indu U \dSt_k$ be as in \autoref{map-cotate}.
We will denote by $\bar Z \colon K \to \Prou U\dSt_k$ the induced diagram and for any $i \in K$ by $s_i \colon \bar Z(i) \to \bar Z$ the induced map.

Let us first remark that $Z$ is an Artin ind-pro-stack locally of finite presentation.
It suffices to prove that $s_i^* \Lcot_{Z}$ is a Tate module on $\bar Z(i)$, for any $i \in K$. Let us fix such an $i$ and denote by $Z_i$ the pro-stack $\bar Z(i)$.

We consider the differential map
\[
s_i^* \Lcot_{Z} \to \Lcot_{Z_i}
\]
It is by definition equivalent to the natural map
\[
\lim \cotangent^{\Pro}_{Z_i}(\bar Z|_{K^{\geq i}}) \to^f \cotangent^{\Pro}_{Z_i}(Z_i)
\]
where $K^{\geq i}$ is the comma category $\comma{i}{K}$ and $\bar Z|_{K^{\geq i}}$ is the induced diagram
\[
K^{\geq i} \to \comma{Z_i}{\Prou U\dSt_S}
\]
Let $\phi_i$ denote the diagram
\[
\phi_i \colon \left(K^{\geq i}\right)\op \to \IPerf(Z_i)
\]
obtained as the kernel of $f$. It is now enough to prove that $\phi_i$ factors through $\Perf(Z_i)$.

Let $j \geq i$ in $K$ and let us denote by $g_{ij}$ the induced map $Z_i \to Z_j$ of pro-stacks.
Let $\bar f \colon L \times \Delta^1 \to \dSt_k$ represents the map $\bar S(i \leq j) \colon \bar S(j) \to \bar S(i) \in \Indu U \dSt_k$ as in assumption \ref{map-diagproj} in \autoref{map-cotate}.
Up to a change of $L$ through a cofinal map, we can assume that the induced diagram
\[
\coker\left(\Oo_{\bar S(i)} \to \bar S(i \leq j)_* \Oo_{\bar S(j)}\right)
\]
is essentially constant -- see assumption \ref{map-structuralring}.
We denote by $\bar h \colon L\op \times \Delta^1 \to \dSt_k$ the induced diagram, so that $g_{ij}$ is the limit of $\bar h$ in $\Prou U \dSt_k$.
For any $l \in L$ we will denote by $h_l \colon Z_{il} \to Z_{jl}$ the map $\bar h(l)$.
Let us denote by $\bar Z_i$ the induced diagram $l \mapsto Z_{il}$ and by $\bar Z_j$ the diagram $l \mapsto Z_{jl}$.
Let also $p_l$ denote the projection $Z_i \to Z_{il}$

We have an exact sequence
\[
\phi_i(j) \to \colim_l p_l^* h_l^* \Lcot_{Z_{jl}} \to \colim_l p_l^* \Lcot_{Z_{il}}
\]
Let us denote by $\psi_{ij}$ the diagram obtained as the kernel
\[
\psi_{ij} \to \cotangent^{\Pro}_{Z_i}(\bar Z_j) \to \cotangent^{\Pro}_{Z_i}(\bar Z_i)
\]
so that $\phi_i(j)$ is the colimit $\colim \psi_{ij}$ in $\IPerf(Z_i)$.
It suffices to prove that the diagram $\psi_{ij} \colon L \to \Perf(Z_i)$ is essentially constant (up to a cofinal change of posets).
By definition, we have
\[
\psi_{ij}(l) \simeq p_l^* \Lcot_{Z_{il}/Z_{jl}} [-1]
\]
Let $m \to l$ be a map in $L$ and $t$ the induced map $Z_{il} \to Z_{im}$. The map $\psi_{ij}(m \to l)$ is equivalent to the map $p_l^* \xi$ where $\xi$ fits in the fibre sequence in $\Perf(Z_{il})$
\[
\mymatrix{
t^* \Lcot_{Z_{im}/Z_{jm}} [-1] \ar[r] \ar[d]_\xi & t^* h_m^* \Lcot_{Z_{jm}} \ar[d] \ar[r] & t^* \Lcot_{Z_{im}} \ar[d] \\
\Lcot_{Z_{il}/Z_{jl}} [-1] \ar[r] & h_l^* \Lcot_{Z_{jl}} \ar[r] & \Lcot_{Z_{il}}
}
\]
We consider the dual diagram
\[
\mymatrix{
t^* \T_{Z_{im}/Z_{jm}} [1] \ar@{<-}[r] \ar@{<-}[d] & t^* h_m^* \T_{Z_{jm}} \ar@{<-}[d] \ar@{<-}[r] & t^* \T_{Z_{im}} \ar@{<-}[d] \\
\T_{Z_{il}/Z_{jl}} [1] \ar@{<-}[r] & h_l^* \T_{Z_{jl}} \ar@{<-}[r] & \T_{Z_{il}} \ar@{}[ul]|{(\sigma)}
}
\]
Using base change along the maps from $S_{im}$, $S_{jm}$ and $S_{jl}$ to the point, we get that the square $(\sigma)$ is equivalent to
\[
\mymatrix{
\pi_* (\id \times s f_m)_* (\id \times s f_m)^* E  & \ar[l] \pi_* (\id \times s)_* (\id \times s)^* E \\
\pi_* (\id \times f_l)_* (\id \times f_l)^* E \ar[u] & \pi_* E \ar[l] \ar[u]
}
\]
where $\pi \colon Z_{il} \times S_{il} \to Z_{il}$ is the projection, where $s \colon S_{im} \to S_{il}$ is the map induced by $m \to l$ and where $E \simeq \ev^* \T_Y$ with $\ev \colon Z_{il} \times S_{il} \to Y$ the evaluation map.
Note that we use here the well known fact $\T_{\Map(X,Y)} \simeq \pr_* \ev^* \T_Y$ where
\[
\mymatrix{
\Map(X,Y) & \Map(X,Y) \times X \ar[r]^-\ev \ar[l]_\pr & Y
}
\]
are the canonical maps.

Now using the projection and base change formulae along the morphisms $s$, $f_l$ and $f_m$ we get that $(\sigma)$ is equivalent to the image by $\pi_*$ of the square
\[
\mymatrix{
E \otimes p^* s_* {f_m}_* \Oo_{S_{jm}} & E \otimes p^* s_* \Oo_{S_{im}} \ar[l] \\
E \otimes p^* {f_l}_* \Oo_{S_{jl}} \ar[u] & E \otimes p^* \Oo_{S_{il}} \ar[u] \ar[l]  
}
\]
We therefore focus on the diagram
\[
\mymatrix{
s_* {f_{m}}_* \Oo_{S_{jm}} & s_* \Oo_{S_{im}} \ar[l] \\
{f_l}_* \Oo_{S_{jl}} \ar[u] & \Oo_{S_{il}} \ar[l] \ar[u]
}
\]
The map induced between the cofibres is an equivalence, using assumption \ref{map-structuralring}.
It follows that the diagram $\psi_{ij}$ is essentially constant, and thus that $Z$ is a Tate stack.
\end{proof}

\section{Formal loops}\label{chapterfloops}
In this part, we will at last define and study the higher dimensional formal loop spaces. We will prove it admits a local Tate structure.

\subsection{Dehydrated algebras and de Rham stacks}
In this part, we define a refinement of the reduced algebra associated to a cdga. This allows us to define a well behaved de Rham stack associated to an infinite stack. Indeed, without any noetherian assumption, the nilradical of a ring -- the ideal of nilpotent elements -- is a priori not nilpotent itself.
The construction below gives an alternative definition of the reduced algebra -- which we call the dehydrated algebra -- associated to any cdga $A$, so that $A$ is, in some sense, a nilpotent extension of its dehydrated algebra.
Whenever $A$ is finitely presented, this construction coincides with the usual reduced algebra. 
\begin{df}
Let $A \in \cdga_k$.
We define its dehydrated algebra as the ind-algebra $A_\mathrm{deh} = \colim_{I} \quot{\homol^0(A)}{I}$ where the colimit is taken over the filtered poset of nilpotent ideals of $\homol^0(A)$. The case $I = 0$ gives a canonical map $A \to A_\mathrm{deh}$ in ind-cdga's.
This construction is functorial in $A$.
\end{df}

\begin{rmq}
Whenever $A$ is of finite presentation, then $A_\mathrm{deh}$ is equivalent to the reduced algebra associated to $A$. In that case, the nilradical $\sqrt{A}$ of $A$ is nilpotent.
Moreover, if $A$ is any cdga, it is a filtered colimits of cdga's $A_\alpha$ of finite presentation. We then have $A_\mathrm{deh} \simeq \colim (A_\alpha)_\mathrm{red}$ in ind-algebras.
\end{rmq}

\begin{lem}
The colimit $B$ of the ind-algebra $A_\mathrm{deh}$ in the category of algebras is equivalent to the reduced algebra $A_\mathrm{red}$.
\end{lem}

\begin{proof}
Let us first remark that $B$ is reduced. Indeed any nilpotent element $x$ of $B$ comes from a nilpotent element of $A$. It therefore belongs to a nilpotent ideal $(x)$.
This define a natural map of algebras $A_\mathrm{red} \to B$. To see that it is an isomorphism, it suffices to say that $\sqrt{A}$ is the union of all nilpotent ideals.
\end{proof}

\begin{df}
Let $X$ be a prestack. We define its de Rham prestack $X_\mathrm{dR}$ as the composition
\[
\mymatrix{
\cdga_k \ar[r]^-{(-)_\mathrm{deh}} & \Indu U(\cdga_k) \ar[r]^-{\Indu U(X)} & \Indu U(\sSets) \ar[r]^-{\colim} & \sSets
}
\]
This defines an endofunctor of $(\infty,1)$-category $\presh(\dAff_k)$.
We have by definition
\[
X_\mathrm{dR}(A) = \colim_{I} X\left( \quot{\homol^0(A)}{I} \right)
\]
\end{df}

\begin{rmq}
If $X$ is a stack of finite presentation, then it is determined by the images of the cdga's of finite presentation. The prestack $X_\mathrm{dR}$ is then the left Kan extension of the functor
\[
\app{\cdgaunbounded_k^{\leq 0\mathrm{,fp}}}{\sSets}{A}{X(A_\mathrm{red})}
\]
\end{rmq}

\begin{df}
Let $f \colon X \to Y$ be a functor of prestacks. We define the formal completion $\hat X_Y$ of $X$ in $Y$ as the fibre product
\[
\mymatrix{
\hat X_Y \cart \ar[r] \ar[d] & X_\mathrm{dR} \ar[d] \\ Y \ar[r] & Y_\mathrm{dR}
}
\]
This construction obviously defines a functor $\mathrm{FC} \colon \presh(\dAff_k)^{\Delta^1} \to \presh(\dAff_k)$.
\end{df}

\begin{rmq}
The natural map $\hat X_Y \to Y$ is formally étale, in the sense that for any $A \in \cdga_k$ and any nilpotent ideal $I \subset \homol^0(A)$ the morphism
\[
\hat X_Y (A) \to \hat X_Y \left(\textstyle \quot{\homol^0(A)}{I} \right) \timesunder[Y\left(\quot{\homol^0(A)}{I} \right)][][-4pt] Y(A)
\]
is an equivalence.
\end{rmq}

\subsection{Higher dimensional formal loop spaces}
Here we finally define the higher dimensional formal loop spaces.
To any cdga $A$ we associate the formal completion $V_A^d$ of $0$ in $\A^d_A$. We see it as a derived affine scheme whose ring of functions $A[\![X_{1\dots d}]\!]$ is the algebra of formal series in $d$ variables $\el{X}{d}$. Let us denote by $U_A^d$ the open subscheme of $V_A^d$ complementary of the point $0$.
We then consider the functors $\dSt_k \times \cdga_k \to \sSets$
\begin{align*}
& \kaplooppre^d_V \colon (X,A) \mapsto \Map_{\dSt_k}(V_A^d, X) \\
& \kaplooppre^d_U \colon (X,A) \mapsto \Map_{\dSt_k}(U_A^d, X)
\end{align*}

\begin{df}
Let us consider the functors $\kaplooppre_U^d$ and $\kaplooppre_V^d$ as functors $\dSt_k \to \presh(\dAff)$. They come with a natural morphism $\kaplooppre_V^d \to \kaplooppre_U^d$.
We define $\kaplooppre^d$ to be the pointwise formal completion of $\kaplooppre_V^d$ into $\kaplooppre_U^d$ :
\[
\kaplooppre^d(X) = \mathrm{FC}\left(\kaplooppre^d_V(X) \to \kaplooppre^d_U(X)\right)
\]
We also define $\kaploop^d$, $\kaploop^d_U$ and $\kaploop^d_V$ as the stackified version of $\kaplooppre^d$, $\kaplooppre^d_U$ and $\kaplooppre^d_V$ respectively.
We will call $\kaploop^d(X)$ the formal loop stack in $X$.
\end{df}

\begin{rmq}
The stack $\kaploop^d_V(X)$ is a higher dimensional analogue to the stack of germs in $X$, as studied for instance by Denef and Loeser in \cite{denefloeser:germs}.
\end{rmq}

\begin{rmq}
By definition, the derived scheme $U_A^d$ is the (finite) colimit in derived stacks
\[
U_A^d = \colim_q \colim_{\el{i}{q}} \Spec\left( A[\![X_{1\dots d}]\!][X^{-1}_{i_1\dots i_q}] \right)
\]
where $A[\![X_{1\dots d}]\!][X^{-1}_{i_1\dots i_q}]$ denote the algebra of formal series localized at the generators $\iel{X^{-1}}{q}$.
It follows that the space of $A$-points of $\kaploop^d(X)$ is equivalent to the simplicial set
\[
\kaploop^d(X)(A) \simeq \colim _{I\subset \homol^0(A)} \lim_q \lim_{\el{i}{q}} \Map\left( \Spec\left(A[\![X_{1\dots d}]\!][X^{-1}_{i_1\dots i_q}]^{\sqrt{I}}\right), X \right)
\]
where $A[\![X_{1 \dots d}]\!][X^{-1}_{i_1\dots i_q}]^{\sqrt{I}}$ is the sub-cdga of $A[\![X_{1 \dots d}]\!][X^{-1}_{i_1\dots i_q}]$ consisting of series
\[
\sum_{\el{n}{d}} a_{\el{n}{d}} X_1^{n_1} \dots X_d^{n_d}
\]
where $a_{\el{n}{d}}$ is in the kernel of the map $A \to \quot{\homol^0(A)}{I}$ as soon as at least one of the $n_i$'s is negative. Recall that in the colimit above, the symbol $I$ denotes a nilpotent ideal of $\homol^0(A)$.
\end{rmq}

\begin{lem}\label{LV-epi}
Let $X$ be a derived Artin stack of finite presentation with algebraisable diagonal (see \autoref{alg-diag}) and let $t \colon T = \Spec(A) \to X$ be a smooth atlas. The induced map $\kaploop_V^d(T) \to \kaploop_V^d(X)$ is an epimorphism of stacks.
\end{lem}

\begin{proof}
It suffices to study the map $\kaplooppre_V^d(T) \to \kaplooppre_V^d(X)$.
Let $B$ be a cdga.
Let us consider a $B$-point $x \colon \Spec B \to \kaplooppre_V^d(X)$. It induces a $B$-point of $X$
\[
\Spec B \to \Spec(B[\![X_{1\dots d}]\!]) \to^x X
\]
Because $t$ is an epimorphism, there exists an étale map $f \colon \Spec C \to \Spec B$ and a commutative diagram
\[
\mymatrix{
\Spec C \ar[r]^-c \ar[d]_f & T \ar[d]^t \\ \Spec B \ar[r] & X
}
\]
It corresponds to a $C$-point of $\Spec B \times_X T$.
For any $n \in \N$, let us denote by $C_n$ the cdga 
\[
C_n := \quot{C[\el{x}{d}]}{(\el{x^n}{d})}
\]
and by $S_n$ the spectrum $\Spec C_n$
We also set $B_n = \quot{B[\el{x}{d}]}{(\el{x^n}{d})}$ and $X_n = \Spec B_n$. Finally, we define $T_n$ as the pullback $T \times_X X_n$.
We will also consider the natural fully faithful functor $\Delta^n \simeq \{0,\dots,n\} \to \N$.
We have a natural diagram
\[
\alpha_0 \colon \Lambda^{2,2} \times \N \amalg_{\Lambda^{2,2} \times \Delta^0} \Delta^2 \times \Delta^0 \to \dSt_k
\]
informally drown has a commutative diagram
\[
\mymatrix@R=0pt{
S_0 \ar[d] \ar@/_20pt/[dd] \ar[r] & \dots \ar[r] & S_n \ar[r] \ar[d] & \dots \\
X_0 \ar[r] & \dots \ar[r] & X_n \ar[r] & \dots \\
T_0 \ar[u] \ar[r] & \dots \ar[r] & T_n \ar[r] \ar[u] & \dots
}
\]
Let $n \in \N$ and let us assume we have built a diagram
\[
\alpha_n \colon (\Lambda^{2,2} \times \N) \amalg_{\Lambda^{2,2} \times \Delta^n}  \Delta^2 \times \Delta^n \to \dSt_k
\]
extending $\alpha_{n-1}$.
There is a sub-diagram of $\alpha_n$
\[
\mymatrix{
S_n \ar[r] \ar[d] & S_{n+1} \\ T_n \ar[r] & T_{n+1} \ar[d]^{t_{n+1}} \\ & X_{n+1}
}
\]
Since the map $t_{n+1}$ is smooth (it is a pullback of $t$), we can complete this diagram with a map $S_{n+1} \to T_{n+1}$ and a commutative square. Using the composition in $\dSt_k$, we get a diagram $\alpha_{n+1}$ extending $\alpha_n$.
We get recursively a diagram $\alpha \colon \Delta^2 \times \N \to \dSt_k$. Taking the colimit along $\N$, we get a commutative diagram
\[
\mymatrix{
\Spec C \ar[d]_f \ar[r] & \colim_n \Spec C_n \ar[d] \ar[rr] && T \ar[d]^t \\
\Spec B \ar[r] & \colim_n \Spec B_n \ar[r] & \Spec(B[\![X_{1\dots d}]\!]) \ar[r] & X
}
\]
This defines a map $\phi \colon \colim \Spec(C_n) \to \Spec(B[\![X_{1\dots d}]\!]) \times_X T$.
We have the cartesian diagram
\[
\mymatrix{
\Spec(B[\![X_{1\dots d}]\!]) \times_X T \ar[r] \ar[d] \cart & X \ar[d] \\ \Spec(B[\![X_{1 \dots d}]\!]) \times T \ar[r] & X \times X
}
\]
The diagonal of $X$ is algebraisable and thus so is the stack $\Spec(B[\![X_{1\dots d}]\!]) \times_X T$. The morphism $\phi$ therefore defines the required map
\[
\Spec(C[\![X_{1 \dots d}]\!]) \to \Spec(B[\![X_{1\dots d}]\!]) \times_X T
\]
\end{proof}

\begin{rmq}
Let us remark here that if $X$ is an algebraisable stack, then $\kaplooppre_V^d(X)$ is a stack, hence the natural map is an equivalence
\[
\kaplooppre_V^d(X) \simeq \kaploop_V^d(X)
\]
\end{rmq}

\begin{lem}\label{LU-fetale}
Let $f \colon X \to Y$ be an étale map of derived Artin stacks. 
For any cdga $A \in \cdga_k$ and any nilpotent ideal $I \subset \homol^0(A)$, the induced map
\[
\theta \colon \mymatrix{
\kaplooppre^d_U(X)(A) \ar[r] & \kaplooppre^d_U(X)\left(\quot{\homol^0(A)}{I}\right) \displaystyle \timesunder[\kaplooppre^d_U(Y)\left(\quot{\homol^0(A)}{I}\right)][][-4pt] \kaplooppre^d_U(Y)(A)
}
\]
is an equivalence.
\end{lem}

\begin{proof}
The map $\theta$ is a finite limit of maps
\[
\mu \colon \mymatrix{
X(\xi A) \ar[r] &  X\left(\xi \left(\quot{\homol^0(A)}{I}\right) \right) \displaystyle \timesunder[Y\left(\xi\left(\quot{\homol^0(A)}{I}\right) \right)][][-4pt] Y (\xi A )
}
\]
where $\xi A = A[\![X_{1 \dots d}]\!][X^{-1}_{i_1\dots i_p}]$ is obtained from the cdga of formal power series in $A$ with $d$ variables by inverting the variables $x_{i_j}$. Let also $\xi (\quot{\homol^0(A)}{I})$ be defined similarly.
The natural map $\xi(\homol^0(A)) \to \xi (\quot{\homol^0(A)}{I})$ is also a nilpotent extension. We deduce from the étaleness of $f$ that the map
\[
\mymatrix{
X(\xi(\homol^0(A))) \ar[r] & X\left(\xi \left(\quot{\homol^0(A)}{I}\right) \right) \displaystyle \timesunder[Y\left(\xi\left(\quot{\homol^0(A)}{I}\right) \right)][][-4pt] Y (\xi( \homol^0(A)) )
}
\]
is an equivalence.
Let now $n \in \N$. We assume that the natural map
\[
\mymatrix{
X(\xi(A_{\leq n})) \ar[r] & X\left(\xi \left(\quot{\homol^0(A)}{I}\right) \right) \displaystyle \timesunder[Y\left(\xi\left(\quot{\homol^0(A)}{I}\right) \right)][][-4pt] Y (\xi(A_{\leq n}))
}
\]
is an equivalence. The cdga $\xi(A_{\leq n+1}) \simeq (\xi A)_{\leq n+1}$ is a square zero extension of $\xi(A_{\leq n})$ by $\homol^{-n-1}(\xi A)$. We thus have the equivalence
\[
\mymatrix{
X(\xi(A_{\leq n+1})) \ar[r]^-\sim & X(\xi(A_{\leq n})) \displaystyle \timesunder[Y(\xi(A_{\leq n}))][][-3pt] Y(\xi(A_{\leq n+1}))
}
\]
The natural map 
\[
\mymatrix{
X(\xi(A_{\leq n+1})) \ar[r] & X\left(\xi \left(\quot{\homol^0(A)}{I}\right) \right) \displaystyle \timesunder[Y\left(\xi\left(\quot{\homol^0(A)}{I}\right) \right)][][-4pt] Y (\xi(A_{\leq n+1}))
}
\]
is thus an equivalence too.
The stacks $X$ and $Y$ are nilcomplete, hence $\mu$ is also an equivalence -- recall that a derived stack $X$ is nilcomplete if for any cdga $B$ we have
\[
X(B) \simeq \lim_n X(B_{\leq n})
\]
Also recall that any Artin stack is nilcomplete.
It follows that $\theta$ is an equivalence.
\end{proof}

\begin{cor}\label{L-fetale}
Let $f \colon X \to Y$ be an étale map of derived Artin stacks.
For any cdga $A \in \cdga_k$ and any nilpotent ideal $I \subset \homol^0(A)$, the induced map
\[
\theta \colon
\mymatrix{
\kaplooppre^d(X)(A) \ar[r] & \kaplooppre^d(X)\left(\textstyle \quot{\homol^0(A)}{I}\right) \displaystyle \timesunder[\kaplooppre^d(Y)\left(\quot{\homol^0(A)}{I}\right)][][-4pt] \kaplooppre^d(Y)(A)
}
\]
is an equivalence.
\end{cor}

\begin{prop}\label{L-epi}
Let $X$ be a derived Deligne-Mumford stack of finite presentation with algebraisable diagonal. Let $t \colon T \to X$ be an étale atlas. The induced map $\kaploop^d(T) \to \kaploop^d(X)$ is an epimorphism of stacks.
\end{prop}

\begin{proof}
We can work on the map of prestacks $\kaplooppre^d(T) \to \kaplooppre^d(X)$.
Let $A \in \cdga_k$. Let $x$ be an $A$-point of $\kaplooppre^d(X)$. It corresponds to a vertex in the simplicial set
\[
\mymatrix{
\displaystyle \colim_I \textstyle {\kaplooppre_V^d(X)\left(\quot{\homol^0(A)}{I}\right)} \displaystyle \timesunder[\kaplooppre_U^d(X)\left(\quot{\homol^0(A)}{I}\right)][][-4pt] \kaplooppre_U^d(X)(A)
}
\]
There exists therefore a nilpotent ideal $I$ such that $x$ comes from a commutative diagram
\[
\mymatrix{
U_{\quot{\homol^0(A)}{I}}^d \ar[d] \ar[r] & U_A^d \ar[d] \\ V_{\quot{\homol^0(A)}{I}} \ar[r]_-v & X
}
\]
Using \autoref{LV-epi} we get an étale morphism $\psi \colon A \to B$ such that the map $v$ lifts to a map $u \colon V_{\quot{B}{J}} \to T$ where $J$ is the image of $I$ by $\psi$.
This defines a point in
\[
\textstyle \kaplooppre^d_U(T)\left(\quot{\homol^0(B)}{J} \right) \displaystyle \timesunder[\kaplooppre^d_U(X)\left(\quot{\homol^0(B)}{J}\right)][][-4pt] \kaplooppre^d_U(X)(B)
\]
Because of \autoref{LU-fetale}, we get a point of $\kaplooppre^d(T)(B)$. We now observe that this point is compatible with $x$.
\end{proof}

In the case of dimension $d=1$, \autoref{LU-fetale} can be modified in the following way. Let $f \colon X \to Y$ be a smooth map of derived Artin stacks.
For any cdga $A \in \cdga_k$ and any nilpotent ideal $I \subset \homol^0(A)$, the induced map
\[
\theta \colon \mymatrix{
\kaplooppre^1_U(X)(A) \ar[r] & \kaplooppre^1_U(X)\left(\quot{\homol^0(A)}{I}\right) \displaystyle \timesunder[\kaplooppre^1_U(Y)\left(\quot{\homol^0(A)}{I}\right)][][-4pt] \kaplooppre^1_U(Y)(A)
}
\]
is essentially surjective.
The following proposition follows.

\begin{prop}
Let $X$ be an Artin derived stack of finite presentation and with algebraisable diagonal. Let $t \colon T \to X$ be a smooth atlas. The induced map $\kaploop^1(T) \to \kaploop^1(X)$ is an epimorphism of stacks.
\end{prop}

\begin{ex}
The proposition above implies for instance that $\kaploop^1(\B G) \simeq \B \kaploop^1(G)$ for any algebraic group $G$ -- where $\B G$ is the classifying stack of $G$-bundles.
\end{ex}

\subsection{Tate structure and determinantal anomaly}\label{determinantalclass}
We saw in \autoref{tatestacks} that to any Tate stack $X$, we can associate a determinantal anomaly. It a class in $\homol^2(X,\Oo_X^{\times})$.
We will prove in this subsection that the stack $\kaploop^d(X)$ is endowed with a structure of Tate stack as soon as $X$ is affine. We will moreover build a determinantal anomaly on $\kaploop^d(X)$ for any quasi-compact and separated scheme $X$.

\begin{lem}\label{L-shy}
For any $B \in \cdga_k$ of finite presentation, the functors 
\[
\kaplooppre^d_U(\Spec B), \kaplooppre^d(\Spec B) \colon \cdga_k \to \sSets
\]
are in the essential image of the fully faithful functor 
\[
\shybounded_k \cap \IP\dAff_k \to \IP\dSt_k \to \dSt_k \to \presh(\dAff)
\]
(see \autoref{indprochamps}).
It follows that $\kaplooppre^d_U(\Spec B) \simeq \kaploop^d_U(\Spec B)$ and $\kaplooppre^d(\Spec B) \simeq \kaploop^d(\Spec B)$.
\end{lem}
\begin{proof}
Let us first remark that $\Spec B$ is a retract of a \emph{finite} limit of copies of the affine line $\A^1$.
It follows that the functor $\kaplooppre_U^d(\Spec B)$ is, up to a retract, a finite limit of functors
\[
Z_E^d \colon A \mapsto \Map\left(k[Y], A[\![X_{1 \dots d}]\!][X^{-1}_{i_1\dots i_q}]\right)
\]
where $E = \{ \el{i}{q} \} \subset F = \{1,\dots,d\}$.
The functor $Z^d_E$ is the realisation of an affine ind-pro-scheme
\[ 
Z^d_E \simeq \colim_n \lim_p \Spec\left( k[ a_{\el{\alpha}{d}}, -n \delta_i \leq \alpha_i \leq p] \right)
\]
where $\delta_i = 1$ if $i \in E$ and $\delta_i = 0$ otherwise. The variable $a_{\el{\alpha}{d}}$ corresponds to the coefficient of $X_1^{\alpha_1} \dots X_d^{\alpha_d}$.
The functor $Z^d_E$ is thus in the category $\shybounded \cap \IP\dAff_k$. The result about $\kaplooppre^d_U(\Spec B)$ then follows from \autoref{shydaff-limits}.
The case of $\kaplooppre^d(\Spec B)$ is similar: we decompose it into a finite limit of functors
\[
G_E^d \colon A \mapsto \colim_{I \subset \homol^0(A)} \Map\left(k[Y], A[\![X_{1 \dots d}]\!][X^{-1}_{i_1\dots i_q}]^{\sqrt{I}}\right)
\]
where $I$ is a nilpotent ideal of $\homol^0(A)$.
We then observe that $G_E^d$ is the realisation of the ind-pro-scheme
\[ 
G^d_E \simeq \colim_{n,m} \lim_p \Spec\left( \quot{k[ a_{\el{\alpha}{d}}, -n \delta_i \leq \alpha_i \leq p]}{J} \right)
\]
where $J$ is the ideal generated by the symbols $a_{\el{\alpha}{d}}^m$ with at least one of the $\alpha_i$'s negative.
\end{proof}

\begin{rmq}
Let $n$ and $p$ be integers and let $k(E,n,p)$ denote the number of families $(\el{\alpha}{d})$ such that $-n \delta_i \leq \alpha_i \leq p$ for all $i$. We have 
\[
Z^d_E \simeq \colim_n \lim_p (\A^1)^{k(E,n,p)}
\] 
\end{rmq}

\begin{df}
From \autoref{L-shy}, we get a functor $\underline \kaploop^d \colon \dAff_k^\mathrm{fp} \to \IP\dSt_k$. It follows from \autoref{L-epi} that $\underline \kaploop^d$ is a costack in ind-pro-stacks. We thus define
\[
\underline \kaploop^d \colon \dSt_k^{\mathrm{lfp}} \to \IP\dSt_k
\]
to be its left Kan extension along the inclusion $\dAff_k^\mathrm{fp} \to \dSt_k^\mathrm{lfp}$ -- where $\dSt_k^\mathrm{lfp}$ is $(\infty,1)$-category of derived stacks locally of finite presentation.
This new functor $\underline \kaploop^d$ preserves small colimits by definition.
\end{df}

\begin{prop}\label{L-indpro}
There is a natural transformation $\theta$ from the composite functor
\[
\mymatrix{
\dSt_k^\mathrm{lfp} \ar[r]^-{\underline \kaploop^d} & \IP\dSt_k \ar[r]^{|-|^\IP} & \dSt_k
}
\]
to the functor $\kaploop^d$.
Moreover, the restriction of $\theta$ to derived Deligne-Mumford stacks of finite presentation with algebraisable diagonal is an equivalence.
\end{prop}
\begin{proof}
There is by definition a natural transformation
\[
\theta \colon | \underline \kaploop^d(-) |^\IP \to \kaploop^d(-)
\]
Moreover, the restriction of $\theta$ to affine derived scheme of finite presentation is an equivalence -- see \autoref{L-shy}. 
The fact that $\theta_X$ is an equivalence for any Deligne-Mumford stack $X$ follows from \autoref{L-epi}.
\end{proof}

\begin{lem}\label{subsets-colim}
Let $F$ be a non-empty finite set.
For any family $(M_D)$ of complexes over $k$ indexed by subsets $D$ of $F$, we have
\[
\colim_{\emptyset \neq E \subset F} \bigoplus_{\emptyset \neq D \subset E} M_D \simeq M_F[d-1]
\]
where $d$ is the cardinal of $F$ (the maps in the colimit diagram are the canonical projections).
\end{lem}

\begin{proof}
We can and do assume that $F$ is the finite set $\{1, \dots, d\}$ and we proceed recursively on $d$.
The case $d = 1$ is obvious.
Let now $d \geq 2$ and let us assume the statement is true for $F \smallsetminus \{d\}$. Let $(M_D)$ be a family as above. We have a cocartesian diagram
\[
\mymatrix{
\displaystyle \colim_{\{d\} \subsetneq E \subset F} \bigoplus_{\emptyset \neq D \subset E} M_D \ar[r] \ar[d]
& \displaystyle \colim_{ \emptyset \neq E \subset F\smallsetminus\{d\}} \bigoplus_{\emptyset \neq D \subset E} M_D \ar[d] \\
\displaystyle M_{\{d\}} \ar[r]
& \displaystyle \colim_{\emptyset \neq E \subset F} \bigoplus_{\emptyset \neq D \subset E} M_D \cocart
}
\]
We have by assumption
\[
\colim_{ \emptyset \neq E \subset F\smallsetminus\{d\}} \bigoplus_{\emptyset \neq D \subset E} M_D \simeq M_{F \smallsetminus \{d\}} [d-2]
\]
and
\begin{align*}
\colim_{\{d\} \subsetneq E \subset F} \bigoplus_{\emptyset \neq D \subset E} M_D
&\simeq M_{\{d\}} \oplus \left(\colim_{\{d\} \subsetneq E \subset F} \bigoplus_{\{d\} \subsetneq D \subset E} M_D \right) \oplus \left(\colim_{\{d\} \subsetneq E \subset F} \bigoplus_{\emptyset \neq D \subset E \smallsetminus \{d\}} M_D\right) \\
&\simeq M_{\{d\}} \oplus M_F[d-2] \oplus M_{F \smallsetminus \{d\}}[d-2]
\end{align*}
The result follows.
\end{proof}

\begin{lem}\label{LU-affine-tate}
For any $B \in \cdga_k$ of \emph{finite presentation}, the ind-pro-stack $\underline \kaploop^d_U(\Spec B)$ is a Tate stack.
\end{lem}
\begin{proof}
Let us first focus on the case of the affine line $\A^1$.
We have to prove that the cotangent complex $\Lcot_{\underline \kaploop^d_U(\A^1)}$ is a Tate module.
For any subset $D \subset F$ we define $M_D^{p,n}$ to be the free $k$-complex generated by the symbols
\[
\{a_{\el{\alpha}{d}}, -n \leq \alpha_i < 0 \text{ if } i \in D, 0 \leq \alpha_i \leq p \text{ otherwise}\}
\]
in degree $0$.
From the proof of \autoref{L-shy}, we have
\[
Z^d_E \simeq \colim_n \lim_p \Spec \left( k\left[ \textstyle \bigoplus_{D \subset E} M_D^{p,n} \right] \right) \text{~~~ and ~~~} \underline \kaploop^d_U(\A^1) \simeq \lim_{\emptyset \neq E \subset F} Z^d_E
\]
where $F = \{1 , \dots , d\}$.
If we denote by $\pi$ the projection $\underline \kaploop_U^d(\A^1) \to \Spec k$, we get
\[
\Lcot_{\underline \kaploop^d_U(\A^1)} \simeq \pi^* \left(\colim_{\emptyset \neq E \subset F} \lim_n \colim_p \bigoplus_{D \subset E} M_D^{p,n}\right)
\simeq \pi^* \left(\lim_n \colim_p \colim_{\emptyset \neq E \subset F} \bigoplus_{D \subset E} M_D^{p,n}\right)
\]
Using \autoref{subsets-colim} we have
\[
\Lcot_{\underline \kaploop^d_U(\A^1)} \simeq \pi^* \left(\lim_n \colim_p M_\emptyset^{p,n} \oplus M_F^{p,n}[d-1] \right)
\]
Moreover, we have $M_\emptyset^{p,n} \simeq M_\emptyset^{p,0}$ and $M_F^{p,n} \simeq M_F^{0,n}$. It follows that $\Lcot_{\underline \kaploop^d_U(\A^1)}$ is a Tate module on the ind-pro-stack $\underline \kaploop^d_U(\A^1)$.
The case of $\underline \kaploop^d_U(\Spec B)$ then follows from \autoref{shydaff-limits} and from \autoref{tate-limits}.
\end{proof}

\begin{lem}\label{LU-ip-fetale}
Let $B \to C$ be an étale map between cdga's of finite presentation. The induced map $f \colon \underline \kaploop^d_U(\Spec C) \to \underline \kaploop^d_U(\Spec B)$ is formally étale -- see \autoref{derivation-ipdst}.
\end{lem}

\begin{proof}
Let us denote $X = \Spec B$ and $Y = \Spec C$.
We have to prove that the induced map
\[
j \colon \Map_{\underline \kaploop^d_U(Y)/-}\left(\underline \kaploop^d_U(Y)[-], \underline \kaploop^d_U(Y)\right) \to \Map_{\underline \kaploop^d_U(Y)/-}\left(\underline \kaploop^d_U(Y)[-], \underline \kaploop_U^d(X)\right) 
\]
is an equivalence of functors $\PIQcoh( \underline \kaploop^d(Y) )^{\leq 0} \to \sSets$.
Since $\underline \kaploop^d_U(Y)$ is ind-pro-affine, we can restrict to the study of the morphism
\[
j_Z \colon \Map_{Z/-}\left(Z[-], \underline \kaploop^d_U(Y)\right) \to \Map_{Z/-}\left(Z[-], \underline \kaploop_U^d(X)\right) 
\]
of functors $\IQcoh(Z)^{\leq 0} \to \sSets$, for any pro-affine scheme $Z$ and any map $Z \to \underline \kaploop^d_U(Y)$.
Let us fix $E \in \IQcoh(Z)^{\leq 0}$. The pro-stack $Z[E]$ is in fact an affine pro-scheme.
Recall that both $\underline \kaploop^d_U(Y)$ and $\underline \kaploop^d_U(X)$ belong to $\shybounded_k$.
It follows from the proof of \autoref{ff-realisation} that the morphism $j_Z(E)$ is equivalent to
\[
|j_Z(E)| \colon \Map_{|Z|/-}\left( |Z[E]|, \kaploop^d_U(Y) \right) \to \Map_{|Z|/-}\left( |Z[E]|, \kaploop^d_U(X) \right) 
\]
where $|-|$ is the realisation functor and the mapping spaces are computed in $\dSt_k$.
It now suffices to see that $|Z[E]|$ is a trivial square zero extension of the derived affine scheme $|Z|$ and to use \autoref{LU-fetale}.
\end{proof}

\begin{prop}\label{L-affine-tate}
Let $\Spec B$ be a derived affine scheme of finite presentation. The ind-pro-stack $\underline \kaploop^d(\Spec B)$ admits a cotangent complex. This cotangent complex is moreover a Tate module.
For any étale map $B \to C$ the induced map $f \colon \underline \kaploop^d(\Spec C) \to \underline \kaploop^d(\Spec B)$ is formally étale -- see \autoref{derivation-ipdst}.
\end{prop}

\begin{proof}
Let us write $Y = \Spec B$.
Let us denote by $i \colon \underline \kaploop^d(Y) \to \underline \kaploop^d_U(Y)$ the natural map. 
We will prove that the map $i$ is formally étale, the result will then follow from \autoref{LU-affine-tate} and \autoref{LU-ip-fetale}.
To do so, we consider the natural map
\[
j \colon \Map_{\underline \kaploop^d(Y)/-}\left(\underline \kaploop^d(Y)[-], \underline \kaploop^d(Y)\right) \to \Map_{\underline \kaploop^d(Y)/-}\left(\underline \kaploop^d(Y)[-], \underline \kaploop_U^d(Y)\right) 
\]
of functors $\PIQcoh( \underline \kaploop^d(Y) )^{\leq 0} \to \sSets$.
To prove that $j$ is an equivalence, we can consider for every affine pro-scheme $X \to \underline \kaploop^d(Y)$ the morphism of functors $\IQcoh(X)^{\leq 0} \to \sSets$
\[
j_X \colon \Map_{X/-}\left( X[-], \underline \kaploop^d(Y) \right) \to \Map_{X/-}\left( X[-], \underline \kaploop^d_U(Y) \right) 
\]
Let us fix $E \in \IQcoh(X)^{\leq 0}$. The morphism $j_X(E)$ is equivalent to
\[
|j_X(E)| \colon \Map_{|X|/-}\left( |X[E]|, \kaploop^d(Y) \right) \to \Map_{|X|/-}\left( |X[E]|, \kaploop^d_U(Y) \right) 
\]
where the mapping space are computed in $\dSt_k$. The map $|j_X(E)|$ is a pullback of the map
\[
f \colon \Map_{|X|/-}\left( |X[E]|, \kaploop^d_V(Y)_\mathrm{dR} \right) \to \Map_{|X|/-}\left( |X[E]|, \kaploop^d_U(Y)_\mathrm{dR} \right) 
\]
It now suffices to see that $|X[E]|$ is a trivial square zero extension of the derived affine scheme $|X|$ and thus $f$ is an equivalence (both of its ends are actually contractible).
\end{proof}

Let us recall from \autoref{determinantalanomaly} the determinantal anomaly
\[
[\mathrm{Det}_{\underline \kaploop^d(\Spec A)}] \in \homol^2\left(\kaploop^d(\Spec A), \Oo_{\kaploop^d(\Spec A)}^{\times}\right)
\]
It is associated to the tangent $\T_{\underline \kaploop^d(\Spec A)} \in \Tateu U_\IP(\underline \kaploop^d(\Spec A))$ through the determinant map.
Using \autoref{L-affine-tate}, we see that this construction is functorial in $A$, and from \autoref{L-epi} we get that it satisfies étale descent.
Thus, for any quasi-compact and quasi-separated (derived) scheme (or Deligne-Mumford stack with algebraisable diagonal), we have a well-defined determinantal anomaly
\[
[\mathrm{Det}_{\underline \kaploop^d(X)}] \in \homol^2\left(\kaploop^d(X), \Oo_{\kaploop^d(X)}^{\times}\right)
\]

\begin{rmq}
It is known since \cite{kapranovvasserot:loop4} that in dimension $d=1$, if $[\mathrm{Det}_{\kaploop^1(X)}]$ vanishes, then there are essentially no non-trivial automorphisms of sheaves of chiral differential operators on $X$.
\end{rmq}

\section{Bubble spaces}\label{chapterBubbles}
In this section, we study the bubble space, an object closely related to the formal loop space. We will then prove the bubble space to admit a symplectic structure.

\subsection{Two lemmas}

In this subsection, we will develop two duality results we will need afterwards.

Let $A \in \cdga_k$ be a cdga over a field $k$. Let $(\el{f}{p})$ be points of $A^0$ whose images in $\homol^0(A)$ form a regular sequence.

Let us denote by $A_{n,k}$ the Kozsul complex associated to the regular sequence $(\el{f^n}{k})$ for $k \leq p$. We set $A_{n,0} = A$ and $A_n = A_{n,p}$ for any $n$.
If $k<p$, the multiplication by $f^n_{k+1}$ induces an endomorphism $\varphi^n_{k+1}$ of $A_{n,k}$. Recall that $A_{n,k+1}$ is isomorphic to the cone of $\varphi^n_{k+1}$:
\[
\mymatrix{
A_{n,k} \ar[r]^{\varphi^n_{k+1}} \ar[d] & A_{n,k} \ar[d] \\ 0 \ar[r] & A_{n,k+1} \cocart
}
\]
Let us now remark that for any couple $(n,k)$, the $A$-module $A_{n,k}$ is perfect.
\begin{lem} \label{dual-over-A}
Let $k \leq p$.
The $A$-linear dual $A_{n,k}^{\quot{\vee}{A}} = \RHomint_A(A_{n,k},A)$ of $A_{n,k}$ is equivalent to $A_{n,k}[-k]$;
\end{lem}
\begin{proof}
We will prove the statement recursively on the number $k$.
When $k = 0$, the result is trivial.
Let $k \geq 0$ and let us assume that $A_{n,k}^{\quot{\vee}{A}}$ is equivalent to $A_{n,k}[-k]$. Let us also assume that for any $a \in A$, the diagram induced by multiplication by $a$ commutes
\[
\mymatrix{
A_{n,k}^{\quot{\vee}{A}} \ar@{-}[r]^-\sim \ar[d]_{\dual a} & A_{n,k}[-k] \ar[d]^a \\
A_{n,k}^{\quot{\vee}{A}} \ar@{-}[r]^-\sim & A_{n,k}[-k]
}
\]
We obtain the following equivalence of exact sequences
\[
\mymatrix{
A_{n,k+1}[-k-1] \ar[r] \ar@{-}[d]^\sim & A_{n,k}[-k] \ar@{-}[d]^\sim \ar[r]^{\varphi^n_{k+1}} & A_{n,k}[-k] \ar@{-}[d]^\sim \\
A_{n,k+1}^{\quot{\vee}{A}} \ar[r] & A_{n,k}^{\quot{\vee}{A}} \ar[r]^{\dual{\left(\varphi^n_{k+1}\right)}} & A_{n,k}^{\quot{\vee}{A}}
}
\]
The statement about multiplication is straightforward.
\end{proof}

\begin{lem} \label{dual-over-k}
Let us assume $A$ is a formal series ring over $A_1$:
\[
A = A_1[\![\el{f}{p}]\!]
\]
It follows that for any $n$, the $A_1$-module $A_n$ is free of finite type and that there is map $r_n \colon A_n \to A_1$ mapping $\el{f^n}{p}[]$ to $1$ and any other generator to zero.
We deduce an equivalence 
\[
A_n \to^\sim A_n^{\quot{\vee}{A_1}} = \RHomint_{A_1}(A_n, A_1)
\]
given by the pairing
\[
\mymatrix{
A_n \otimes_{A_1} A_n \ar[r]^-{\times} & A_n \ar[r]^{r_n} & A_1
}
\]
\end{lem}
\begin{rmq}
Note that we can express the inverse $A_n^{\quot{\vee}{A_1}} \to A_n$ of the equivalence above: it maps a function $\alpha \colon A_n \to A_1$ to the serie
\[
\sum_{\underline i} \alpha(f^{\underline i}) f^{n-1-\underline i}
\]
where $\underline i$ varies through the uplets $(\el{i}{p})$ and where $f^{\underline i} = f_1^{i_1} \dots f_p^{i_p}$.
\end{rmq}

\subsection{Definition and properties}

We define here the bubble space, obtained from the formal loop space. We will prove in the next sections it admits a structure of symplectic Tate stack.

\begin{df}
The formal sphere of dimension $d$ is the pro-ind-stack
\[
\formalsphere^d = \lim_n \colim_{p \geq n} \Spec(A_p \oplus \Homint_A(A_n,A)) \simeq \lim_n \colim_{p \geq n} \Spec(A_p \oplus A_n[-d])
\]
where $A = k[\el{x}{d}]$ and $A_n = \quot{A}{(\el{x^n}{d})}$.
\end{df}

\begin{rmq}
The notation $\Spec(A_p \oplus A_n[-d])$ is slightly abusive. The cdga $A_p \oplus A_n[-d]$ is not concentrated in non positive degrees. In particular, the derived stack $\Spec(A_p \oplus A_n[-d])$ is not a derived affine scheme.
It behaves like one though, regarding its derived category:
\[
\Qcoh(\Spec(A_p \oplus A_n[-d])) \simeq \dgMod_{A_p \oplus A_n[-d]}
\]
\end{rmq}

Let us define the ind-pro-algebra 
\[
\Oo_{\formalsphere^d} = \colim_n \lim_{p \geq n} A_p \oplus A_n[-d]
\]
where $A_p \oplus A_n[-d]$ is the trivial square zero extension of $A_p$ by the module $A_n[-d]$.
For any $m \in \N$, let us denote by $\formalsphere^d_m$ the ind-stack
\[
\formalsphere^d_m = \colim_{p \geq m} \Spec(A_p \oplus A_m[-d])
\]

\begin{df}\label{dfbubble}
Let $T$ be a derived Artin stack. We define the $d$-bubble stack of $T$ as the mapping ind-pro-stack
\[
\bubblestack(T) = \Mapstack(\formalsphere^d, T) \colon \Spec B \mapsto \colim_n \lim_{p \geq n} T \left(B \otimes (A_p \oplus A_n[-d]) \right)
\]
Again, the cdga $A_p \oplus A_n[-d]$ is not concentrated in non positive degree. This notation is thus slightly abusive and by $T(B \otimes (A_p \oplus A_n[-d]))$ we mean
\[
\Map(\Spec(A_p \oplus A_n[-d]) \times \Spec B,T)
\]
We will denote by $\bar \bubblestack(T)$ the diagram $\N \to \Prou U \dSt_k$ of whom $\bubblestack(T)$ is a colimit in $\IP\dSt_k$.
Let us also denote by $\bubblestack_m(T)$ the mapping pro-stack
\[
\bubblestack_m(T) = \Map(\formalsphere^d_m, T) \colon \Spec B \mapsto \lim_{p \geq m} T \left(B \otimes (A_p \oplus A_m[-d]) \right)
\]
and $\bar \bubblestack_m(T) \colon \{ p \in \N | p \geq m\}\op \to \dSt_S$ the corresponding diagram.
In particular
\[
\bubblestack_0(T) = \Map(\formalsphere^d_0, T) \colon \Spec B \mapsto \lim_{p} T \left(B \otimes A_p\right)
\]
Those stacks come with natural maps
\[
\mymatrix{
\bubblestack_0(T) \ar[r]^-{s_0} & \bubblestack(T) \ar[r]^-r & \bubblestack_0(T)
}
\]
\[
\mymatrix{
\bubblestack_m(T) \ar[r]^-{s_m} & \bubblestack(T)
}
\]
\end{df}

\begin{prop}\label{B-and-L-IP}
If $T$ is an affine scheme of finite type, the bubble stack $\bubblestack(T)$ is the product in ind-pro-stacks
\[
\mymatrix{
\bubblestack(T) \ar[r] \ar[d] \cart[3] & {} \underline{\kaploop}_V^d(T) \ar[d] \\ {} \underline \kaploop_V^d(T) \ar[r] & {} \underline \kaploop_U^d(T)
}
\]
\end{prop}

\begin{proof}
There is a natural map $V_k^d \to \formalsphere^d$ induced by the morphism
\[
\colim_n \lim_{p \geq n} A_p \oplus A_n[-d] \to \lim_p A_p
\]
Because $T$ is algebraisable, it induces a map $\bubblestack(T) \to \underline \kaploop_V^d(T)$ and thus a diagonal morphism
\[
\delta \colon \bubblestack(T) \to \underline \kaploop_V^d(T) \timesunder[\underline \kaploop_U^d(T)] \underline \kaploop_V^d(T)
\]
We will prove that $\delta$ is an equivalence.
Note that because $T$ is a (retract of a) finite limit of copies of $\A^1$, we can restrict to the case $T = \A^1$.
Let us first compute the fibre product $Z = \underline \kaploop^d_V(\A^1) \times_{\underline \kaploop_U^d(\A^1)} \underline \kaploop_V^d(\A^1)$. It is the pullback of ind-pro-stacks
\[
\mymatrix{
Z \ar[r] \ar[d] \cart &
\displaystyle \lim_p \Spec\left( k[ a_{\el{\alpha}{d}}, 0 \leq \alpha_i \leq p] \right) \ar[d] \\
\displaystyle \lim_p \Spec\left( k[ a_{\el{\alpha}{d}}, 0 \leq \alpha_i \leq p] \right) \ar[r] &
\displaystyle \colim_n \lim_p \lim_{I \subset J} \Spec\left( k[ a_{\el{\alpha}{d}}, -n \delta_{i \in I} \leq \alpha_i \leq p] \right)
}
\]
where $J = \{1, \dots , d\}$ and $\delta_{i \in I} = 1$ if $i \in I$ and $0$ otherwise.
For any subset $K \subset J$ we define $M_K^{p,n}$ to be the free complex generated by the  symbols
\[
\{a_{\el{\alpha}{d}}, -n \leq \alpha_i < 0 \text{ if } i \in K, 0 \leq \alpha_i \leq p \text{ otherwise}\}
\]
We then have the cartesian diagram
\[
\mymatrix{
Z \ar[r] \ar[d] \cart & \lim_p \Spec\left( k[ M^{p,0}_\emptyset ] \right) \ar[d] \\
\lim_p \Spec\left( k[ M^{p,0}_\emptyset ] \right) \ar[r] &
\colim_n \lim_p \lim_{I \subset J} \Spec\left(k\left[ \bigoplus_{K \subset I} M_K^{p,n} \right] \right)
}
\]
Using \autoref{subsets-colim} we get 
\[
Z \simeq \colim_n \lim_p \Spec\left( k\left[M^{p,0}_\emptyset \oplus M^{0,n}_J[d]\right] \right)
\]
\end{proof}

\begin{rmq}
Let us consider the map $\lim_p A_p \to A_0 \simeq k$ mapping a formal serie to its coefficient of degree $0$. The $(\lim A_p)$-ind-module $\colim A_n[-d]$ is endowed with a natural map to $k[-d]$. This induces a morphism $\Oo_{\formalsphere^d} \to k \oplus k[-d]$ and hence a map $\mathrm S^d \to \formalsphere^d$, where $\mathrm S^d$ is the topological sphere of dimension $d$. We then have a rather natural morphism
\[
\bubblestack^d(X) \to \Mapstack(\mathrm S^d,X)
\]
\end{rmq}

\subsection{Its tangent is a Tate module}

We will prove in this subsection that the bubble stack is a Tate stack.
To do so, we could bluntly apply \autoref{map-tate} but we will give here a direct proof of that statement.
We will get another decomposition of its tangent complex that will be needed when proving $\bubblestack^d(T)$ is symplectic.

\begin{prop}\label{prop-formal-tate}
Let us assume that the Artin stack $T$ is locally of finite presentation.
For any $m \in \N$ we have an exact sequence
\[
\mymatrix{
s_m^*r^* \Lcot_{\bubblestack^d(T)_0} \ar[r] & s_m^*\Lcot_{\bubblestack^d(T)} \ar[r] & s_m^*\Lcot_{\bubblestack^d(T)/\bubblestack^d(T)_0}
}
\]
where the left hand side is an ind-perfect module and the right hand side is a pro-perfect module.

In particular, the middle term is a Tate module, and the ind-pro-stack $\bubblestack^d(T)$ is a Tate stack.
\end{prop}
\begin{proof}
Throughout this proof, we will write $\bubblestack$ instead of
$\bubblestack^d(T)$ and $\bubblestack_m$ instead of $\bubblestack^d(T)_m$ for any $m$.
Let us first remark that $\bubblestack$ is an Artin ind-pro-stack locally of finite presentation.
It suffices to prove that $s_m^* \Lcot_{\bubblestack}$ is a Tate module on $\bubblestack_m$, for any $m \in \N$. We will actually prove that it is an elementary Tate module.
We consider the map
\[
s_m^* r^* \Lcot_{\bubblestack_0} \to s_m^* \Lcot_{\bubblestack}
\]
It is by definition equivalent to the natural map
\[
\cotangent^{\Pro}_{\bubblestack_m}(\bubblestack_0) \to^f \lim \cotangent^{\Pro}_{\bubblestack_m}(\bar \bubblestack_{\geq m}(T))
\]
where $\bar \bubblestack_{\geq m}(T)$ is the restriction of $\bar \bubblestack(T)$ to $\{ n \geq m \} \subset \N$.
Let $\phi$ denote the diagram
\[
\phi \colon \{ n \in \N | n \geq m\}\op \to \IPerf(\bubblestack_m(T))
\]
obtained as the cokernel of $f$. It is now enough to prove that $\phi$ factors through $\Perf(\bubblestack_m(T))$. Let $n \geq m$ be an integer and let $g_{mn}$ denote the induced map $\bubblestack_m(T) \to \bubblestack_n(T)$. We have  an exact sequence
\[
s_m^* r^* \Lcot_{\bubblestack_0(T)} \simeq g_{mn}^* s_n^* r^* \Lcot_{\bubblestack_0(T)} \to g_{m,n}^* \Lcot_{\bubblestack_n(T)} \to \phi(n)
\]
Let us denote by $\psi(n)$ the cofiber
\[
s_n^* r^* \Lcot_{\bubblestack_0(T)} \to \Lcot_{\bubblestack_n(T)} \to \psi(n)
\]
so that $\phi(n) \simeq g_{mn}^* \psi(n)$.
This sequence is equivalent to the colimit (in $\IPerf(\bubblestack_n(T))$) of a cofiber sequence of diagrams $\{ p \in \N | p \geq n \}\op \to \Perf(\bubblestack_n(T))$
\[
\lambda^{\Pro}_{\bubblestack_n(T)}(\bar \bubblestack_0(T)) \to \lambda^{\Pro}_{\bubblestack_n(T)}(\bar \bubblestack_n(T) ) \to \bar \psi(n)
\]
It suffices to prove that the diagram $\bar \psi(n) \colon \{ p \in \N | p \geq n \}\op \to \Perf(\bubblestack_n(T))$ is (essentially) constant.
Let $p \in \N$, $p \geq n$. The perfect complex $\bar \psi(n)(p)$ fits in the exact sequence
\[
t_{np}^* \varepsilon_{np}^* \Lcot_{\bubblestack_{0,p}(T)} \to \pi_{n,p}^* \Lcot_{\bubblestack_{n,p}(T)} \to \bar \psi(n)(p)
\]
where $t_{np} \colon \bubblestack_n(T) \to \bubblestack_{n,p}(T)$ is the canonical projection and $\varepsilon_{np} \colon \bubblestack_{n,p}(T) \to \bubblestack_{0,p}(T)$ is induced by the augmentation $\Oo_{S_{n,p}} \to \Oo_{S_{0,p}}$.
It follows that $\bar \psi(n)(p)$ is equivalent to
\[
t_{np}^* \Lcot_{\bubblestack_{n,p}(T)/\bubblestack_{0,p}(T)}
\]
Moreover, for any $q \geq p \geq n$, the induced map $\bar \psi(n)(p) \to \bar \psi(n)(q)$ is obtained (through $t_{nq}^*$) from the cofiber, in $\Perf(\bubblestack_{n,q}(T))$
\[
\mymatrix@R=2mm{
\alpha_{npq}^* \varepsilon_{np}^* \Lcot_{\bubblestack_{0,p}(T)} \ar[r] \ar@{}[rdddd]|*{(\sigma)} & \alpha_{npq}^* \Lcot_{\bubblestack_{n,p}(T)} \ar[dddd] \ar[r] & \alpha_{npq}^* \Lcot_{\bubblestack_{n,p}(T)/\bubblestack_{0,p}(T)} \ar[dddd] \\
\varepsilon_{nq}^* \alpha_{0pq}^* \Lcot_{\bubblestack_{0,p}(T)} \ar@{=}[u] \ar[ddd] \\ \\ \\
\varepsilon_{nq}^* \Lcot_{\bubblestack_{0,q}(T)} \ar[r] & \Lcot_{\bubblestack_{n,q}(T)} \ar[r] & \Lcot_{\bubblestack_{n,q}(T)/\bubblestack_{0,q}(T)}
}
\]
where $\alpha_{npq}$ is the map $\bubblestack_{n,q}(T) \to \bubblestack_{n,p}(T)$.
Let us denote by $(\sigma)$ the square on the left hand side above.
Let us fix a few more notations
\[
\mymatrix@R=7mm@C=5mm{
&
\bubblestack_{n,p}(T) \times S_{0,p} \ar[dd]_{\varphi_{np}} \ar[dl]_(0.6){a_{0p}} & &
\bubblestack_{n,q}(T) \times S_{0,p} \ar[dd]_{\psi_{npq}} \ar[ll] \ar[rr] & &
\bubblestack_{n,q}(T) \times S_{0,q} \ar[dd]^{\varphi_{nq}} \\
S_{0,p} \ar[dd]_{\xi_{np}} \\ &
\bubblestack_{n,p}(T) \times S_{n,p} \ar[dl]^{a_{np}} \ar@{-}[d] & &
\bubblestack_{n,q}(T) \times S_{n,p} \ar[ll] \ar[rr]^{b_{npq}} \ar@{-}[d] \ar[ld] & &
\bubblestack_{n,q}(T) \times S_{n,q} \ar[dl]^{a_{nq}} \ar[dd]^{\varpi_{nq}} \ar[dr]^(0.7){\ev_{nq}} \\
S_{n,p} & \ar[d]^(0.35){\varpi_{np}} & S_{n,p} \ar@{-}[ll]_(0.4)= & \ar[d] & S_{n,q} \ar[ll]_(0.3){\beta_{npq}} & & T\\ &
\bubblestack_{n,p}(T) & & \bubblestack_{n,q}(T) \ar[ll]_{\alpha_{npq}} \ar@{-}[rr]_{=} & & \bubblestack_{n,q}(T)
}
\]
The diagram $(\sigma)$ is then dual to the diagram
\[
\mymatrix{
\alpha_{npq}^* \varepsilon_{np}^* {\varpi_{0p}}_* \ev_{0p}^* \T_T &
\alpha_{npq}^* {\varpi_{np}}_*  \ev_{np}^* \T_T \ar[l] \\
\varepsilon_{nq}^* {\varpi_{0q}}_* \ev_{0q}^* \T_T \ar[u] &
{\varpi_{nq}}_* \ev_{nq}^* \T_T \ar[l] \ar[u]
}
\]
Moreover, the functor $\varpi_{np}$ (for any $n$ and $p$) satisfies the base change formula. This square is thus equivalent to the image by ${\varpi_{nq}}_*$ of the square
\[
\mymatrix{
{\psi_{npq}}_* {b_{npq}}_* b_{npq}^* \psi_{npq}^* \ev_{nq}^* \T_T &
{b_{npq}}_* b_{npq}^* \ev_{nq}^* \T_T \ar[l] \\
{\varphi_{nq}}_* \varphi_{nq}^* \ev_{nq}^* \T_T \ar[u] &
\ev_{nq}^* \T_T \ar[l] \ar[u]
}
\]
Using now the projection and base change formulae along the morphisms $\varphi_{nq}$, $b_{npq}$ and $\psi_{npq}$, we see that this last square is again equivalent to
\[
\mymatrix{
(a_{nq}^* {\beta_{npq}}_* {\xi_{np}}_* \Oo_{S_{0,p}}) \otimes (\ev_{nq}^* \T_T) &
(a_{nq}^* {\beta_{npq}}_* \Oo_{S_{n,p}}) \otimes (\ev_{nq}^* \T_T) \ar[l]
\\
(a_{nq}^* {\xi_{nq}}_* \Oo_{S_{0,q}}) \otimes (\ev_{nq}^* \T_T) \ar[u] &
(a_{nq}^* \Oo_{S_{n,q}}) \otimes (\ev_{nq}^* \T_T) \ar[l] \ar[u]
}
\]
We therefore focus on the diagram
\[
\mymatrix{
\Oo_{S_{n,q}} \ar[r] \ar[d] & {\xi_{nq}}_* \Oo_{S_{0,q}} \ar[d] \\
{\beta_{npq}}_* \Oo_{S_{n,p}} \ar[r] & {\beta_{npq}}_* {\xi_{np}}_* \Oo_{S_{0,p}}
}
\]
By definition, the fibres of the horizontal maps are both equivalent to 
$A_n[-d]$
and the map induced by the diagram above is an equivalence.
We have proven that for any $q \geq p \geq n$ the induced map $\bar \psi(n)(p) \to \bar \psi(n)(q)$ is an equivalence.
It implies that $\Lcot_{\bubblestack(T)}$ is a Tate module.
\end{proof}

\subsection{A symplectic structure (shifted by \texorpdfstring{$d$}{d})}
In this subsection, we will prove the following
\begin{thm}\label{B-symplectic}
Assume $T$ is $q$-shifted symplectic.
The ind-pro-stack $\bubblestack^d(T)$ admits a symplectic Tate structure shifted by $q-d$.
Moreover, for any $m \in \N$ we have an exact sequence
\[
s_m^* r^* \Lcot_{\bubblestack^d(T)_0} \to s_m^* \Lcot_{\bubblestack^d(T)} \to s_m^* r^* \T_{\bubblestack^d(T)_0}[q-d]
\]
\end{thm}

\begin{proof}
Let us start with the following remark: the residue map $r_n \colon A_n \to k = A_1$ defined in \autoref{dual-over-k} defines a map $\Oo_{\formalsphere^d} \to k[-d]$.
From \autoref{ipdst-form}, we have a $(q-d)$-shifted closed $2$-form on $\bubblestack^d(T)$. We have a morphism from \autoref{prop-formsareforms}
\[
\Oo_{\bubblestack^d(T)}[q-d] \to \Lcot_{\bubblestack^d(T)} \otimes \Lcot_{\bubblestack^d(T)}
\]
in $\PIPerf(\bubblestack^d(T))$.
Let $m \in \N$. We get a map
\[
\Oo_{\bubblestack^d(T)_m}[q-d] \to s_m^* \Lcot_{\bubblestack^d(T)} \otimes s_m^* \Lcot_{\bubblestack^d(T)}
\]
and then
\[
s_m^* \T_{\bubblestack^d(T)} \otimes s_m^* \T_{\bubblestack^d(T)} \to \Oo_{\bubblestack^d(T)_m}[q-d]
\]
in $\IPPerf(\bubblestack^d(T)_m)$. We consider the composite map
\[
\theta \colon s_m^* \T_{\bubblestack^d(T)/\bubblestack^d(T)_0} \otimes s_m^* \T_{\bubblestack^d(T)/\bubblestack^d(T)_0} \to
s_m^* \T_{\bubblestack^d(T)} \otimes s_m^* \T_{\bubblestack^d(T)} \to \Oo_{\bubblestack^d(T)_m}[q-d]
\]
Using the \autoref{describe-form} and the proof of \autoref{prop-formal-tate} we see that $\theta$ is induced by the morphisms (varying $n$ and $p$)
\[
\mymatrix{
{\varpi_{np}}_* \left( E \otimes E \otimes \ev_{np}^* \left( \T_T \otimes \T_T \right) \right) \ar[r]^-A & {\varpi_{np}}_* \left( E \otimes E [q] \right) \ar[r]^-B & {\varpi_{np}}_* \left( \Oo_{\bubblestack^d(T)_{np} \times S_{n,p}} [q] \right)
}
\]
where $E = a_{np}^* {\xi_{np}}_* {h_{np}}_* \gamma_n^! \Oo_{\A^d}$ and the map $A$ is induced by the symplectic form on $T$. The map $B$ is induced by the multiplication in $\Oo_{S_{n,p}}$.
This sheaf of functions is a trivial square zero extension of augmentation ideal ${\xi_{np}}_* {h_{np}}_* \gamma_n^! \Oo_{\A^d}$ and $B$ therefore vanishes.
It follows that the morphism
\[
s_m^* \T_{\bubblestack^d(T)} \otimes s_m^* \T_{\bubblestack^d(T)/\bubblestack^d(T)_0} \to
s_m^* \T_{\bubblestack^d(T)} \otimes s_m^* \T_{\bubblestack^d(T)} \to \Oo_{\bubblestack^d(T)_m}[q-d]
\]
factors through $s_m^* \T_{\bubblestack^d(T)_0} \otimes s_m^* \T_{\bubblestack^d(T)/\bubblestack^d(T)_0}$.
Now using \autoref{prop-formal-tate} we get a map of exact sequences in the category of Tate modules over $\bubblestack^d(T)_m$
\[
\mymatrix{
s_m^* \T_{\bubblestack^d(T)/\bubblestack^d(T)_0} \ar[r] \ar[d]_{\tau_m} &
s_m^* \T_{\bubblestack^d(T)} \ar[r] \ar[d] & s_m^* r^* \T_{\bubblestack^d(T)_0} \ar[d] \\
s_m^* r^* \Lcot_{\bubblestack^d(T)_0}[d-q] \ar[r] & s_m^* \Lcot_{\bubblestack^d(T)}[d-q] \ar[r] & s_m^* \Lcot_{\bubblestack^d(T)/\bubblestack^d(T)_0}[d-q]
}
\]
where the maps on the sides are dual one to another.
It therefore suffices to see that the map $\tau_m \colon s_m^* \T_{\bubblestack^d(T)/\bubblestack^d(T)_0} \to s_m^* r^* \Lcot_{\bubblestack^d(T)_0}[d-q]$ is an equivalence.
We now observe that $\tau_m$ is a colimit indexed by $p \geq m$ of maps
\[
g_{pm}^* t_{pp}^* \left( \varepsilon_{pp}^* \Lcot_{\bubblestack^d(T)_{0p}} \to \T_{\bubblestack^d(T)_{pp}/\bubblestack^d(T)_{0p}} \right)
\]
Let us fix $p \geq m$ and $G = a_{pp}^* {\xi_{pp}}_* \Oo_{S_{0p}}$. The map $
F_p \colon \T_{\bubblestack^d(T)_{pp}/\bubblestack^d(T)_{0p}} \to \varepsilon_{pp}^* \Lcot_{\bubblestack^d(T)_{0p}}$ at hand is induced by the pairing
\[
\T_{\bubblestack^d(T)_{pp}/\bubblestack^d(T)_{0p}} \otimes \varepsilon_{pp}^* \T_{\bubblestack^d(T)_{0p}} \simeq
\mymatrix{
{\varpi_{pp}}_* \left(E \otimes \ev_{pp}^* \T_T \right) \otimes {\varpi_{pp}}_* \left( G \otimes \ev_{pp}^* \T_T \right) \ar[d] \\
{\varpi_{pp}}_* \left( E \otimes \ev_{pp}^* \T_T \otimes G \otimes \ev_{pp}^* \T_T \right) \ar[d] \\
{\varpi_{pp}}_* \left( E \otimes G \right)[q] \ar[d] \\
{\varpi_{pp}}_* \left( \Oo_{\bubblestack^d(T)_{pp} \times S_{pp}} \right) [q] \ar[d] \\
\Oo_{\bubblestack^d(T)_{pp}} [q-d]
}
\]
We can now conclude using \autoref{dual-over-k}.
\end{proof}


\phantomsection

\end{document}